\documentclass[11pt,letterpaper]{article}
\usepackage{amsthm,amsmath,amssymb}
\usepackage[style=alphabetic, maxbibnames=15, maxcitenames=15, natbib=true,
  maxalphanames=10, backend=biber, sorting=nty, backref=true]{biblatex}
\DefineBibliographyStrings{english}{backrefpage = {page},backrefpages = {pages},}

\usepackage{graphicx,xcolor}
\usepackage{nicefrac}
\usepackage[normalem]{ulem}
\usepackage{setspace}
\usepackage{accents}
\usepackage[utf8]{inputenc}
\usepackage{longtable}
\usepackage{multirow}
\usepackage{subcaption}
\usepackage{url}
\usepackage{algorithm,algorithmic}
\usepackage[utf8]{inputenc} 
\usepackage[T1]{fontenc}    
\usepackage{hyperref}       
\usepackage{url}            
\usepackage{booktabs}       
\usepackage{pifont}
\usepackage{amsfonts}       
\usepackage{nicefrac}       
\usepackage{microtype}  
\usepackage{wrapfig}

\DeclareMathOperator*{\argmin}{\arg\!\min}
\DeclareMathOperator*{\argmax}{\arg\!\max}

\usepackage[margin=1in]{geometry}
\def\grad{\nabla}

\def\ba{\mathbf{a}}
\def\bb{\mathbf{b}}

\def\bp{\mathbf{p}}

\def\br{\mathbf{r}}

\def\bv{\mathbf{v}}

\def\bx{\mathbf{x}}  
\def\by{\mathbf{y}}
\def\bz{\mathbf{z}}

\def\bC{\mathbf{C}}
\def\bD{\mathbf{D}}
\def\bE{\mathbf{E}}
\def\bF{\mathbf{F}}

\def\bI{\mathbf{I}}

\def\bL{\mathbf{L}}

\def\bT{\mathbf{T}}

\def\cA{\mathcal{A}}

\def\cE{\mathcal{E}}
\def\cF{\mathcal{F}}
\def\cG{\mathcal{G}}

\def\cK{\mathcal{K}}
\def\cL{\mathcal{L}}
\def\cM{\mathcal{M}}

\def\cO{\mathcal{O}}
\def\cP{\mathcal{P}}

\def\cX{\mathcal{X}}
\def\cY{\mathcal{Y}}
\def\cZ{\mathcal{Z}}

\def\mE{\mathbb{E}}

\def\smskip{\smallskip}

\def\texitem#1{\par\smskip\noindent\hangindent 25pt
               \hbox to 25pt {\hss #1 ~}\ignorespaces}

\def\abs#1{\left|#1\right|}
\def\norm#1{\left\|#1\right\|}

\newcommand{\BEAS}{\begin{eqnarray*}}
\newcommand{\EEAS}{\end{eqnarray*}}
\newcommand{\BEA}{\begin{eqnarray}}
\newcommand{\EEA}{\end{eqnarray}}
\newcommand{\BEQ}{\begin{eqnarray}}
\newcommand{\EEQ}{\end{eqnarray}}
\newcommand{\BIT}{\begin{itemize}}
\newcommand{\EIT}{\end{itemize}}
\newcommand{\BNUM}{\begin{enumerate}}
\newcommand{\ENUM}{\end{enumerate}}

\newcommand{\BA}{\begin{array}}
\newcommand{\EA}{\end{array}}

\newcommand{\ones}{\mathbf 1}

\newcommand{\reals}{\mathbb{R}}
\newcommand{\integers}{\mathbb{Z}}

\newcommand{\diag}{\mathop{\bf diag}}

\newcommand{\dom}{\mathop{\bf dom}}

\newif\ifpagenumbering
\pagenumberingtrue

\pagenumberingfalse

\newsavebox{\theorembox}
\newsavebox{\lemmabox}
\newsavebox{\defnbox}
\newsavebox{\corollarybox}
\newsavebox{\remarkbox}
\newsavebox{\assbox}
\savebox{\theorembox}{\noindent\bf Theorem}
\savebox{\lemmabox}{\noindent\bf Lemma}
\savebox{\defnbox}{\noindent\bf Definition}
\savebox{\corollarybox}{\noindent\bf Corollary}
\savebox{\remarkbox}{\noindent\bf Remark}
\newtheorem{remark}{\usebox{\remarkbox}}[section]
\newtheorem{defn}{\usebox{\defnbox}}

\def\fprod#1{\left\langle#1\right\rangle}

\def\ind#1{\mathbb{I}_{#1}}

\def\id{\mathbf{I}}

\def\fM{\mathfrak{M}}
\def\bdelta{\boldsymbol{\delta}}
\def\sa#1{\textcolor{black}{#1}}

\def\rev#1{\textcolor{black}{#1}}
\def\eh#1{\textcolor{black}{#1}}
\def\btx{\tilde{\bx}}

\def\eyy#1{{#1}}
\def\na#1{{#1}}

\def\eyh#1{{#1}}

\usepackage{todonotes}

\makeatletter
\def\munderbar#1{\underline{\sbox\tw@{$#1$}\dp\tw@\z@\box\tw@}}
\makeatother

\newtheorem{theorem}{Theorem}
\newtheorem{result}{Main Result}

\newtheorem{lemma}{Lemma}

\newtheorem{assumption}{Assumption}

\usepackage{hyperref}
\hypersetup{colorlinks,citecolor=blue,linktocpage,breaklinks=true}

\begin{document}

%

%

\title{Randomized Primal-Dual Methods with 
Adaptive Step Sizes}

\author{ Erfan Yazdandoost Hamedani\thanks{Department of Systems and Industrial Engineering, The University of Arizona, Tucson, AZ, USA \{erfany@arizona.edu, afrooz@arizona.edu\}}
\quad Afrooz Jalilzadeh$^*$
\quad Necdet Serhat Aybat \thanks{Department of Industrial and Manufacturing Engineering, The Pennsylvania State University, University Park, PA, USA \{nsa10@psu.edu\}. Research of N. S. Aybat was partially supported by ONR grant N00014-21-1-2271.}
\vspace{5mm} }
\date{}
\maketitle

\begin{abstract}
In this paper we propose a class of randomized primal-dual methods incorporating line search to contend with large-scale saddle point~(SP) problems defined by a convex-concave function $\cL(\bx,y)\triangleq \na{\sum_{i=1}^M}f_i(x_i)+\Phi(\bx,y)-h(y)$. We analyze the convergence rate of the proposed method under mere convexity and strong convexity assumptions of $\cL$ in $\bx$-variable. In particular, assuming $\grad_y\Phi(\cdot,\cdot)$ is Lipschitz and $\grad_\bx\Phi(\cdot,y)$ is coordinate-wise Lipschitz for any fixed $y$, the ergodic sequence generated by the algorithm achieves the $\cO(M/k)$ convergence rate in the expected primal-dual gap.
Furthermore, assuming that $\cL(\cdot,y)$ is strongly convex for any $y$, and that $\Phi(\bx,\cdot)$ is affine for any $\bx$, the scheme enjoys a faster rate of $\cO(M/k^2)$ in terms of primal solution suboptimality. We implemented the proposed algorithmic framework to solve kernel matrix learning problem, and tested it against other state-of-the-art first-order methods.
\end{abstract}
\section{Introduction}
\label{sec:intro}
\na{Let $(\cX_i,\norm{\cdot}_{\cX_i})$ for $i\in\cM\triangleq \{1,2,\hdots,M\}$ and $(\cY,\norm{\cdot}_{\cY})$ be finite dimensional, normed vector spaces such that $\cX_i=\reals^{m_i}$ for $i\in\cM$.} 
Let $\bx=[x_i]_{i\in\cM}\in\Pi_{i\in\cM}\cX_i\triangleq \cX=\reals^m$ where {$m\triangleq \sum_{i\in\cM}m_i$}. In this paper, we study the following 
saddle point~(SP) problem:
\begin{equation}\label{eq:original-problem}
\begin{aligned}
(P):\quad &\min_{\bx\in\cX}\max_{y\in\cY}~ \cL(\bx,y),\\
&{\cL(\bx,y)\triangleq} \sum_{i\in\cM}f_i(x_i)+\Phi(\bx,y)-h(y),
\end{aligned}
\end{equation}
where {$h:\cY\rightarrow \reals\cup\{+\infty\}$ and
$f_i:\cX_i\rightarrow \reals\cup\{+\infty\}$ for all $i\in\cM$ are (possibly nonsmooth) closed \rev{$\mu_i$-convex functions with respect to $\norm{\cdot}_{\cX_i}$ for some $\mu_i\geq 0$,}
 and the coupling function} $\Phi:\cX\times\cY\rightarrow \reals$ is 
convex in $\bx$ and concave in $y$ \rev{and, it} satisfies certain differentiability assumptions -- see Assumption~\ref{assum}.

{Our study is motivated by large-scale problems with a \emph{coordinate-friendly} structure~\cite{peng2016coordinate}, i.e., {for any $i\in\cM$,} the 
{amount of work} 
to compute the partial-gradient $\grad_{x_i}\Phi(\bx,y)$ 
{is} 
${m_i/m\approx} 1/M$ fraction of the {work required for the} full-gradient $\grad_\bx\Phi(\bx,y)$ computation.} {Our objective is to design an efficient first-order \emph{randomized} block-coordinate primal-dual method to compute a saddle point of the structured convex-concave function $\cL$ in~\eqref{eq:original-problem}, and to investigate its convergence properties under \rev{\emph{mere}} and \rev{\emph{strong convexity}} settings.} \eyh{Typically, the first-order methods rely on the knowledge of \rev{global} Lipschitz constants 
to select an appropriate step-size 
{with} a convergence guarantee. In practical settings, such {constants may not be readily available or it can be difficult to compute them}; hence, one may need to 
{consider} line-search methods. 
\rev{Since} \emph{exact} line-search methods are often difficult to implement, \rev{one practical avenue is to adopt 
backtracking 
to estimate} the \emph{local} Lipschitz {constants}, {which usually leads to larger steps}. Hence, in this paper, we propose a randomized
block-coordinate primal-dual algorithm with \rev{\emph{backtracking}} to efficiently solve {the} SP problem {in}~\eqref{eq:original-problem}.} 

{Before we 
discuss
important applications, it should be emphasized that \eqref{eq:original-problem} covers
regularized convex optimization problems with nonlinear constraints as a special case\footnote{{We can show that Assumption~\ref{assum} is satisfied for this example through arguing that the dual iterate sequence of the proposed method is almost surely bounded. Thus, one does not need the boundedness of dual domain to argue for the existence of global constants $\{L_{x_i x_i}\}_{i\in\cM}$, instead bounded dual sequence is sufficient for our proof arguments.}}, i.e.,}
\begin{align}\label{eq:constrained-problem}
\min_{\bx}
\rho(\bx)\triangleq \varphi(\bx)+\sum_{i\in\cM}f_i(x_i) \ \ \hbox{s.t.}\ \ g(\bx)\in -\cK,
\end{align}
where {$\cK\subseteq\cY^*$} is a closed convex cone \na{in the dual space $\cY^*$}; $f_i:\cX_i\rightarrow\reals\cup\{+\infty\}$ is a convex (possibly nonsmooth) regularizer function for $i\in\cM$; $\varphi:\cX\rightarrow\reals$ is a smooth convex function having a coordinate-wise Lipschitz continuous gradient; 
and $g:\cX\rightarrow\cY^*$ is a smooth $\cK$-convex, \na{coordinate-wise Lipschitz function {with} a coordinate-wise Lipschitz Jacobian.} 
This problem can be written as a special case of \eqref{eq:original-problem} by setting $\Phi(\bx,y)=\varphi(\bx)+\fprod{g(\bx),y}$ and $h(y)=\ind{\cK^*}(y)$, where {$\cK^*\subseteq\cY$} denotes the dual cone of $\cK$ {and $\ind{\cK^*}(\cdot)$ is the indicator function of $\cK^*$.}
{An advantage of \sa{primal-dual methods to solve such formulation, over primal methods that preserve feasibility at every iteration,} lies in utilizing the corresponding dual variables in order to boost the primal convergence 
{through appropriately} controlling the constraint violations.}

{\bf Application.} Many interesting problems arising in 
machine learning~(ML), 
{signal 
and image processing,} finance, etc., can be formulated as a special case of \eqref{eq:original-problem}. 
Some instances of such problems include: i) \emph{distributionally robust optimization} \cite{namkoong2016stochastic}; ii) \emph{kernel matrix learning} \cite{lanckriet2004learning,gonen2011multiple}; iii) \emph{distance metric learning} \cite{xing2003distance}; (iv) training \emph{ellipsoidal machines}~\cite{shivaswamy2007ellipsoidal}; (v) \emph{two-player zero-sum game} with nonlinear payoff~\citep{boyd2004convex,chen2017accelerated}.
In the following, we will briefly discuss some of these problem instances and their formulations as a special case of \eqref{eq:original-problem}. 

{\bf Distributionally robust optimization (DRO):} Let $(\Omega,\cF,\mathbb{P})$ be a probability space where $\Omega=\{\zeta_1,\hdots, \zeta_m\}$, $\ell:X\times \Omega\to \reals$ is a convex loss function {over a convex bounded set $X\subset \cX$}, and we define $\ell_i(u)\triangleq \ell(u;\zeta_i)$. {The aim of DRO is to optimize the worst} case performance under uncertainty and to compute solutions with some 
confidence level~\cite{namkoong2016stochastic}. {This class of problems can be formulated as}
\begin{align}\label{eq:DRO}
    \min_{u\in X}\max_{\bp\in\cP}~\mE_{\zeta\sim \mathbb{P}}[\ell(u;\zeta)]=\sum_{i=1}^m p_i\ell_i(u) ,
\end{align}
where $\cP$ represents {an uncertainty set over the probability distributions.}
For instance, $\cP=\{\by\in\Delta_m : V(\bp,\frac{1}{m}\ones_m)\leq \rho\}$ is an uncertainty set considered in 
\cite{namkoong2016stochastic}, where 
{$\Delta_m$} is an $m$-dimensional {probability} simplex, and $V(Q,P)$ denotes a divergence measure for 
probability measures $Q$ and $P$. {Assuming $V(\bp,\frac{1}{m}\ones_m)=\sum_{i=1}^mV_i(p_i,\frac{1}{m})$} and introducing variables $\lambda\in\reals_+$ and $\eta\in\reals$, we can 
{dualize} the divergence constraint in \eqref{eq:DRO} to obtain the following equivalent problem: 
\begin{align*}
    \min_{\substack{u\in X\\ \lambda\geq 0\\ \eta\in\reals}}\max_{\substack{\bp\in\reals^m_+}}\sum_{i=1}^m {p_i\ell_i(u) - \tfrac{\lambda}{m} \big(V_i(p_i,\tfrac{1}{m})- \tfrac{\rho}{m}\big)+\eta(p_i-\tfrac{1}{m})} .
\end{align*}
Let $y=[u^\top~\lambda~\ \eta]^\top$ and $\bx=\bp$, then multiplying the above problem by -1 we can reformulate the problem as \eqref{eq:original-problem} by defining 
$f_i(x_i)=\ind{\reals_+}(x_i)$ and $h(y)=\ind{X\times\reals_+\times\reals}(y)$ \rev{as indicator functions}.

{\bf Learning a kernel matrix:} Suppose we are given a {training set} 
consisting of feature vectors $\{\ba_i\}_{i=1}^m\subset\reals^n$, and the corresponding labels $\{b_i\}_{i=1}^m\subset\{-1,+1\}$. Consider {$q\in\integers_+$} different embedding of the data and let $K_i\in\mathbb{S}^m_{+}$ be the corresponding kernel matrix {for $i=1,\ldots,q$}. The objective is to learn a kernel matrix $K$ belonging {to a class of kernel matrices $\cK\subset \mathbb{S}^m_{+}$ such that it minimizes} the training error of an {$\ell_2$-norm soft-margin nonlinear} SVM {over $K\in\cK$} -- see~\cite{lanckriet2004learning} for more details. {In this setting, it is assumed that the class $\cK$ is described as a convex set} generated by \eyy{$\{K_i\}_{i=1}^q$,} {e.g.,} 
\begin{equation}
\label{eq:Kernel_Class}
\cK\triangleq \big\{\sum_{i=1}^q y_i K_i:\ y_i\geq 0,\ i=1,\ldots, q\big\}{\subset\mathbb{S}^m_{+}}.
\end{equation}
{Then, 
learning over the class $\cK$ in~\eqref{eq:Kernel_Class} is formulated as} 
\begin{align}
\label{eq:kernel_learn_simple}
\min_{\substack{y\in\reals^q_+:\\ \fprod{\br,y}=c,}} \max_{\substack{\bx:\ 0\leq \bx\leq C\ones_m, \\ \ \fprod{\bb,\bx}=0}} 2\bx^\top\ones_m - \sum_{i=1}^q {y_i}\bx^\top H(K_i)\bx-\lambda\norm{\bx}_2^2,
\end{align}
where $c,C{>}0$ and $\lambda\geq 0$  are model parameters, $y=[y_i]_{i=1}^q$, $\br=[\text{trace}(K_i)]_{i=1}^q$, $\bb=[b_i]_{i=1}^m$ and $H(K_i)\triangleq \diag(\bb)K_i\diag(\bb)$. Multiplying the objective function by -1, this problem can be formulated as a special case of \eqref{eq:original-problem}.

{Problems in the aforementioned applications are typically \rev{large-scale, and} standard primal-dual methods do not scale well with the problem dimension and their iterations are memory expensive; therefore, {in terms of the efficiency of work required per-iteration,} the advantages of randomized block-coordinate schemes will be evident as problem dimension increases.}  
\begin{table*}[htb]
\small
\centering
{\small
\caption{Comparison of different methods in Merely Convex (MC) and Strongly Convex (SC) settings. In convergence rates, $k$ denotes the iteration counter. 
The work-per-iteration for \rev{\cite{juditsky2011first,malitsky2018proximal,hamedani2021primal} is $\cO(M)$} while it is $\cO(1)$ for the others.}
\label{sc_table}}
\begin{tabular}{|c|c|c|c|c|c|} \hline
	  &  \multicolumn{3}{c|}{\textbf{Properties}} &\multicolumn{2}{c|}{\textbf{ Iteration complexity}} \\ \hline
 \textbf{Paper}& \textbf{Non-bilinear $\Phi$}  & \textbf{Random blocks}& \textbf{Line search} & \textbf{C-C}& \textbf{SC-C}
\\ \hline\hline
\cite{chambolle2017stochastic}& \ding{55}&\checkmark&\ding{55}&$\mathcal O(M/k)$& $\mathcal O(M/k^2)$ \\ \hline
\cite{xu2021first}&\ding{55}&\checkmark&\checkmark&$\mathcal O(M/(M+k))$& $-$\\ \hline
\cite{dang2014randomized}&\ding{55}&\checkmark&\ding{55}&$\mathcal O(M/k)$& $\mathcal O(M/k^2)$ \\ \hline
\cite{tran2020new}&\ding{55}&\checkmark&\ding{55}&$\mathcal O(M/k)$& $\mathcal O(M^2/k^2)$\\ \hline
\cite{alacaoglu2022convergence}&\ding{55}&\checkmark&\ding{55}&$\mathcal O(M/k)$& $-$\\ \hline
\cite{alacaoglu2022complexity} &\ding{55}&\checkmark&\ding{55}&$\mathcal O(M/k)$& $-$ \\
\hline
\cite{juditsky2011first}&\checkmark&\ding{55}&\ding{55}&$\mathcal O(1/k)$ & $\mathcal O(1/k^2)$\\ \hline
\cite{malitsky2018proximal}&\checkmark&\ding{55}&\checkmark&$\mathcal O(1/k)$ & $-$
\\ \hline
\cite{hamedani2021primal}&\checkmark&\ding{55}&\checkmark&$\mathcal O(1/k)$ & $\mathcal O(1/k^2)$\\ \hline\hline
{\bf This paper}&\checkmark&\checkmark&\checkmark&$\mathcal O(M/k)$ & $\mathcal O(M/k^2)$\\ \hline
\end{tabular}
\end{table*}
\subsection{Related Work}
\eyh{Saddle point problems have received a 
{significant} attention recently due to their vast applicability and {modeling} flexibility. Here, we briefly review some recent work that is closely related to ours {--see also Table \ref{sc_table} for a detailed comparison}.}

{\bf Bilinear SP:} There have been several work proposing efficient algorithms to solve convex-concave SP problems with {a \emph{bilinear} 
coupling} function, i.e., $\Phi(x,y)=\fprod{Ax,y}$ for some linear map $A:\cX\to\cY^*$, 
{e.g.,}~\cite{chambolle2011first,chen2014optimal,chambolle2016ergodic,he2016accelerated,li2021new,alacaoglu2022convergence,alacaoglu2022complexity}. \cite{chambolle2016ergodic} considered an SP problem with a composite 
structure and a convergence rate of $\cO(1/K)$ for convex-concave and $\cO(1/K^2)$ for strongly convex-{affine} 
{settings} are shown. Later \cite{malitsky2018first} proposed a primal-dual method with linesearch with the same rate results as {in}~\cite{chambolle2016ergodic}. 

{\bf Non-bilinear SP:} There has been a vast body of research, e.g., \cite{juditsky2011first,he2015mirror,kolossoski2017accelerated,malitsky2018proximal,malitsky2020forward,hamedani2021primal,zhang2021robust,zhang2022sapd+}, studying SP problems with non-bilinear 
{coupling} functions. Indeed, non-bilinear SP problems can be viewed as a special case of Variational Inequality~(VI) {problems}. In 
{an important work by}~\cite{nemirovski2004prox}, a prox-type extra-gradient based method (known as \texttt{Mirror-prox}) is proposed. Assuming that the monotone operator is $L$-Lipschitz continuous and the constraint set is compact, 
it is shown that the ergodic iterate sequence converges with $\cO(L/K)$ rate --also see 
\cite{he2015mirror} for extension of \texttt{Mirror-prox} to SP problems with a composite structure. 
Later \texttt{Mirror-prox} has been extended to exploit SP problems for strongly convex-
{affine} setting; in particular, a \emph{multi-stage} method that repeatedly calls \texttt{Mirror-prox} is proposed by \cite{juditsky2011first}, and $\cO(1/K^2)$ rate is shown 
{for the strongly convex-affine setting when $\cY$ is a \emph{compact} set.} 
Later, \cite{malitsky2018proximal} also considered a monotone VI problem {involving} a non-smooth function with {an} easy-to-compute proximal map. The author proposed a proximal
extrapolated gradient method, {\texttt{PEGM}}, with {an} ergodic convergence rate of $\cO(1/K)$. The proposed method enjoys a backtracking scheme to estimate the local Lipschitz
{constants} of the monotone map --for a backtracking line-search method {tailored to} SP problems, see 
\cite{hamedani2021primal} and for a more general setting of monotone inclusion problems, see 
\cite{malitsky2020forward}.

{\bf Block coordinate:} {As we indicated earlier, none of these methods mentioned above exploits the block-coordinate structure of \eqref{eq:original-problem}. However,} in a large-scale setting, the computation of full-gradient and/or prox operator might be prohibitively expensive; hence, presenting a strong motivation for using the partial-gradient and/or separable structure of the problem at each iteration of the algorithm. Therefore, the computation may be broken into smaller pieces; thereby, inducing tractability {per iteration, at the cost of possibly slower convergence in terms of 
overall iteration complexity.}
There has been a 
vast body of work on randomized block-coordinate descent schemes for primal optimization problems by \cite{nesterov2012efficiency,luo1992convergence,xu2013block,richtarik2014iteration,jalilzadeh2018optimal}; but, there are far fewer studies on randomization of block coordinates for SP algorithms. 
For \sa{some papers motivated 
by the} regularized empirical risk minimization (ERM) of linear predictors arising in 
{ML}, see \cite{zhu2015adaptive,yu2015doubly,zhang2017stochastic,chambolle2017stochastic}.

Furthermore, {there are some related recent work} by \cite{zhong2014fast,gao2016randomized}
on block-coordinate ADMM-type algorithms to solve convex optimization problem with linear constraints. 
Assuming coordinate-wise Lipschitz differentiability of $g$ and $q$, 
$\cO(1/(1+\gamma k))$ convergence rate is shown under mere convexity assumption, where {$\gamma=\frac{M'}{M}=\frac{N'}{N}$ and $M'$ ($N'$)} denotes the number of $x_i$ ($y_i$) coordinates updated at each iteration.

Majority of the previous work on block-coordinate algorithms for SP problems require a bilinear coupling term in the problem formulation~\citep{dang2014randomized,fercoq2015coordinate,valkonen2016block}. {However, to the best of our knowledge, none of the existing methods can handle the framework discussed in this paper -- the closest one to ours is \cite{xu2021first} which can exploit the block-coordinate structure and does not assume the knowledge of Lipschitz constants via employing line-search; though, it is for constrained optimization problems, not for more general SP problems considered in this paper.} \sa{More precisely, \cite{xu2021first} 
considered} a convex minimization problem with functional constraints, $\min\left\{ f(\bx)+g(\bx)~:~A\bx=b,~G(\bx)\leq 0\right\}$, where $g$ and component functions of $G$ are Lipschitz differentiable convex functions, and $f$ is a proper closed convex function (possibly nonsmooth). 
When the function $f(\bx)$ has a separable structure, $\sum_{i\in\cM}f_i(x_i)$, a randomized block-coordinate linearized augmented Lagrangian method, \texttt{BLALM}, with a convergence rate of $\cO(1/(1+\frac{k}{M}))$ is proposed. 
Note \texttt{BLALM} cannot deal with \eqref{eq:original-problem} when $\Phi$ is not linear in $y$. 
\eyh{Finally, in a recent paper by~\cite{tran2020new}, a 
randomized block-coordinate primal-dual algorithm for solving convex composite optimization problem with linear constraints {is proposed}. It is shown that the 
algorithm achieves a last-iterate convergence of $\cO(M/k)$ and $\cO(M^2/k^2)$ for convex and strongly convex objective functions.}
\paragraph{\rev{Notation and Definitions.}} \rev{Throughout the paper $\norm{\cdot}$ denotes the Euclidean norm, i.e., $\norm{\cdot}_2$. Define $\bar{\mu}\triangleq\max_{i\in\cM}\mu_i$ and $\underbar{$\mu$}\triangleq\min_{i\in\cM}\mu_i$. Let $F:\cX\to\reals\cup\{+\infty\}$ such that $F(\bx)\triangleq\sup_{y\in\cY}\cL(\bx,y)$ for $\bx\in\cX$.
Under strongly convex-concave setting, i.e., \rev{$\underbar{$\mu$}>0$}, we assume that $L_{yy}=0$ which implies that $\Phi(\bx,\cdot)$ is affine. In this scenario one can assume that $\Phi(\bx,y)=\fprod{g(\bx),y}$ for some continuously differentiable vector-valued function $g:\cX\to\cY^*$. Therefore, the problem in \eqref{eq:original-problem} can be equivalently represented as $\min_{\bx\in\cX}F(\bx)$, where 
\rev{\begin{align}
\label{eq:primal_problem}
F(\bx)=f(\bx)+h^*(g(\bx)),\quad\forall \bx\in\cX,
\end{align}}%
and $h^*$ denotes the conjugate function of $h$.}

\subsection{\rev{Our Contributions}}
In this paper we studied large-scale SP problems \na{with a general structure:} the coupling function is \emph{neither} bilinear \emph{nor} {separable}. To efficiently handle 
large-scale SP problems, we propose a randomized block-coordinate primal-dual algorithm with 
\sa{backtracking}. The proposed algorithm uses momentum acceleration and is {equipped with Bregman distance functions} that can generalize previous methods such as~\cite{chambolle2017stochastic}. These type of schemes are the method of choice for the SP problems with a \emph{coordinate-friendly} structure so that the computational tasks performed on each block coordinate at each iteration are significantly cheaper compared to full-gradient computations.

\rev{Let $\cG:\cZ\to\reals\cup\{+\infty\}$ denote the primal-dual gap function defined as follows:}
\begin{align}\label{s-gap}
    \rev{\cG(\bar{\bx}, \bar{y})}\triangleq \sup_{{(\bx,y)\in \cX\times\cY}}\{\cL(\bar{\bx},y)-\cL(\bx,\bar{y})\},
\end{align}
where $\cZ=\cX\times \cY$. Whenever a saddle point of \eqref{eq:original-problem} exists,
 under Lipschitz continuity of $\grad_y\Phi(\cdot,\cdot)$ and coordinate-wise Lipschitz continuity of $\grad_\bx\Phi(\cdot,y)$ for any fixed $y$, we prove that the iterate sequence converges to a saddle point $(\bx^*,y^*)$ in a.s. sense, and we also show convergence rate guarantee in terms of \rev{the expected gap $\bE[\cG(\bar{\bx}^k,\bar{y}^k)]$ when $\underbar{$\mu$}=0$, and in the solution error 
 \sa{in terms of} $\bE[\norm{\bx^k-\bx^*}^2]$ and $\bE[F(\bar\bx^K)-F(\bx^*)]$ when $\underbar{$\mu$}>0$, where $\{(\bar{\bx}^k, \bar{y}^k)\}_k$ is \rev{an} ergodic average sequence of the iterates.}
 \begin{result}\label{result1}
{\it 
Suppose global Lipschitz constants are available.
Under convex-concave setting, {for any $\epsilon>0$, $\bz_\epsilon=(\bx_\epsilon,y_\epsilon)$ such that $\bE[\cG(\bz_\epsilon)]\leq \epsilon$ 
can be 
computed} within $\cO(M/\epsilon)$ primal-dual 
{oracle calls}. Moreover, when 
\rev{$\Phi$ is strongly convex in $\bx$ and 
linear in $y$, an $\epsilon$-optimal primal solution $\bx_\epsilon$, i.e., $\bE[\norm{\bx_\epsilon-\bx^*}^2]\leq \epsilon$ and $\bE[F(\bx_\epsilon)-F(\bx^*)]\leq \epsilon$,}  can be obtained within $\cO(\sqrt{M}/\sqrt{\epsilon})$ primal-dual 
oracle calls. {Each call to primal and dual oracles require evaluating $\grad_{x_i}\Phi$ for some $i\in\cM$ and $\grad_y\Phi$, respectively. See Theorem~\ref{thm:main-body} for details.}}
\end{result}
 To the best of our knowledge, our proposed method is the only randomized block-coordinate primal-dual algorithm that can handle general SP problems as in \eqref{eq:original-problem}, and our rate results achieve the 
 lower complexity bounds~\citep{chen2014optimal} for our setting; hence, they are \rev{unimprovable, i.e., \emph{optimal}.}

{Another contribution that 
\sa{immensely increases} the algorithmic applicability is 
the novel 
backtracking linesearch scheme adopted 
within the proposed randomized block-coordinate primal-dual method. The step-size selection for each block is closely related to the coordinatewise Lipschitz constant of \rev{the partial gradient corresponding to that block}; however, in practice, these constants are largely unknown -- one needs to know a constant for each block and the number of \rev{blocks} could be very large. Moreover, even if they can be estimated correctly -- which is not the case in many settings,  \rev{these estimates lead to very conservative step-size selections due to \rev{their} being global constants}. 
Our technique not only alleviates the burden of estimating largely unknown constants; but, also make the convergence much faster in practice as it corresponds to using local Lipschitz constants which leads to larger step-sizes while retaining the theoretical convergence rate guarantees of primal-dual methods using the global constants.}


\begin{result}
{\it Suppose that the global Lipschitz constants are not available. 
The iterates generated by our method with backtracking line search
converges to a saddle point point almost surely with the same oracle complexity as in Main Result  \ref{result1} for both convex and strongly convex settings {up to $\cO(1)$ constants. See Theorem~\ref{thm:backtrack-body} for details.} 
}
\end{result}
\section{Preliminaries}
In this section, we state the main definitions and assumptions that we need for our convergence analysis.  
\begin{defn}
\label{def:bregman}
Let $f(\bx)\triangleq \sum_{i\in\cM}f_i(x_i)$ \na{and $\cM=\{1,\ldots,M\}$}, 
 and define $U_i\in\reals^{{m}\times {m_i}}$ for $i\in\cM$ such that $\bI_m=[U_1,\hdots,{U_M}]$, where $\bI_m$ denotes the $m\times m$ identity matrix. Let
$\varphi_{\cY}:\cY\rightarrow\reals$ be a differentiable \rev{function}
on an open set containing
$\dom h$. 
Suppose 
$\varphi_{\cY}$ 
is 1-strongly convex with respect to
$\norm{\cdot}_{\cY}$. 
Let $\bD_{\cY}:\cY\times\cY\rightarrow\reals_+$ be a Bregman distance function corresponding to $\varphi_{\cY}$,
i.e., $\bD_{\cY}(y,\bar{y})\triangleq \varphi_{\cY}(y)-\varphi_{\cY}(\bar{y})-\fprod{\grad \varphi_{\cY}(\bar{y}),y-\bar{y}}$.
The dual space of $\cY$ is denoted by 
$\cY^*$, and 
$\norm{\cdot}_{\cY^*}:\cY^*\rightarrow\reals$ such that $\norm{y'}_{\cY^*}\triangleq\max\{\fprod{y',y}:\ \norm{y}_{\cY}\leq 1\}$ denotes the dual norm. 
\na{
For each $i\in\cM$, given an arbitrary norm $\norm{\cdot}_{\cX_i}$ on $\cX_i$, \sa{similarly} define $\norm{\cdot}_{\cX_i^*}:\cX_i^*\rightarrow\reals$ and $\bD_{\cX_i}:\cX_i\times\cX_i\rightarrow\reals$ for some $\varphi_{\cX_i}$ 
that is differentiable and 1-strongly convex with respect to
$\norm{\cdot}_{\cX_i}$.}
\end{defn}
{\begin{defn}
\na{Given} a diagonal matrix \rev{$\bC=\diag([c_i]_{i\in\cM})$ for some $c_i\geq 0$ for $i\in\cM$, define $\norm{\cdot}_{\bC}:\cX\rightarrow\reals$ such that $\norm{\bx}_{\bC}\triangleq\sqrt{\sum_{i\in\cM}c_i\rev{\norm{x_i}_{\cX_i}^2}}$; furthermore, $\bD_{\cX}^{\bC}(\bx,\bar{\bx})\triangleq\sum_{i\in\cM}c_i\bD_{\cX_i}(x_i,\bar{x}_i)$ for all $\bx,\bar{\bx}\in\cX$.}
\end{defn}}%
Next, we state our assumptions on {$f$, $h$ and $\Phi$}. 
\begin{assumption}\label{assum}
Suppose, \sa{for all $i\in\cM$,} $f_i:\cX_i\rightarrow\reals\cup\{+\infty\}$ is a closed convex function and \na{its convexity modulus 
\sa{with respect to} $\norm{\cdot}_{\cX_i}$ is $\mu_i\geq 0$,} and $h:\cY\rightarrow\reals\cup\{+\infty\}$ is a closed convex function. \na{Moreover, suppose that $\{f_i\}_{i\in\cM}$, $h$ and $\Phi:\cX\times\cY\rightarrow\reals$ satisfy the following assumptions:}\\
{\bf (i)} for any fixed {$y\in\dom h$,} $\Phi(
\cdot,y)$ is convex \rev{on $\dom f$,} coordinate-wise Lipschitz differentiable {on an open set containing $\dom f$}, 
and for all $i\in\cM$, there exists \rev{$L_{x_ix_i}\geq 0$ 
such that} {for any $\bar{\bx}\in\dom f$ and $v\in\cX_i$ satisfying $\bar{\bx}+U_iv\in\dom f$, and $\bar{y}\in\dom h$,}
\begin{equation}
\label{eq:Lxx}
\norm{\grad_{x_i} \Phi(\bar{\bx}+U_iv,\bar{y})-\grad_{x_i} \Phi(\bar{\bx},\rev{\bar{y}})}_{\na{\cX_i^*}}\leq L_{x_ix_i}\norm{v}_{\na{\cX_i}};
\end{equation}
{\bf (ii)} for any fixed {$\bar{\bx}\in\dom f$, $\Phi(\bar{\bx},\cdot)$ is concave \rev{on $\dom h$,} differentiable on an open set containing $\dom h$,} {and there exists $L_{yy}\geq 0$ and {$L_{yx_i}>0$} for all $i\in\cM$} such that for any {$y,\bar{y}\in\dom h$, $v\in\cX_i$ and $i\in\cM$ satisfying $\bar{\bx}+U_iv\in\dom f$,}
\begin{align}
\norm{\grad_y \Phi(\bar{\bx}+U_iv,\bar{y})-\grad_y \Phi(\bar{\bx},y)}_{\cY^*}\leq L_{yy}\norm{y-\bar{y}}_\cY + L_{yx_i}\norm{v}_{\na{\cX_i}}\rev{;} \label{eq:Lyy}
\end{align}
{\bf (iii)} {
for any $i\in\cM$, $\argmin_{x_i \in \cX_i} \big\{ tf_i(x_i)+\fprod{s,x_i}+{\bD_{\cX_i}(x_i,\bar{x}_i)} \big\}$ can be computed efficiently for any $\bar{x}_i\in\dom f_i$, $s\in\cX_i^*$ and $t>0$.} Similarly, $\argmin_{y\in\cY}\{th(y)+\fprod{s,y}+{\bD_\cY(y,\bar{y})}\}$ is easy to compute for any $\bar{y}\in\dom h$, $s\in\cY^*$ and $t>0$.
\end{assumption}
We also define some \sa{$M\times M$} diagonal matrices to simplify the notation in the rest of the paper:
\begin{align}
\label{eq:diagonals}
\fM\triangleq \diag([\mu_i]_{i\in\cM}),\quad \bL_{\bx\bx}\rev{\triangleq}\diag([L_{x_ix_i}]_{i\in\cM}),\quad \bL_{y\bx}\rev{\triangleq}\diag([L_{yx_i}]_{i\in\cM}).
\end{align}%
Moreover, we define the 
\rev{largest} coordinate-wise Lipschitz constants \sa{for $\{L_{x_ix_i}\}_{i\in\cM}$ and $\{L_{y x_i}\}_{i\in\cM}$:}
\begin{align}\label{def:Lmax}
\overline{L}_{xx}\triangleq \max_{i\in\cM}L_{x_ix_i},\qquad \overline{L}_{yx}\triangleq \max_{i\in\cM} L_{yx_i}.
\end{align}

\begin{algorithm}[h!]
   \caption{Randomized Block-coordinate {Accelerated} Primal-Dual~\rev{(\texttt{RB-APD})} {Algorithm}}
   \label{alg:RBPD}
\begin{algorithmic}[1]
   \STATE {\bfseries Input:} 
   $(\bx_0,y_0)\in 
   \rev{\dom f\times \dom h}$, $\{\mu_i\}_{i\in\cM}\subseteq \reals_+$, 
   \rev{$\gamma^0>0$, $\bar{\tau}\in\left(0,\tfrac{1}{\bar{\mu}(M-1)}\right)$, where $\bar\mu\triangleq \max_{i\in\cM}\mu_i$} 
   \STATE {$(\bx_{-1},y_{-1})\gets(\bx_0,y_0)$},\quad  $\underbar{$\mu$}\gets\min_{i\in\cM}\mu_i$   
   \STATE \rev{$\tilde{\tau}^0\gets\bar{\tau}$,\quad  $\sigma^{-1}\gets\gamma^0\bar{\tau}$}
   \FOR{$k\geq 0$}
   \STATE $\sigma^k\gets \gamma^k\tilde{\tau}^k$,\quad $\theta^k\gets \frac{\sigma^{k-1}}{\sigma^k}$
   		\STATE {$q^k\gets {M}(\grad_y \Phi(\bx^{k},{y^{k}})-\grad_y\Phi(\bx^{k-1},{y^{k-1}}))$}
		\STATE \label{algeq:s} $s^k\gets \grad_y \Phi(\bx^{k},{y^k})+\theta^{k}{q^k}$
   		\STATE \label{algeq:y} $y^{k+1}\gets \argmin_{y\in\cY} h(y){-\fprod{s^k,y}}+\frac{1}{\sigma^k}\bD_\cY(y,y^k)$\label{algeq:y-problem} 	
   		\STATE Choose $i_k\in\cM$ uniformly at {random}
        \STATE \rev{$\tau_{i_k}^k\gets \big(\frac{1}{M}(\mu_{i_k}+\frac{1}{\tilde{\tau}^k})-\mu_{i_k}\big)^{-1}$}
        \STATE {$\bx^{k+1}\gets \bx^k$}
        \STATE \label{algeq:x} $x_{i_k}^{k+1}\gets \argmin_{\rev{x\in\cX_{i_k}}} f_{i_k}(x)+\fprod{\grad_{x_{i_k}}\Phi(\bx^k,y^{k+1}),x}+\frac{1}{\tau_{i_k}^k}\bD_{\cX_{i_k}}(x,x_{i_k}^k)$
   	\STATE \rev{$\gamma^{k+1}\gets \gamma^k(1+\underbar{$\mu$}\tilde{\tau}^k)$}
   	\STATE $\tilde{\tau}^{k+1}\gets\tilde{\tau}^k\sqrt{\frac{\gamma^k}{\gamma^{k+1}}}$,\quad  $k\gets k+1$
   \ENDFOR
\end{algorithmic}
\end{algorithm}
\section{Randomized \sa{Block-Coordinate} Accelerated Primal-dual \sa{Algorithms}}
In this section, we state 
\sa{Randomized Block-Coordinate Accelerated Primal-Dual~(\texttt{RP-APD}) algorithm and its extension incorporating backtracking line search~(\texttt{RB-APD-B})} \rev{to solve~\eqref{eq:original-problem} for some given arbitrary norms $\norm{\cdot}_{\cX_i}$ on $\cX_i$ \sa{and $\norm{\cdot}_{\cY}$ on $\cY$, and some Bregman functions $\bD_{\cX_i}$ and $\bD_{\cY}$} as in Definition~\ref{def:bregman}.} 
We \sa{first} propose a randomized block-coordinate accelerated primal-dual (\texttt{RB-APD}) method (see  {Algorithm~\ref{alg:RBPD}}) consists of a single loop primal-dual steps. After the initialization of parameters, a dual ascent step is taken in the direction of $\grad_y\Phi$ with a momentum term in terms of $\grad_y\Phi$ to gain acceleration for general convex-concave problems. This can be viewed as a generalization of the approach proposed by~\cite{hamedani2021primal} with $M=1$ (\na{it also generalizes the commonly used extrapolation step when the function $\Phi$ is bilinear}\footnote{Majority of the existing methods use  past iterates to gain momentum, \rev{e.g., \cite{chambolle2011first,chambolle2016ergodic,chambolle2017stochastic,alacaoglu2022convergence,alacaoglu2022complexity}} use the momentum term $(1+\theta^k)\bx^k-\theta^k \bx^{k-1}$. This iteration can be recovered by our method when $\Phi$ is bilinear.}). Then, a primal block-coordinate descent step is taken 
{using} $\grad_{x_i}\Phi$ for a uniformly chosen random block coordinate.

\begin{algorithm}[h!]
   \caption{Randomized Block-coordinate Accelerated Primal-Dual algorithm with Backtracking (\texttt{RB-APD-B})}
   \label{alg:RBPDB}
\begin{algorithmic}[1]
   \STATE {\bfseries Input:} 
   $(\bx_0,y_0)\in 
   \rev{\dom f\times \dom h}$, $\{\mu_i\}_{i\in\cM}\subseteq \reals_+$, $c_\alpha,c_\beta,\delta\geq 0$, $\eta\in (0,1)$, 
   $\gamma_0>0$, \rev{$\bar{\tau}\in\left(0,\tfrac{1}{\bar{\mu}(M-1)}\right)$, where $\bar\mu\triangleq \max_{i\in\cM}\mu_i$}
   \STATE {$(\bx_{-1},y_{-1})\gets(\bx_0,y_0)$},\quad $\underbar{$\mu$}\gets\min_{i\in\cM}\mu_i$
   \STATE $\tilde{\tau}^0\gets\bar{\tau}$,\quad  \rev{$\sigma^{-1}\gets\gamma^0\bar{\tau}$}, 
   \FOR{$k\geq 0$}
   \STATE Choose $i_k\in\cM$ uniformly at {random}
   \LOOP
   \STATE $\sigma^k\gets \gamma^k\tilde{\tau}^k$,\quad $\theta^k\gets \frac{\sigma^{k-1}}{\sigma^k}$
   \STATE $\alpha^{k+1}\gets c_\alpha/\sigma^k$,\quad $\beta^{k+1}\gets c_\beta/\sigma^k$
   		\STATE $q^k\gets {M}(\grad_y \Phi(\bx^{k},{y^{k}})-\grad_y\Phi(\bx^{k-1},{y^{k-1}}))$
		\STATE \label{algeq:s-B} $s^k\gets \grad_y \Phi(\bx^{k},{y^k})+\theta^{k}{q^k}$
   		{\STATE \label{algeq:y} $y^{k+1}\gets \argmin_{y\in\cY} h(y)-\fprod{s^k,y}+\frac{1}{\sigma^k}\bD_\cY(y,y^k)$} 		
        \STATE \rev{$\tau_{i_k}^k\gets \big(\frac{1}{M}(\mu_{i_k}+\frac{1}{\tilde{\tau}^k})-\mu_{i_k}\big)^{-1}$}
        \STATE {$\bx^{k+1}\gets\bx^k$}
        \STATE \label{algeq:x} 
        {
            $x_{i_k}^{k+1}\gets
        \argmin_{x\in\cX_{i_k}}f_{i_k}(x)+\langle\grad_{x_{i_k}}\Phi(\bx^k,y^{k+1}),x\rangle+\na{\frac{1}{\tau_{i_k}^k}\bD_{\cX_{i_k}}(x,x_{i_k}^k)}$
        } 
   	\IF{$\eh{C^k_*}\leq -\delta\Big[\frac{\rev{M}}{\tau_{i_k}^k}\bD_{\cX_i}(x_{i_k}^{k+1},x_{i_k}^k)+\frac{1}{\sigma^k}\bD_\cY(y^{k+1},y^k)\Big]$}
   	\STATE {\bf go to} {line~21}
   	\ELSE
   	    \STATE $\tilde{\tau}^k\gets\eta\tilde{\tau}^k$
   	\ENDIF
   	\ENDLOOP
   	\STATE \label{algeq:exit} \rev{$\gamma^{k+1}\gets \gamma^k(1+\underbar{$\mu$}\tilde{\tau}^k)$}
   	\STATE $\tilde{\tau}^{k+1}\gets\tilde{\tau}^k\sqrt{\frac{\gamma^k}{\gamma^{k+1}}}$,\quad  $k\gets k+1$
   \ENDFOR
\end{algorithmic}
\end{algorithm}

In many practical settings, 
\sa{the Lipschitz constants may not be readily available and computing/estimating them} to select an appropriate step-size can be difficult. Therefore, we propose a novel backtracking linesearch that can be combined with \texttt{RB-APD} (called \texttt{RB-APD-B}) to adaptively select primal and dual step-sizes without the knowledge of Lipschitz constants. To this end, we will define a test function \eh{$C^k_*$} 
\rev{that 
implicitly estimates local Lipschitz constants 
in order to accept or reject the tested primal-dual step-sizes at each bactracking iteration.} In fact, at iteration $k\geq 0$ such a test function can be calculated using only the information 
\rev{related to} coordinate $i_k$, \rev{i.e., checking the test function does not involve computing $\grad_{x_i}\Phi$ for $i\in\cM\setminus\{i_k\}$}. Starting from an initial arbitrary step-size, at each iteration we reduce the step-sizes by a factor of $\eta$ until \eh{$C^k_*$} 
falls below a certain threshold. The details of 
\rev{our method} is displayed in Algorithm~\ref{alg:RBPDB}. 

Next, we formally define \sa{the test function we adopted within \texttt{RB-APD-B}}.
\begin{defn}
\label{eq:Tk}
For any $k\geq 0$, given \rev{$\tilde\tau^k,\sigma^k,\theta^k>0$, \sa{and $c_\alpha,c_\beta\geq 0$ such that $M(c_\alpha+c_\beta)\leq 1$,} 
define $\bT^k\triangleq\diag\left(\big[\frac{1}{\tau_i^k}\big]_{i\in\cM}\right)$ 
\sa{where} $\tau_i^k\triangleq \big(\frac{1}{M}(\mu_{i}+\frac{1}{\tilde{\tau}^k})-\mu_{i}\big)^{-1}$ for $i\in\cM$.}
We define the test function for the backtracking line-search as \sa{follows:}
\rev{\small{
\begin{align}\label{eq:test-function}
C^k_*\triangleq \nonumber &M\Big(\Phi(\bx^{k+1},y^{k+1})-\Phi(\bx^k,y^{k+1})-\fprod{\grad_\bx\Phi(\bx^k,y^{k+1}),~\bx^{k+1}-\bx^k}\Big)\nonumber\\
&+\frac{M\sigma^k}{2c_\alpha}\norm{\grad_y\Phi(\bx^{k+1},y^{k+1})-\grad_y\Phi(\bx^k,y^{k+1})}_{\cY^*}^2+\frac{M\sigma^k}{2c_{\beta}}\norm{\grad_y\Phi(\bx^k,y^{k+1})-\grad_y\Phi(\bx^k,y^k)}^2_{\cY^*}
\nonumber\\
&-M\bD_\cX^{\bT^k}(\bx^{k+1},\bx^k) -\Big(\frac{1-M(c_\alpha+c_\beta)}{\sigma^k}\Big)\bD_{\cY}(y^{k+1},y^k).
 \end{align}}}%
\end{defn}
\begin{remark}
\sa{One can also consider an alternative \eh{test function} 
involving only partial gradients:} 
{\small
\begin{align*}
\widetilde{C}^k_*\triangleq & M\fprod{\grad_{x_{i_k}}\Phi(\bx^{k+1},y^{k+1})-\grad_{x_{i_k}}\Phi(\bx^k,y^{k+1}),x_{i_k}^{k+1}-x_{i_k}^k} +\frac{M\sigma^k}{2c_\alpha}\norm{\grad_y\Phi(\bx^{k+1},y^{k+1})-\grad_y\Phi(\bx^k,y^{k+1})}_{\cY^*}^2\nonumber\\
&+\frac{M\sigma^k}{2c_\beta}\norm{\grad_y\Phi(\bx^k,y^{k+1})-\grad_y\Phi(\bx^k,y^k)}^2_{\cY^*}  -\frac{M}{\tau_{i_k}^k}\bD_{\cX_i}(x_{i_k}^{k+1},x_{i_k}^k) -\Big(\frac{1-M(c_\alpha+c_\beta)}{\sigma^k}\Big)\bD_{\cY}(y^{k+1},y^k).\nonumber
\end{align*}}%
Note that convexity of $\Phi(\cdot,y)$ 
implies that $\widetilde{C}^k_*$ upper bounds $C^k_*$.
It is worth highlighting that \sa{both $C_*^k$ and $\widetilde{C}^k_*$ only use} the partial gradient information of $\grad_{x_{i_k}}\Phi$ and step-size $\tau_{i_k}^k$ 
\rev{related to the randomly picked coordinate $i_k\in\cM$, and does not involve computing $\grad_{x_i}\Phi$ \sa{at $\bx^{k+1}$ for any} $i\in\cM\setminus\{i_k\}$}.
\end{remark}

\section{Convergence Analysis}
\label{sec:methodology}
In this section, we discuss the convergence properties of {\texttt{RB-APD-B} and \texttt{RB-APD} algorithms} in Theorems~\ref{thm:backtrack-body} and~\ref{thm:main-body}, respectively, which {are the main results} of this paper. All related proofs are provided in the appendix. 
In the rest, $\bE[\cdot]$ denotes the expectation 
operation.

\begin{assumption}
\label{rem:bregman}
    For the case $\underbar{$\mu$}>0$, we assume that the Bregman distance generating function $\varphi_{\cX_i}(x_i)=\frac{1}{2}\norm{x_i}_{\cX_i}^2$ for $\norm{x_i}_{\cX_i}=\sqrt{\fprod{x_i,x_i}}$ for all $i\in\cM$, which leads to $\bD_{\cX}(\bx,\bx')=\frac{1}{2}\norm{\bx-\bx'}^2$. On the other hand, for the case $\underbar{$\mu$}=0$, \sa{after setting $\fM=\mathbf{0}_{M\times M}$ without loss of generality, i.e., treating $\mu_i=0$ for all $i\in\cM$,} a more general distance generating function $\varphi_{\cX_i}(x_i)$ can be chosen as defined in Definition~\ref{def:bregman} assuming that it has a $L_{\varphi_{\cX}}$-Lipschitz continuous gradient for all $i\in\cM$. \sa{That said,} for the case $\underbar{$\mu$}=0$, one can still work with the original $\fM$ without setting it to $\mathbf{0}_{M\times M}$ if one uses quadratic Bregman as in \sa{the case} $\underbar{$\mu$}>0$ case, i.e., setting $\bD_{\cX}(\bx,\bx')=\frac{1}{2}\norm{\bx-\bx'}^2$.
\end{assumption}

\sa{In \texttt{RB-APD}, stated in~Algorithm~\ref{alg:RBPD}, we considered a particular step size sequence. However, \texttt{RB-APD} can be shown to work for a larger class of step sizes.}
We next describe such step-size conditions to ensure the convergence guarantee for 
\sa{\texttt{RB-APD}}. Indeed, we provide \eh{\emph{(i)} a set of conditions that provide upper bounds on $\tau_{i_k}^k$ and $\sigma^k$ depending on the Lipschitz constants, for any $k\geq 0$ --see \eqref{eq:step-size-condition-pd-new1} and \eqref{eq:step-size-condition-pd-new}; \emph{(ii)} a set of recursive inequalities that connect primal and dual step-sizes --see \eqref{eq:step-size-tau-B} and \eqref{eq:step-size-theta-B}.}
\begin{assumption}\label{assum:step-size}
{There exists $\{[\tau_i^k]_{i\in\cM},\sigma^k,\theta^k\}_{k\geq 0}$ such that $\theta^0=1$, and} 
\begin{subequations}
\begin{align}
&\frac{1-\delta}{\tau_{i_k}^k} \geq L_{x_{i_k}x_{i_k}}+\frac{L^2_{yx_{i_k}}}{\alpha^{k+1}},\label{eq:step-size-condition-pd-new1}\\
&\frac{1-\delta}{\sigma^k} \geq M\theta^k(\alpha^{k}+\beta^k)+\eyh{\frac{ML^2_{yy}}{\beta^{k+1}}}, \label{eq:step-size-condition-pd-new}\\
&t^{k}(\bT^{k}+\fM)\succeq t^{k+1}\big(\bT^{k+1}+(1-\frac{1}{M})\fM\big), \label{eq:step-size-tau-B}\\
&{\frac{t^k}{\sigma^k}\geq \frac{t^{k+1}}{\sigma^{k+1}}},\quad t^{k+1}\theta^{k+1}=t^k, \label{eq:step-size-theta-B}
\end{align}
\end{subequations}
for some positive $\{t^k,\alpha^k\}_{k\geq 0}$ such that $t^0=1$, non-negative $\{\beta^k\}_{k\geq 0}$ and $\delta\in [0,1)$, \sa{where $\bT^k\triangleq [\frac{1}{\tau_i^k}\id_{m_i}]_{i\in\cM}$.} 
\end{assumption}

\begin{remark}\label{rem:step}
\sa{Given $\tilde\tau^0,\gamma^0>0$,} stepsize update rule in Algorithm~\ref{alg:RBPD} implies that \sa{$\sigma^0=\gamma^0\tilde\tau^0$, $\theta^0=1$ and} for all $k\geq 0$, 
\begin{equation*}
\sa{\tau_i^k=\left(\frac{1}{M}\Big(\mu_i+\frac{1}{\tilde\tau^k}\Big)-\mu_i\right)^{-1}\ \forall~i\in\cM,}\quad \theta^{k+1}=\frac{1}{\sqrt{1+{\underbar{$\mu$}}\tilde{\tau}^k}},\quad \sigma^{k+1}=\sigma^k/\theta^{k+1},\quad \tilde{\tau}^{k+1}=\theta^{k+1}\tilde{\tau}^k.
\end{equation*}
Suppose the parameters \sa{$\tilde\tau^0$ and $\gamma^0$} are initialized such that 
\sa{$\bT^0$ and $\sigma^0$ satisfy}
\begin{subequations}
\label{eq:initial_step_condition}
\begin{align}
&\Big((1-\delta)\bT^0-\bL_{\bx\bx}\Big)\frac{1}{\sigma^0}\succeq \frac{1}{c_\alpha}\bL^2_{y\bx},\\
&\ 1-\big(\delta+M(c_\alpha+c_\beta)\big)\geq \frac{\rev{M} L_{yy}^2}{c_\beta}(\sigma^0)^2,
\end{align}
\end{subequations}
for 
\sa{some} $\delta, c_\alpha,\ c_\beta\in\reals_+$ such that $M(c_\alpha+c_\beta)+\delta\leq 1$ satisfying $c_\alpha,~c_\beta> 0$ when $L_{yy}>0$, and $c_\alpha>0$, $c_\beta=0$ when $L_{yy}=0$. Then $\{[\tau_i^k]_{i\in\cM},\sigma^k,\theta^k\}_{k\geq 0}$ satisfies Assumption \ref{assum:step-size} by selecting $t^k=\sigma^k/\sigma^0$.
\end{remark}

\sa{As it is apparent from the remark above, for the convex-concave setting, i.e., $\underbar{$\mu$}=0$, with $\fM=\mathbf{0}$, a constant step size sequence can be selected for the \texttt{RB-APD}, i.e., $\tilde\tau^k=\tilde\tau^0$, $\sigma^k=\sigma^0$, and $\theta^k=1$ for all $k\geq 0$ such that $(\tilde\tau^0,\sigma^0)$ satisfy \eqref{eq:initial_step_condition}. We show that this choice implies $\cO(1/K)$ rate for the expected gap function; hence, $\cO(1/K)$ rate for the primal suboptimality. On the other hand, in strongly convex setting, i.e., $\fM\succ 0$, an accelerated convergence rate of $\cO(1/K^2)$ for the primal suboptimality can be obtained by decreasing the primal step-size $\tau_i^k$ and increasing the dual step-size $\sigma^k$.}

\sa{Next, we first consider the scenario where the Lipschitz constants are not available, and we show that \texttt{RB-APD-B} stated in Algorithm \ref{alg:RBPDB} is well-defined, i.e., the condition in Line 15 holds in finite number of backtracking steps, and the generated step-size sequence $\{[\tau_i^k]_{i\in\cM},\sigma^k,\theta^k\}$ satisfy the conditions in \eqref{eq:step-size-tau-B} and \eqref{eq:step-size-theta-B}. Later we show that \texttt{RB-APD} can be analyzed as a special case of \texttt{RB-APD-B}, and this is why we are considering \texttt{RB-APD-B} first in the following discussion.}

\rev{
\begin{lemma}\label{assum:step-size-B}
Under \sa{Assumptions~\ref{assum} and~\ref{rem:bregman}}, consider \texttt{RB-APD-B} displayed in Algorithm~\ref{alg:RBPDB} for any given $\delta\in[0,1)$ and $c_\alpha,c_\beta\geq 0$ such that $M(c_\alpha+c_\beta)+\delta\leq 1$. When $L_{yy}>0$, set $c_\alpha,~c_\beta> 0$; otherwise, when $L_{yy}=0$, set $c_\alpha>0$ and $c_\beta=0$.
The \texttt{RB-APD-B} iterate and step-size sequences, i.e., $\{\bx^k,y^k\}_{k\geq 0}$ and $\{[\tau^k_i]_{i\in\cM},\sigma^k,\theta^k\}_{k\geq 0}$, are well-defined; more precisely, for any $k\geq 0$, 
\begin{equation}
    \eh{C^k_*}\leq -\delta\Big[\rev{M}\bD_\cX^{\bT^k}(\bx^{k+1},\bx^k)+\frac{1}{\sigma^k}\bD_\cY(y^{k+1},y^k)\Big], \label{eq:step-size-Ck-B}
\end{equation}
holds after finite number of backtracking iterations, where \eh{$C^k_*$} is defined in \eqref{eq:test-function} using $\{
\rev{\sigma^k,\bT^k}\}_k$ and $c_\alpha,c_\beta$ as above. Furthermore, $\{[\tau_i^k]_{i\in\cM},\sigma^k,\theta^k\}_{k\geq 0}$ satisfy \eqref{eq:step-size-tau-B} and \eqref{eq:step-size-theta-B} for $\{t^k\}_{k\geq 0}$ such that $t^k=\sigma^k/\sigma^0$ for $k\geq 0$.
\end{lemma}}
\begin{proof}
    \sa{See Section~\ref{sec:proof-step} in the appendix.}
\end{proof}

We are now ready to state the convergence \sa{guarantees} of \sa{\texttt{RB-APD-B}, stated in Algorithm \ref{alg:RBPDB}.} 
\begin{theorem}\label{thm:backtrack-body} 
\rev{Suppose Assumptions~\ref{assum} \sa{and~\ref{rem:bregman} hold}. Let $\delta\in[0,1)$, $c_\alpha>0$ and $c_\beta\geq 0$ 
\sa{be} chosen as stated below.}
For any given $(\bx_0,y_0)\in\dom f\times \dom h$, $\gamma^0>0$ and $\bar{\tau}\in\big(0,\frac{1}{\bar{\mu}(M-1)}\big)$, \sa{the number of backtracking steps executed within} {\texttt{RB-APD-B}}, stated in Algorithm \ref{alg:RBPDB}, 
\sa{is finite and bounded by {$1+\log_{1/\eta}(\frac{\bar{\tau}}{\Psi})$} uniformly for $k\geq 0$} for some $\Psi>0$\footnote{$\Psi$ is a function of the problem and algorithm parameters, and its exact form is provided in the appendix.}. Let $\{\bx^k,y^k\}_{k\geq 0}$ denote the iterate sequence generated by {\texttt{RB-APD-B}}. For $K\geq 1$, let $T^K{\triangleq}\sum_{k=0}^{K-1}t^k$, $\bar{\bx}^K=\frac{M}{T^K+M-1}\Big(\sum_{k=0}^{K-2}\big(t^k-(1-\tfrac{1}{M})t^{k+1}\big)\bx^{k+1}+t^{K-1}\bx^K\Big)$ and $\bar{y}^K=\frac{1}{T^K}\sum_{k=0}^{K-1}t^ky^{k+1}$ for $\{t^k\}_{k\geq 0}$ such that $t^k=\sigma^k/\sigma^0$ for $k\geq 0$.\\
\textbf{(Part I.)} Suppose $\underbar{$\mu$}=0$ and \rev{$\dom f\times \dom h$ is compact}. \sa{If $L_{yy}>0$, then
assume that} $M(c_\alpha+c_\beta)+\delta{\leq 1}$ {holds for some $c_\alpha,c_\beta>0$;}
\sa{otherwise, i.e., $L_{yy}=0$, assume that 
${M} c_\alpha+\delta{\leq 1}$ for some $c_\alpha>0$ and let $c_\beta=0$. For any $(\bx,y)\in\dom f\times \dom h$, define}
\begin{align*}
    B_1(\bx)\triangleq M\bD_\cX^{(1+\frac{1}{M})\bT^0+\fM}(\bx,\bx^0),\qquad 
B_2(y)\triangleq \Big(\tfrac{1}{\sigma^0}+\theta^0(M-1)L_{yy}\Big)\bD_\cY(y,y^0).
\end{align*}
Then, the following bound holds
\rev{\begin{align}\label{eq:general-rate-result-body}
&\bE\left[\cG(\bar\bx^K,\bar y^K)\right]\leq \frac{1}{T^K}\left(\bar B+\sup_{\bx\in\dom f}B_1(\bx)+\sup_{y\in\dom h}B_2(y)\right),\quad \forall~K\geq 1,
\end{align}}%
for some constant\footnote{See appendix for its dependence on algorithm and problem parameters.} $\bar{B}\in\reals_+$, and $T^K 
=\Omega(K)$, implying $\cO(1/K)$ sublinear rate for $\bE\left[\cG(\bar\bx^K,\bar y^K)\right]$. \sa{Finally,} if a saddle point for \eqref{eq:original-problem} exists and $\delta>0$, then 
$\{(\bx^k,y^k)\}_{k\geq 0}$ converges to a saddle point almost surely.\\
\textbf{(Part II.)} Suppose $\underbar{$\mu$}>0$ and $L_{yy}=0$. 
Assume 
${M} c_\alpha+\delta\in(0,1]$ and $c_\beta=0$.
If a saddle point $(\bx^*,y^*)$ for \eqref{eq:original-problem}
exists, then \sa{$\{\bx^k\}_{k\geq 0}$} converges to $\bx^*$ and $\{y^k\}$ has a limit point almost surely. \sa{Moreover,}
\rev{
 \begin{align}
 \nonumber\bE\left[\tfrac{\gamma^K}{2}\big\|{\bx^*-\bx^{K}}\big\|_{\cX}^2 +(1-M c_\alpha)\bD_{\cY}(y^*,y^{K})\right]& \leq 
 \tfrac{\gamma^0}{2}\norm{\bx^*-\bx^0}^2_{\cX}+\bD_{\cY}(y^*,y^0)\\&\quad +\sigma^0(M-1)\Big(\cL(\bx^0,y^*)-\cL({\bx}^*,y^*)\Big), \label{eq:xy-bound-II-body}
 \end{align}
 \begin{align}
\nonumber\bE\left[F(\bar\bx^K)-F(\bx^*)\right]&\leq \frac{1}{T^K}\Big(\tfrac{\gamma^0}{2\sigma^0}\norm{\bx^*-\bx^0}^2_{\cX}+\tfrac{1}{\sigma^0}\sup_{y\in\dom h}\bD_{\cY}(y,y^0)\\&\quad+(M-1)\big(F(\bx^0)-F(\bx^*)\big)\Big),
\label{eq:F-bound-body}
\end{align}}%
\sa{hold} \rev{for all $K\geq 1$. Furthermore, both $\gamma^K=\Omega(K^2)$ and $T^K=\Omega(K^2)$; hence, ${\bE\big[}\norm{\bx^K-\bx^*}_{\cX}^2{\big]}=\cO(1/K^2)$ and 
$0\leq \bE\left[F(\bar\bx^K)-F(\bx^*)\right]\leq\cO(1/K^2)$.} \sa{Finally, if $\delta>0$, then any limit point of $\{\bx^k,y^k\}$ is a saddle point almost surely.}
\end{theorem}
\begin{proof}
\sa{See Section~\ref{sec:proof} in the appendix. The proof consists of three main parts: in the first part, a one-step result is shown in Section~\ref{sec:proof_onestep} followed by the proof of Lemma~\ref{assum:step-size-B} in Section~\ref{sec:proof-step}; in the second part, asymptotic convergence is shown in Section~\ref{sec:proof-asymp}; and finally, in the third part, convergence rate is established in Section~\ref{sec:proof-rate}.}
\end{proof}
Note that the term $\cL(\bx^0,y^*)-\cL({\bx}^*,y^*)$ in \eqref{eq:xy-bound-II-body} can be bounded above by $F(\bx^0)-F(\bx^*)$.

\sa{Below we state the convergence guarantees of \texttt{RB-APD}, i.e., Algorithm \ref{alg:RBPD}, which can be considered as a special case of \texttt{RB-APD-B}. Indeed, whenever $\bT^0$ and $\sigma^0$ are chosen such that \eqref{eq:initial_step_condition} holds for some $c_\alpha,c_\beta\geq 0$ and $\delta\in[0,1)$ as described in Theorem~\ref{thm:backtrack-body}, one can argue that stopping condition for backtracking in \eqref{eq:step-size-Ck-B} holds for $\bT^k$ and $\sigma^k$ generated by \texttt{RB-APD} in Algorithm~\ref{alg:RBPD} without executing any backtracking step, i.e., $\{\bx^k,y^k\}_{k\geq 0}$ generated by \texttt{RB-APD}, corresponding to $\{\bT^k,\sigma^k,\theta^k\}_{k\geq 0}$ as in Remark~\ref{rem:step}, satisfies \eqref{eq:step-size-Ck-B} for all $k\geq 0$.}
\begin{theorem}\label{thm:main-body}
Suppose Assumptions~\ref{assum} \sa{and~\ref{rem:bregman} hold}, $\gamma^0>0$ and $\bar{\tau}\in\big(0,\frac{1}{\bar{\mu}(M-1)}\big)$ \sa{are chosen such that \eqref{eq:initial_step_condition} holds for some $c_\alpha,c_\beta\geq 0$ and $\delta\in[0,1)$ as described in Theorem~\ref{thm:backtrack-body}}. Let
$\{\bx^k,y^k\}_{k\geq 0}$ be the iterate sequence
generated by {\texttt{RB-APD}} \sa{initialized at an arbitrary $(\bx_0,y_0)\in\dom f\times \dom h$}, and let
\rev{$\bar{\bz}^K=(\bar{\bx}^K,\bar{y}^K)$} and $T^K$ be defined for $K\geq 1$ as in Theorem~\ref{thm:backtrack-body}.

\textbf{(Part I.)} {Suppose $\underbar{$\mu$}=0$ and $\dom f\times\dom h$ is compact. The stepsize rule in Algorithm~\ref{alg:RBPD} implies $\tilde{\tau}^k=\tilde{\tau}^0$, $\sigma^k=\sigma^0$ and $\theta^k=1$ for {all} $k\geq 0$; hence, $t^k=1$ for $k\geq 0$, implying $T^K=K$.} \sa{Moreover, \eqref{eq:general-rate-result-body} holds for all $K\geq 1$, which implies $\bE[\cG(\bar\bz^K)] \leq \cO(1/K)$. Finally,}
if a saddle point for \eqref{eq:original-problem} exists and
\sa{\eqref{eq:initial_step_condition} holds with $\delta>0$}, then 
$\{(\bx^k,y^k)\}_{k\geq 0}$ {almost surely} converges to a saddle point. 

\textbf{(Part II.)} Suppose $\underbar{$\mu$}>0$ and $L_{yy}=0$.
If a saddle point for \eqref{eq:original-problem}
exists, then $\{\bx^k\}_{k\geq 0}$ converges to $\bx^*$ and $\{y^k\}$ has a limit point almost surely.\footnote{Since $\underbar{$\mu$}>0$, $\bx^*$ 
denotes the unique solution to $\min_{\bx\in\cX} F(\bx)$.}  \sa{Furthermore, both \eqref{eq:xy-bound-II-body} and \eqref{eq:F-bound-body} hold for all $K\geq 1$ such that} $\gamma^K=\Omega(K^2)$ and $T^K=\Omega(K^2)$; hence, ${\bE\big[}\norm{\bx^K-\bx^*}_{\cX}^2{\big]}=\cO(1/K^2)$ and 
$0\leq \bE\left[F(\bar\bx^K)-F(\bx^*)\right]\leq\cO(1/K^2)$. \sa{Finally, if \eqref{eq:initial_step_condition} holds with $\delta>0$,} then any limit point \sa{of $\{(\bx^k,y^k)\}_{k\geq 0}$} 
 is a saddle point almost surely.
\end{theorem}
\begin{proof}
    See Section~\ref{sec:proof} in the appendix.
\end{proof}

\section{Numerical Experiment}
In this section, we implement our scheme on a kernel matrix learning problem described in section \ref{sec:intro}, and benchmark with 
\texttt{Mirror-prox} by \cite{he2015mirror}, proximal extra gradient method (\texttt{PEGM}) by \cite{malitsky2018first}, and accelerated primal-dual (\texttt{APDB}) by \cite{hamedani2021primal}. The experiments are performed on a machine running 64-bit Windows 10 with Intel i7-8650U @2.11GHz and 16GB RAM. 

\subsection{Learning a Kernel Matrix}
We {consider the formulation in~\eqref{eq:kernel_learn_simple} for the kernel class $\cK$ in~\eqref{eq:Kernel_Class}, and we} use a similar setup as in \citep{lanckriet2004learning}. In particular, we consider three kernel functions ($q=3$); polynomial kernel function $\phi_1(\ba,\bar{\ba})=(1+\ba^\top\bar{\ba})^2$, Gaussian kernel function $\phi_2(\ba,\bar{\ba})=\exp(-0.5(\ba-\bar{\ba})^\top(\ba-\bar{\ba})/0.1)$, and linear kernel function $\phi_3(\ba,\bar{\ba})=\ba^\top\bar{\ba}$ to compute $K_1, K_2, K_3{\in\mathbb{S}^m_+}$, respectively, {where $m$ denotes the number data points}. We set $\lambda =1$, $c=\sum_{\ell=1}^3r_\ell$, where $r_\ell=\hbox{trace}(K_\ell)$ for $\ell=1,2,3$. The kernel matrices are normalized as in \citep{lanckriet2004learning}; thus, $\diag(K_\ell)={\mathbf{1}_m}$ and $r_\ell=m$ for each $\ell$. We consider {two} different datasets: a subset of \texttt{svmguide1} {($m=3000$, $n=4$)} and a subset of \texttt{MNIST} to classify digits of $4$ and $9$ {($m=7500$, $n=784$)} from \citep{chang2011libsvm}, {where $n$ denotes the number of features}. We use 80\% of data points as the training set and the rest as the test set. 

The algorithms are compared in terms of relative error for the solution ($\|{\bx^k-\bx^*}\|_2/\|{\bx^*}\|_2$), where $(\bx^*,y^*)$ denotes a saddle point for the considered problem. The purpose of this experiment is to benchmark our method against other methods in terms of 
{convergence behavior observed in practice}. To this end, {the optimal primal solution $\bx^*$ 
to} the problem {in}~\eqref{eq:kernel_learn_simple} is 
{computed calling the} 
commercial optimization solver MOSEK through CVX~\citep{grant2008cvx}.

{\bf Effect of number of blocks.} In Figure \ref{fig:block}, we compare the performance of \texttt{RB-APD-B} for different numbers of primal blocks, $M\in\{1,10,50,100,800\}$ for \texttt{svmguide1} and $M\in\{1,10,50,100,1000\}$ for \texttt{MNIST} {datasets}. \texttt{RB-APD-B} with $M=100$ primal blocks has the best performance compared to other block partition strategies. The main reason is that  partitioning minimization variable into blocks reduces per iteration complexity and allows more iterations to be performed in a given time interval. {That said, employing an excessive number of blocks, e.g., $M=1000$ for \texttt{MNIST} dataset, degrades the overall performance measured by the total number of primal and dual oracle calls, {compared} to adopting a moderate number of blocks, e.g., $M=100$ for \texttt{MNIST}. Thus, 
one may conclude that \sa{there is a 
sweet spot for the number of blocks that optimizes the} convergence behavior in terms of oracle complexity.} \vspace*{-2mm}
\begin{figure}[htb]
    \centering
        \hspace*{-3mm}\subfloat[\texttt{svmguide1} dataset]{
    \includegraphics[scale=0.15]{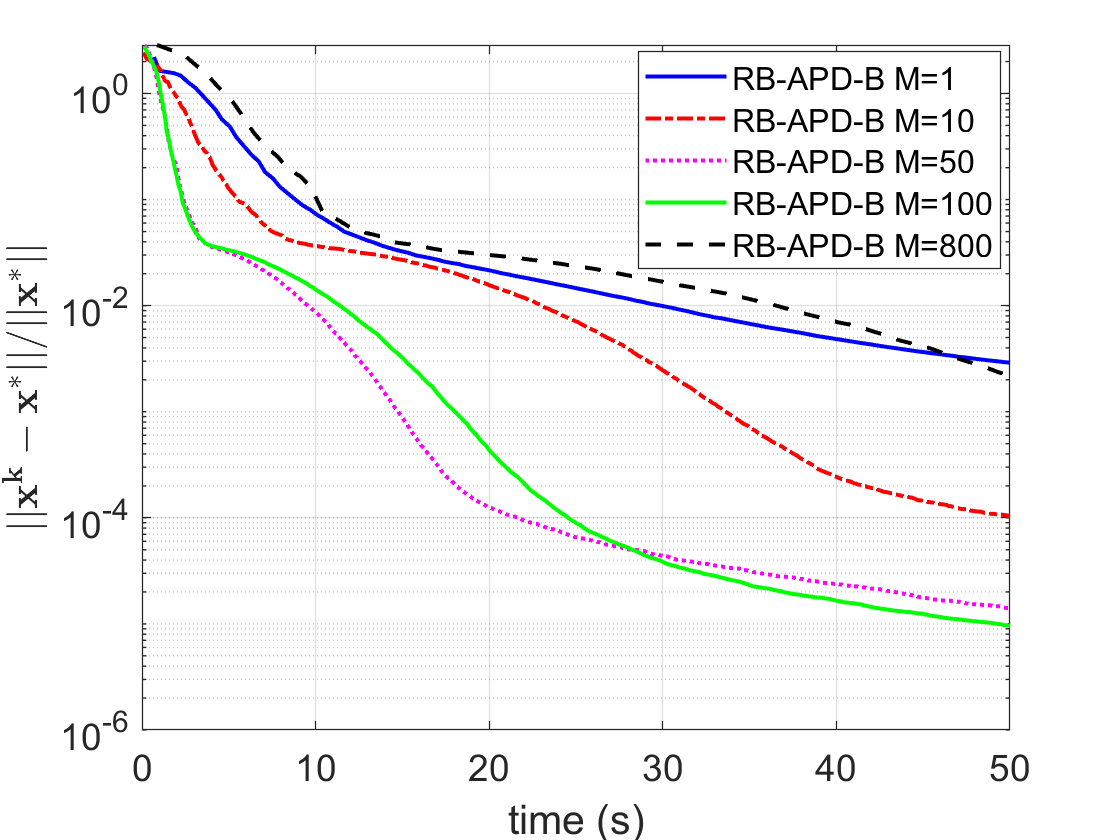}}\hspace*{-5mm}
    \subfloat[\texttt{MNIST} dataset]{
    \includegraphics[scale=0.15]{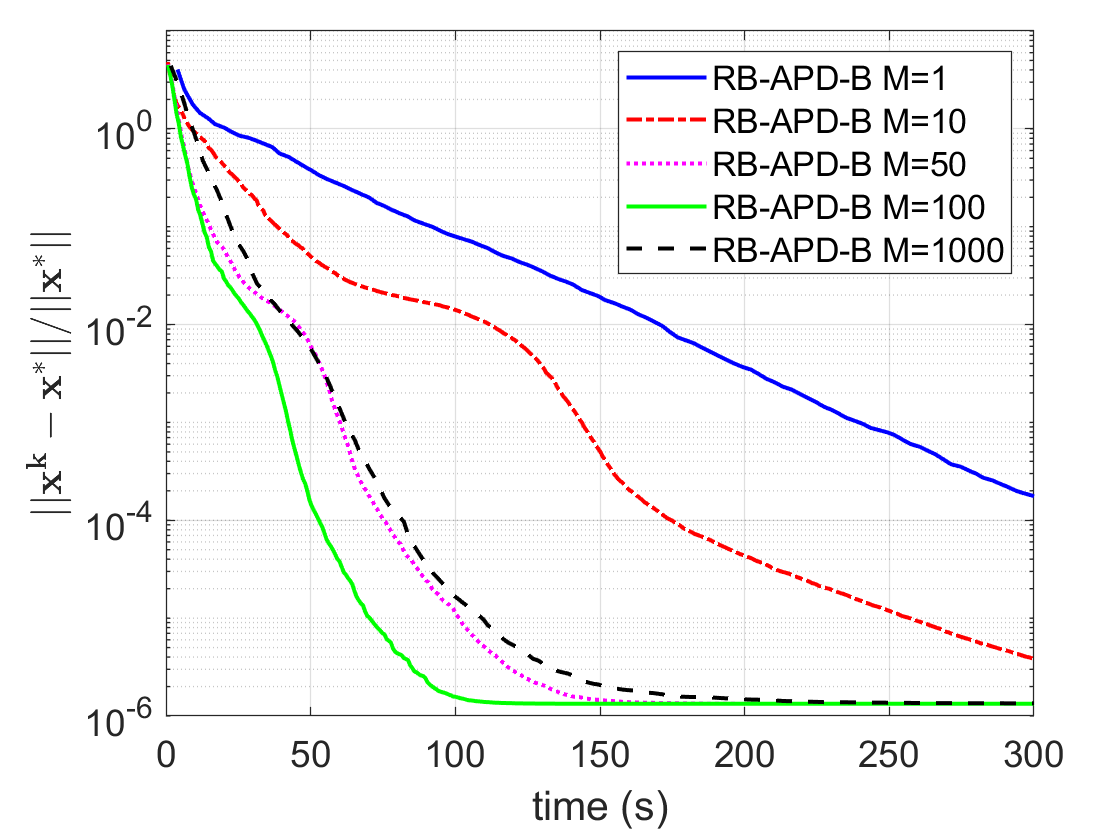}}
    \caption{Performance of \texttt{RB-APD-B} for various number of primal blocks.}
    \label{fig:block}
    \vspace*{-2mm}
\end{figure}

{\bf Comparison with other methods.} The comparison of our method against others is  
{demonstrated} in Figure \ref{fig:all}. In this experiment, we select the step-size of \citep{he2015mirror} constant while other methods enjoy a backtracking linesearch. All the methods use the same initial step-size of $\tau=10^{-2}$ and parameter $\eta=0.7$. We observe that {our 
algorithm \texttt{RB-APD-B} 
is competitive against the state-of-the-art methods we compare}. The reason is that our method potentially benefits from low per-iteration complexity due to using a block-coordinate approach as well as larger block-specific step-sizes due to exploring local coordinate-wise Lipschitz \sa{constants} via backtracking method. 
{\texttt{Mirror-prox} by~\cite{he2015mirror} uses constant step-size depending on the global Lipschitz parameter and 
employs full gradient 
$\grad_{\bx}\Phi$ to update $\bx$ while {both} \texttt{APDB}~\cite{hamedani2021primal} and \texttt{PEGM}~\cite{malitsky2018first} 
{employ} backtracking line search, but similarly use full gradient $\grad_{\bx}\Phi$ to update $\bx$.}\vspace*{-2mm}
\begin{figure}[htb]
    \centering
    \hspace*{-3mm}\subfloat[\texttt{svmguide1} dataset]{
    \includegraphics[scale=0.15]{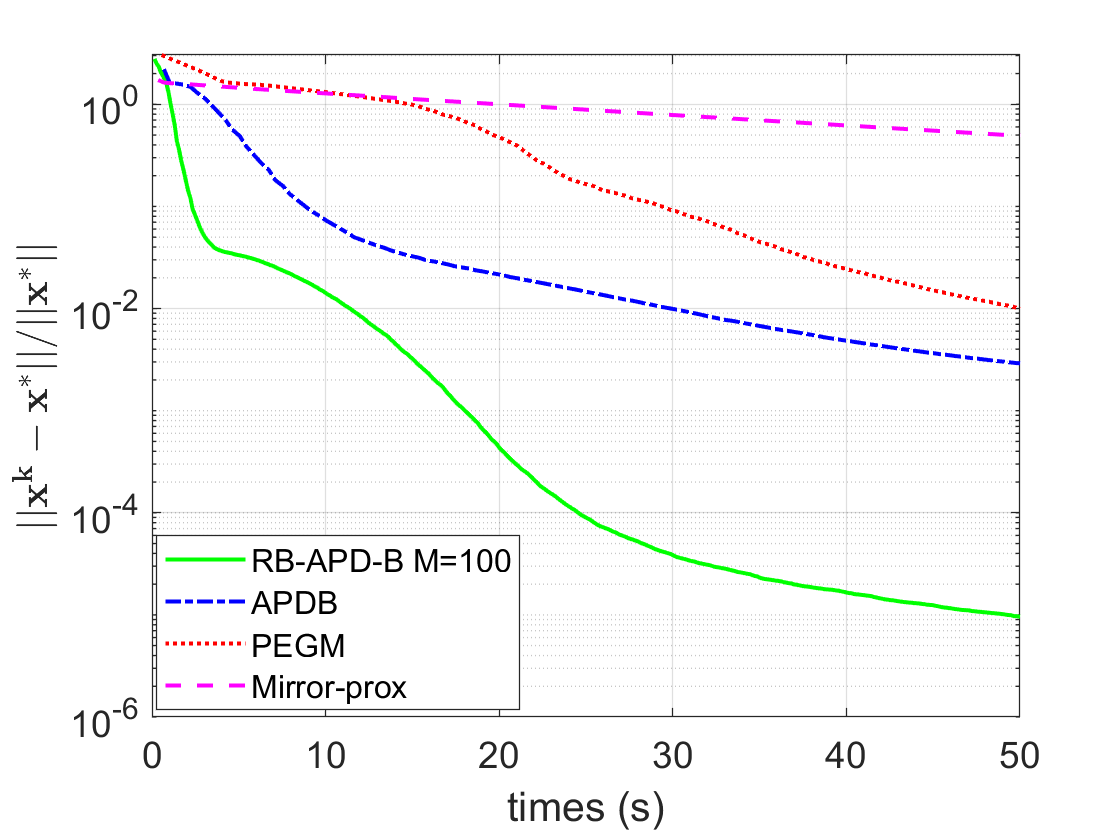}}\hspace*{-4mm}
    \subfloat[\texttt{MNIST} dataset]{
    \includegraphics[scale=0.15]{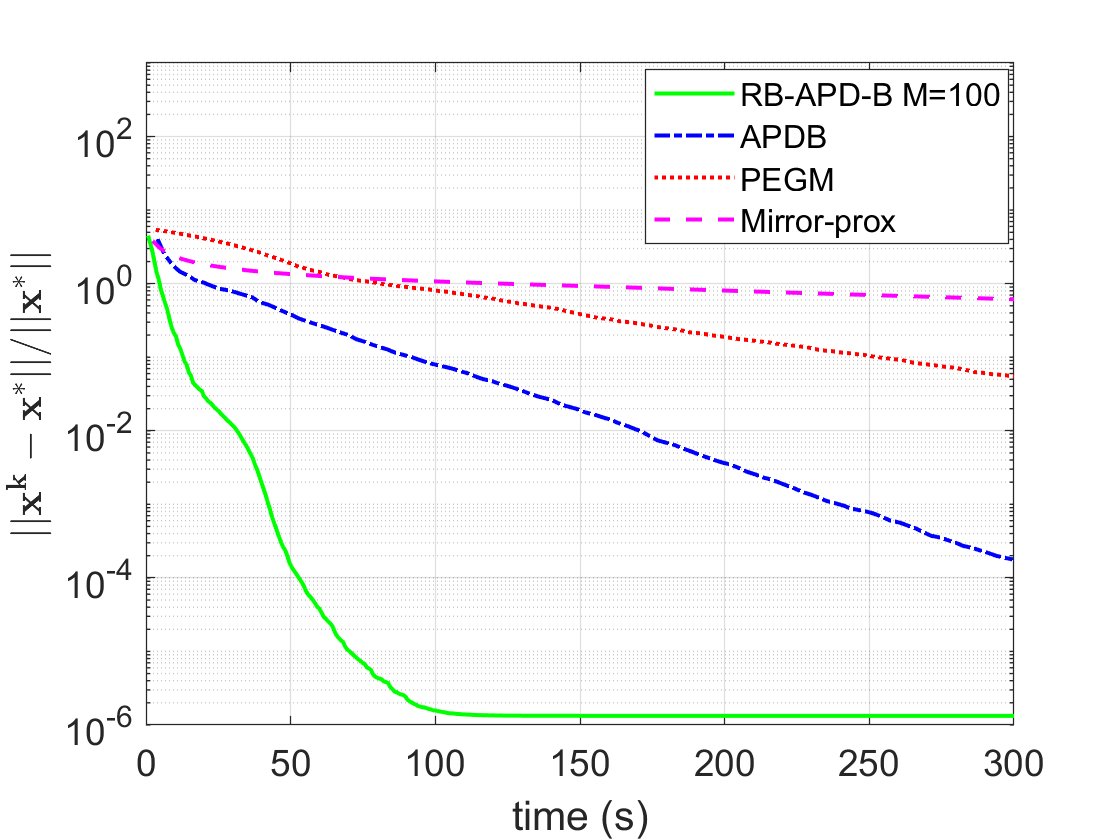}}
    \caption{Comparing the performance of \texttt{RB-APD-B} with other methods for different datasets.}
    \label{fig:all}
    \vspace*{-2mm}
\end{figure}
\vspace*{-2mm}

\subsection{Quadratic constrained quadratic programming (QCQP)}
In this subsection, we compare our method against \texttt{APDB} by~\cite{hamedani2021primal} and \texttt{BLALM} by~\cite{xu2021first} on randomly generated QCQP problems with various dimensions. In fact, we consider the following QCQP problem 
\begin{align*}
&\min_{\bx\in X} ~ f(\bx)\triangleq \frac{1}{2}\bx^\top A_0\bx+b_0^\top \bx\\
&\hbox{s.t.} \quad g(\bx)\triangleq \frac{1}{2}\bx^\top A_1\bx+b_1^\top \bx -c_1\leq 0, 
\end{align*}
where $X = [-1, 1]^m$, \sa{and $A_0, A_1\in \reals^{m\times m}$ are symmetric} positive semidefinite matrices generated
randomly with a block diagonal structure, \sa{$b_0, b_1\in \reals^{m}$} are generated randomly with elements drawn from standard Gaussian distribution, and $c_1 \in \reals$ is generated randomly with elements drawn
from uniform distribution over $[0,1]$.

The goal of this experiment is to examine the effect of the dimension of the primal-variable ($m$) on the runtime of the proposed method and compare it with existing algorithms. To this end, we fix the termination criteria as $\max\{\abs{f(\bx_k)-f(\bx^*)},[g(\bx_k)]_+\}/m\leq \epsilon$ and increase parameter $m$ to compare the running time of algorithms. In particular, we let $\epsilon=10^{-6}$ and $m=10^3i$ for $i\in\{1,3,5,7,9\}$. The plot is shown in Figure \ref{fig:qcqp} and we observe that \texttt{RB-APD-B} outperforms the other two methods and their gap increases as $m$ increases.

\begin{figure}[htb]
\centering
  \includegraphics[scale=0.16]{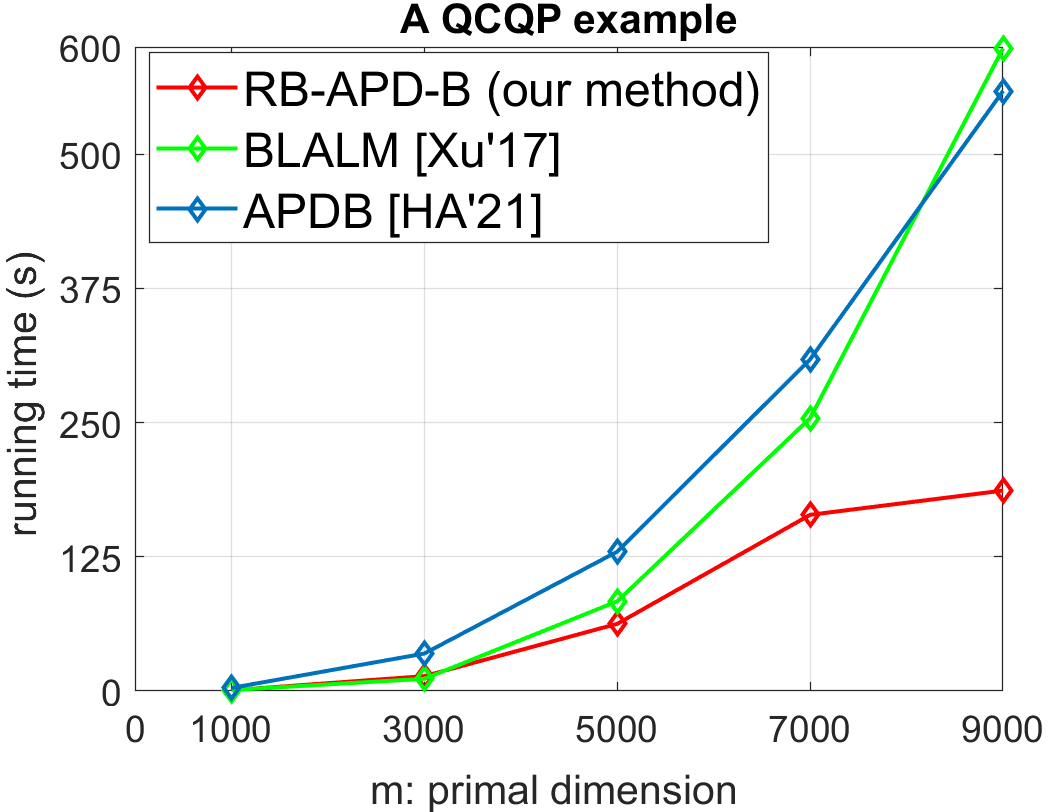}
  \caption{Comparing the effect of $m$ on the running time.}
  \label{fig:qcqp}
\end{figure}

\section{Concluding Remarks}
{The 
step-sizes for \rev{a typical first-order method 
are} selected based on the global Lipschitz constant of the gradient map to guarantee convergence. However, these constants may not be readily available in practice; and even if they are known, they pottentially lead to conservative step-sizes. We developed a novel randomized block coordinate primal-dual method equipped with backtracking line-search \sa{technique} to solve large-scale SP problems. The method can contend with \rev{non-bilinear, non-separable coupling functions} 
$\Phi(\bx,y)$ possibly with multiple primal blocks. At each iteration, a primal block is randomly selected and updated, following a dual update with a momentum term involving $\grad_y\Phi$.} 

We showed that for convex-concave setting, the proposed method achieves $\cO(M/k)$ convergence rate in the expected primal-dual gap. 
Furthermore, {assuming $\Phi(\bx,y)$ is 
strongly convex in $\bx$ and 
affine 
in $y$, 
{our method} enjoys a faster rate of $\cO(M/k^2)$ in terms of 
$\bE[\norm{\bx_k-\bx^*}^2]$ and $\bE[F(\bx^k)-F(\bx^*)]$, where $F(\cdot)$ denotes the primal function.} To the best of our knowledge, our proposed method is the only randomized block-coordinate primal-dual algorithm that can handle the general SP problems {as in \eqref{eq:original-problem}} and our rate results are optimal for this class. Furthermore, our method {avoids step-size selection issues related to the use of global Lipschitz constants through employing backtracking line-search scheme we developed. Indeed, the proposed line-search scheme tailored for 
$\texttt{RB-APD}$ help us alleviate the burden of knowing the global Lipschitz constants by estimating them locally.}


\printbibliography
\newpage

\appendix
\onecolumn
\setcounter{section}{0}
\setcounter{theorem}{0}

\section{SUPPORTING LEMMAS AND DEFINITIONS}\label{append}

\rev{\bf Notation:} In the rest, $\bE[\cdot]$ denotes the expectation 
operation and the conditional expectation is denoted by $\bE^{k}[\cdot]\triangleq\bE\left[\cdot\mid \rev{\cF_{k}}\right]$, where \rev{$\mathcal F_{k}\triangleq\sigma\big(\{i_0,\hdots,\rev{i_{k-1}}\}\big)$} is the $\sigma$-algebra generated by $\{i_0,\ldots,i_{k-1}\}$ for $k\geq 1$. Moreover, \rev{given a 
diagonal matrix $\cA=\diag([a_i]_{i\in\cM})$ for some $\{a_i\}_{i\in\cM}\subset\reals_{++}$,}
for any $\bx\in\cX$ and $\bar\bx\in\dom f$, we define
\begin{align*}
    \bD^{\cA}_{\cX}(\bx,\bar\bx)\triangleq\sum_{i\in\cM}a_i\bD_{\cX_i}(x_i,\bar{x}_i).
\end{align*}
For any $\bdelta\in\cX^*$, we define $\norm{\bdelta}_{*,\cA}\triangleq\sqrt{\sum_{i\in\cM}a_i\norm{\delta_i}_{\cX_i^*}^2}$.

Note that for any $y\in{\dom h}$, {$\bar{\bx}\in\dom f$} \na{and $i\in\cM$}, \eqref{eq:Lxx} implies that
{
\begin{align}
0 \leq {\Phi(\bar{\bx}+U_iv,y)- \Phi(\bar{\bx},y)}-\fprod{\grad_{x_i}\Phi(\bar{\bx},y),~v} \leq\tfrac{1}{2}L_{x_ix_i}\norm{v}_{\na{\cX_i}}^2,  \label{eq:Lxx_bound}
\end{align}}%
{for all $v\in\cX_i$ such that $\bar{\bx}+U_i v\in\dom f$}. {Similarly \eqref{eq:Lyy} and concavity of {$\Phi(\bx,\cdot)$} imply that for any {$\bx\in\dom f$, the following inequality holds for all $y,\bar{y}\in\dom h$:}
{
\begin{align}
0\geq \Phi(\bx,y) - \Phi(\bx,\bar{y})-\fprod{\grad_y\Phi(\bx,\bar{y}),~y-\bar{y}} \geq -\tfrac{1}{2}L_{yy}\norm{y-\bar{y}}_\cY^2. 
\label{eq:Lyy_bound}
\end{align}}}

\eh{Next, we define a test function $C^k(\bx,\by)$ for the iteration $k\geq 0$ to accept or reject the \sa{assignment of a} given point $(\bx,y)$ \sa{to $(\bx^{k+1},y^{k+1})$}. 
\begin{defn}
For any $k\geq 0$, given \rev{$\tilde\tau^k,\sigma^k,\theta^k>0$, 
\sa{let} $\bT^k=\diag\left(\big[\frac{1}{\tau_i^k}\big]_{i\in\cM}\right)$ such that $\tau_i^k\triangleq \big(\frac{1}{M}(\mu_{i}+\frac{1}{\tilde{\tau}^k})-\mu_{i}\big)^{-1}$ for $i\in\cM$.}
We define
{\small
\begin{align}\label{eq:test-function-B}
C^k(\bx,y)\triangleq \nonumber&M\rev{\Big(\Phi(\bx,y)-\Phi(\bx^k,y)-\fprod{\grad_\bx\Phi(\bx^k,y),~\bx-\bx^k}\Big)} +\frac{M}{2\alpha^{k+1}}\norm{\grad_y\Phi(\bx,y)-\grad_y\Phi(\bx^k,y)}_{\cY^*}^2\nonumber\\
&+\frac{M}{2\beta^{k+1}}\norm{\grad_y\Phi(\bx^k,y)-\grad_y\Phi(\bx^k,y^k)}^2_{\cY^*}-M\bD_\cX^{\bT^k}(\bx,\bx^k) -\Big(\frac{1}{\sigma^k}-\theta^k M(\alpha^k+\beta^k)\Big)\bD_{\cY}(y,y^k) 
 \end{align}}%
{for some positive sequence $\{\alpha^k,\beta^k\}$.}
\end{defn}
Indeed, one can easily verify that $C^k_*=C^k(\bx^{k+1},y^{k+1})$ for $\{\alpha^k,\beta^k\}$ sequence defined as in \texttt{RB-APD-B} algorithm.}

\begin{lemma}
\label{lem_app:prox}
Let $\cX$ be a finite dimensional normed vector space with norm $\norm{.}_{\cX}$, $f:\cX\rightarrow\reals\cup\{+\infty\}$ be a closed convex function with convexity modulus $\mu\geq 0$ with respect to $\norm{.}_{\cX}$, and $\bD:\cX\times\cX\rightarrow\reals_+$ be a Bregman distance function corresponding to a strictly convex function $\phi:\cX\rightarrow\reals$ that is differentiable on an open set containing $\dom f$. Given $\bar{x}\in\dom f$ and $t>0$, let
\begin{align}
\label{eq_app:prox}
x^+=\argmin_{x\in\cX} f(x)+t \bD(x,\bar{x}).
\end{align}
Then for all $x\in\cX$, the following inequality holds:
\begin{eqnarray}
f(x)+t\bD(x,\bar{x})\geq f(x^+) + t\bD(x^+,\bar{x})+t \bD(x,x^+)+\frac{\mu}{2}\norm{x-x^+}_{\cX}^2. \label{eq_app:bregman}
\end{eqnarray}
\end{lemma}
\begin{proof}
This result is a trivial extension of Property 1 in~\cite{Tseng08_1J}. The first-order optimality condition for \eqref{eq_app:prox} implies that
$0\in\partial f(x^+)+ t\grad_x \bD(x^+,\bar{x})$ -- where $\grad_x \bD$ denotes the partial gradient with respect to the first argument. Note that for any $x\in \dom f$, we have $\grad_x \bD(x,\bar{x})=\grad \phi(x)-\grad \phi(\bar{x})$. Hence,
$t(\grad \phi(\bar{x})-\grad \phi({x}^+))\in\partial f(x^+).$
Using the convexity inequality for $f$, we get
\begin{eqnarray*}
f(x)\geq f(x^+)+t\fprod{\grad \phi(\bar{x})-\grad \phi({x}^+),~x-x^+}+\frac{\mu}{2}\norm{x-x^+}_{\cX}^2.
\end{eqnarray*}
The result in \eqref{eq_app:bregman} immediately follows from this inequality.
\end{proof}
\begin{lemma}\cite{Robbins71}\label{lem:supermartingale}
{Let $(\Omega, \cF,\mathbb{P})$ be a probability space, and for each $k\geq 0$ suppose $a^k$ and $b^k$ are finite nonnegative $\cF^k$-measurable random variables where $\{\cF^k\}_{k\geq 0}$ is a sequence sub-$\sigma$-algebras of $\cF$ such that $\cF^k\subset\cF^{k+1}$ for $k\geq 0$. If $\mathbb{E}[a^{k+1}|\cF^k]\leq a^k-b^k$ \sa{for all $k\geq 0$}, then
then $a=\lim_{k\rightarrow \infty}a^k$ exists almost surely, and $\sum_{k=0}^\infty b^k<\infty$.}
\end{lemma}

\begin{lemma}\label{lem:inner-w}
\rev{Given a 
diagonal matrix $\cA=\diag([a_i]_{i\in\cM})$ for some $\{a_i\}_{i\in\cM}\subset\reals_{++}$, and an arbitrary sequence} 
$\{\bdelta^k\}_{k\geq 0}\subset\rev{\cX_*}$, \rev{let $\{\bv^k\}_{k\geq 0}\subset\cX$ be 
such that} 
$\bv^{k+1}\triangleq\argmin_{\bx\in\cX}$ $\{-\fprod{\bdelta^k,\bx}+\bD_\cX^{\cA}(\bx,\bv^k)\}$. Then for all $k\geq 0$ and $\bx\in\cX$, $$\fprod{\bdelta^k,\bx-\bv^k}\leq \bD_\cX^{\cA}(\bx,\bv^k)-\bD_\cX^{\cA}(\bx,\bv^{k+1})+\frac{1}{2}\norm{\bdelta^k}_{*,\cA^{-1}}^2.$$
\end{lemma}
\begin{proof}
Since $\bv^{k+1}$ computation is separable in 
$i\in\cM$, one can apply Lemma \ref{lem_app:prox} for each coordinate 
to obtain a bound for $\fprod{\bdelta^k,\bx-\bv^{k+1}}$. Then we have that
{\footnotesize
\begin{align*}
\fprod{\bdelta^k,\bx-\bv^k}&=\fprod{\bdelta^k,\bx-\bv^{k+1}}+\fprod{\bdelta^k,\bv^{k+1}-\bv^k}\nonumber\\
&\leq \bD_\cX^{\cA}(\bx,\bv^k)-\bD_\cX^{\cA}(\bx,\bv^{k+1})-\bD_\cX^{\cA}(\bv^{k+1},\bv^k)+\fprod{\bdelta^k,\bv^{k+1}-\bv^k}\nonumber\\
&\leq \bD_\cX^{\cA}(\bx,\bv^k)-\bD_\cX^{\cA}(\bx,\bv^{k+1})+\frac{1}{2}\norm{\bdelta^k}_{*,\cA^{-1}}^2,
\end{align*}}%
where in the last inequity we used 
$\fprod{\delta^k_i,v_i^{k+1}-v^k_i}\leq\tfrac{a_i}{2}\norm{v_i^{k+1}-v^k_i}_{\cX_i}^2+\tfrac{1}{2a_i}\norm{\delta_i^k}_{\cX_i^*}^2$ for $i\in\cM$ together with $\bD_\cX^\cA(\bv^{k+1}-\bv^k)\geq \tfrac{1}{2}\norm{\bv^{k+1}-\bv^k}_\cA^2$.
\end{proof}

\section{Proof of Convergence}
\label{sec:proof}
Before proving the asymptotic convergence of the iterate sequence and related rate results, we restate the theorems with more details here for completeness. We divide the proof of theorems in three parts: (i) In section \ref{sec:proof-step}, we analyze the proposed backtracking method to derive some key inequalities; 
(ii) in section \ref{sec:proof-asymp}, we show asymptotic convergence of the iterate sequence generated by \texttt{RB-APD} and \texttt{RB-APD-B}; (iii) finally, in section \ref{sec:proof-rate} we establish the convergence rate of the proposed methods. 

\begin{remark}
Note that when the problem is strongly convex, we assume that $L_{yy}=0$ which means that $\Phi(\bx,\cdot)$ is affine for any $\bx\in\cX$. In this scenario, \rev{let $\bx^*$ be the unique solution to $\min\{F(\bx):\ \bx\in\cX\}$ for $F$ defined in \eqref{eq:primal_problem}.} \sa{It} is desirable to provide convergence guarantees for \rev{computing an $\epsilon$-optimal point $\bx_\epsilon$. In this context, $\epsilon$-optimality can be defined either in function values, i.e., $\bE[F(\bx_\epsilon)-F(\bx^*)]\leq \epsilon$, or 
in the solution space, i.e., $\bE[\norm{\bx_\epsilon-\bx^*}_\cX^2]\leq \epsilon$.}
\end{remark}

We begin by restating the results for \rev{\texttt{RB-APD-B}, stated in Algorithm \ref{alg:RBPDB}, under merely convex and strongly convex settings separately.}
\begin{theorem}\label{thm:backtrack} 
Suppose \sa{Assumptions~\ref{assum} and~\ref{rem:bregman} hold}. Let $\delta\in[0,1)$, $c_\alpha>0$ and $c_\beta\geq 0$ are chosen as stated below. Define \rev{$\munderbar{\mu}=\min_{i\in\cM}\mu_i$}, $\bar\mu\triangleq \max_{i\in\cM}\mu_i$, \eh{$\overline{L}_{xx}=\max_{i\in\cM}L_{x_ix_i}$}, \eh{$\overline{L}_{yx}=\max_{i\in\cM}L_{yx_i}$}, and
\eh{
\begin{align}
	\label{eq:tauhat_bound-1,2}
	&\Psi_1\triangleq\frac{c_\alpha \bar{b}}{2\gamma_0\overline L_{yx}^2}\zeta,\quad \Psi_2\triangleq \frac{\sqrt{c_\beta(1-(M(c_\alpha+c_\beta)+\delta))}}{\gamma_0\sqrt{M}L_{yy}}, \\
	&\zeta\triangleq -1+\sqrt{1+\frac{4(1-\delta)\gamma_0}{Mc_\alpha }\frac{\overline L_{yx}^2}{\bar{b}^2}},\quad 
	\bar{b}\triangleq \overline L_{xx}+\frac{(1-\delta)(M-1)\bar\mu}{M}.
\end{align}}%

For any given $(\bx_0,y_0)\in\dom f\times \dom h$ and arbitrary $\gamma^0>0$, $\bar{\tau}\in\left(0,\tfrac{1}{\bar{\mu}(M-1)}\right)$, {\texttt{RB-APD-B}}, stated in Algorithm \ref{alg:RBPDB}, is well-defined, i.e., the number of backtracking 
\sa{steps} is finite and bounded by {$1+\log_{1/\eta}(\frac{\bar{\tau}}{\Psi})$} uniformly for $k\geq 0$ for some $\Psi>0$. Let $\{\bx^k,y^k\}_{k\geq 0}$ denote the iterate sequence generated by {\texttt{RB-APD-B}}. 
For all $K\geq 1$, let $T^K{\triangleq}\sum_{k=0}^{K-1}t^k$ and 
\begin{align*}
\bar{\bx}^K=\frac{1}{T^K+M-1}\Big(\sum_{k=0}^{K-2}(Mt^k-(M-1)t^{k+1})\bx^{k+1}+Mt^{K-1}\bx^K\Big),\quad \bar{y}^K=\frac{1}{T^K}\break\sum_{k=0}^{K-1}t^ky^{k+1}
\end{align*}
for $\{t^k\}_{k\geq 0}$ such that $t^k=\sigma^k/\sigma^0$ for $k\geq 0$.
\\
\textbf{(Part I.)} Suppose $\underbar{$\mu$}=0$ and $\dom f\times \dom h$ is compact.
\sa{If $L_{yy}>0$, then
assume that} $M(c_\alpha+c_\beta)+\delta{\leq 1}$ {holds for some $c_\alpha,c_\beta>0$;}
\sa{otherwise, i.e., $L_{yy}=0$, assume that 
${M} c_\alpha+\delta{\leq 1}$ for some $c_\alpha>0$ and let $c_\beta=0$.}
For this setting, {$\Psi=\Psi_1$ if $L_{yy}=0$ and $\Psi=\min\{\Psi_1,\Psi_2\}$ if $L_{yy}>0$}; moreover, for all $K\geq 1$, the following bound on the expected primal-dual gap holds:
\begin{subequations}
\label{eq:general-rate-result}
\rev{\begin{align}
&\bE\left[\cG(\bar\bx^K,\bar y^K)\right]\leq \frac{1}{T^K}\left(\bar B_1+\bar B_2+\sup_{\bx\in\dom f}B_1(\bx)+\sup_{y\in\dom h}B_2(y)\right) \label{eq:rate-final}\\
&\bar B_1\triangleq \frac{(M-1)L_{\varphi_{\cX}}^2}{2}~\bE\Big[\sum_{k=0}^{+\infty}t^k\norm{\tilde\bx^{k+1}-\bx^k}^2_{\bT^k+\fM}\Big]<+\infty,\\
&\bar B_2\triangleq (M-1)\sup\{\cL(\bx^0,y)-\cL(\bx,y):\ (\bx,y)\in\dom f\times\dom h\}<+\infty,\\
&B_1(\bx)\triangleq M\bD_\cX^{(1+\frac{1}{M})\bT^0+\fM}(\bx,\bx^0),\\
&B_2(y)\triangleq \Big(\tfrac{1}{\sigma^0}+\theta^0(M-1)L_{yy}\Big)\bD_\cY(y,y^0),
\end{align}}%
\end{subequations}
\sa{with} $T^K 
=\Omega(K)$, implying $\cO(1/K)$ sublinear rate for $\bE\left[\cG(\bar\bx^K,\bar y^K)\right]$. \sa{Finally, if a saddle point for \eqref{eq:original-problem} exists and {$\delta>0$ is chosen}, then 
$\{(\bx^k,y^k)\}_{k\geq 0}$ converges to a saddle point almost surely.}\\ 
\textbf{(Part II.)} Suppose $\underbar{$\mu$}>0$ and $L_{yy}=0$; hence, $\Phi$ has the following form: $\Phi(\bx,y)=\fprod{g(\bx),y}$ for some $g:\cX\to\cY^*$ such that $\Phi$ is convex in $\bx$ on $\dom f$ for any $y\in\dom h$. Let $F(\bx)=\argmax_{y\in\cY}\cL(\bx,y)$; thus, $F(\bx)=f(\bx)+h^*(g(\bx))$ for $\bx\in\cX$.

\sa{Let $c_\alpha>0$ and $\delta\in[0,1)$ such that
${M} c_\alpha+\delta\in(0,1]$, and $c_\beta=0$.}
For this setting, $\Psi=\Psi_1$. 
If a saddle point $(\bx^*,y^*)$ for \eqref{eq:original-problem}
exists, then $\{(\bx^k,y^k)\}_{k\geq 0}$ converges to $\bx^*$ and $\{y^k\}$ has a limit point almost surely. Moreover,
\rev{
 \begin{align}
 \bE\left[\tfrac{\gamma^K}{2}\big\|{\bx^*-\bx^{K}}\big\|_{\cX}^2 +(1-M c_\alpha)\bD_{\cY}(y^*,y^{K})\right] \leq 
 &\tfrac{\gamma^0}{2}\norm{\bx^*-\bx^0}^2_{\cX}+\bD_{\cY}(y^*,y^0)\nonumber\\
 &+\sigma^0(M-1)\Big(\cL(\bx^0,y^*)-\cL({\bx}^*,y^*)\Big), \label{eq:xy-bound-II}
 \end{align}
 \begin{align}
& \bE\left[F(\bar\bx^K)-F(\bx^*)\right]\leq \frac{1}{T^K}\Big(\tfrac{\gamma^0}{2\sigma^0}\norm{\bx^*-\bx^0}^2_{\cX}+\tfrac{1}{\sigma^0}\sup_{y\in\dom h}\bD_{\cY}(y,y^0)+(M-1)\big(F(\bx^0)-F(\bx^*)\big)\Big),\label{eq:F-bound}
\end{align}}%
\rev{hold for all $K\geq 1$ \sa{with} both $\gamma^K=\Omega(K^2)$ and $T^K=\Omega(K^2)$; hence, ${\bE\big[}\norm{\bx^K-\bx^*}_{\cX}^2{\big]}=\cO(1/K^2)$ and 
$0\leq \bE\left[F(\bar\bx^K)-F(\bx^*)\right]\leq\cO(1/K^2)$.} \sa{Finally, if $\delta>0$, then any limit point of $\{(\bx^k,y^k)\}_{k\geq 0}$ is a saddle point almost surely.}
\end{theorem}

Next, we state the convergence rate result for \rev{\texttt{RB-APD}, displayed in Algorithm \ref{alg:RBPD}, in detail.}

\begin{theorem}\label{thm:main}
Suppose \sa{Assumptions~\ref{assum} and~\ref{rem:bregman} hold}, and define \rev{$\munderbar{\mu}\triangleq\min_{i\in\cM}\mu_i$} and $\bar\mu\triangleq \max_{i\in\cM}\mu_i$. \sa{Suppose $\gamma^0>0$ and $\bar{\tau}\in\left(0,\tfrac{1}{\bar{\mu}(M-1)}\right)$ are chosen such that \eqref{eq:initial_step_condition} holds for some $c_\alpha,c_\beta\geq 0$ and $\delta\in[0,1)$ as described in Theorem~\ref{thm:backtrack}.} Let
$\{\bx^k,y^k\}_{k\geq 0}$ be the iterate sequence
generated by 
{\texttt{RB-APD}},
and let $(\bar{\bx}^K,\bar{y}^K)$ and $T^K$ be defined for $K\geq 1$ as in Theorem~\ref{thm:backtrack}.

\textbf{(Part I.)} {Suppose $\underbar{$\mu$}=0$ and $\dom f\times \dom h$ is compact. The stepsize rule in Algorithm~\ref{alg:RBPD} implies that $\tilde{\tau}^k=\tilde{\tau}^0$, $\sigma^k=\sigma^0$ and $\theta^k=1$ for {all} $k\geq 0$; hence, $t^k=1$ for $k\geq 0$, implying $T^K=K$. Moreover, \eqref{eq:general-rate-result} holds for all $K\geq 1$.}
Finally, if a saddle point for \eqref{eq:original-problem} exists and
{$\delta>0$}, then 
$\{(\bx^k,y^k)\}_{k\geq 0}$ {almost surely} converges to a saddle point. 

\textbf{(Part II.)} \rev{Suppose $\underbar{$\mu$}>0$ and $L_{yy}=0$.}
If a saddle point for \eqref{eq:original-problem}
exists, then $\{\bx^k\}_{k\geq 0}$ converges to $\bx^*$ and $\{y^k\}$ has a limit point almost surely,\footnote{Since $\underbar{$\mu$}>0$, $\bx^*$ must be the unique $\bx$-coordinate of any saddle point.} \rev{where $\bx^*$ denotes the unique solution to $\min\{F(\bx):\ \bx\in\cX\}$.} 
\sa{Furthermore, both \eqref{eq:xy-bound-II} and \eqref{eq:F-bound} hold for all $K\geq 1$ such that} $\gamma^K=\Omega(K^2)$ and $T^K=\Omega(K^2)$; hence, ${\bE\big[}\norm{\bx^K-\bx^*}_{\cX}^2{\big]}=\cO(1/K^2)$ and 
$0\leq \bE\left[F(\bar\bx^K)-F(\bx^*)\right]\leq\cO(1/K^2)$. \sa{Finally, if \eqref{eq:initial_step_condition} holds with $\delta>0$,} then any limit point \sa{of $\{(\bx^k,y^k)\}_{k\geq 0}$} 
 is a saddle point almost surely.
\end{theorem}
\subsection{One-step analysis}\label{sec:proof_onestep}
To prove the results 
\rev{of Theorems~\ref{thm:backtrack} and~\ref{thm:main},} we first provide a one-step analysis in Lemma \ref{lem:one-step} to provide a bound on the progress of iterates in terms of the 
\rev{coupling} function $\cL$. This is the main building block of our convergence analysis. 

\begin{lemma}\label{lem:one-step}
Suppose Assumption~\ref{assum} holds, \rev{let
$\{\bx^k,y^k\}_{k\geq 0}$ be 
generated by the following recursion:
\begin{subequations}
\label{eq:generic_recursion}
\begin{align}
    &q^k\gets {M}(\grad_y \Phi(\bx^{k},{y^{k}})-\grad_y\Phi(\bx^{k-1},{y^{k-1}}))\label{eq:qk}\\
    &s^k\gets \grad_y \Phi(\bx^{k},{y^k})+\theta^{k}{q^k}\label{eq:sk}\\
    &y^{k+1}\gets \argmin_y h(y){-\fprod{s^k,y}}+\frac{1}{\sigma^k}\bD_\cY(y,y^k) \label{eq:y-subproblem-generic}\\
    &\tilde x_{i}^{k+1}\gets \argmin_x f_{i}(x)+\fprod{\grad_{x_{i}}\Phi(\bx^k,y^{k+1}),x} +{\frac{1}{\tau_{i}^k}\bD_{\cX_{i}}(x,x_{i}^k)},\quad \forall~i\in\cM \label{eq:tilde-problem-B}\\
    &\bx^{k+1}\gets \bx^k\nonumber\\
    &\mbox{Choose $i_k\in\cM$ uniformly at random}\nonumber\\
    &x_{i_k}^{k+1}\gets \tilde x_{i_k}^{k+1},\nonumber
\end{align}
\end{subequations}
for some 
positive parameters $\{\tau_i^k\}_{i\in\cM}$, $\sigma^k$ and $\theta^k$ for $k\geq 0$.\footnote{\texttt{RB-APD} and \texttt{RB-APD-B}, stated in Algorithm~\ref{alg:RBPD} and Algorithm~\ref{alg:RBPDB}, respectively, both satisfy this recursion for some primal-dual stepsize sequences.}} 
Then for \rev{any} 
\rev{$(\bx,y)\in\dom f\times\dom h$,}
\begin{align}
&\cL(\bx^{k+1},y)-\cL(\bx,y^{k+1})\leq Q^k(\bz) - R^{k+1}(\bz) + (M-1)(\theta^k H^k(\bz)-H^{k+1}(\bz)), \nonumber\\
& \quad + (M-1)\Big((1-\theta^k)\big(\cL(\bx^k,y)-\cL(\bx,y)\big)+\rev{(1-\theta^k)_+}L_{yy}\bD_{\cY}(y,y^k)\Big)+
C^k(\bx^{k+1},y^{k+1})+\cE^k(\bx), \label{eq:one-step-main}
\end{align}
holds for all $k\geq 0$ \rev{for any $\{\alpha^k\}\subset\reals_{++}$, and for any $\{\beta^k\}\subset\reals_{++}$ when $L_{yy}>0$ (and $\beta^k=0$ for all $k\geq 0$ when $L_{yy}=0$, defining $0^2/0=0$),} where \rev{$(\cdot)_+=\max\{\cdot,0\}$,} $Q^k(\cdot)$, $R^{k+1}(\cdot)$, $H^{k}(\cdot)$, and $\cE^k(\bx)$ are defined as follows:
{\small
\begin{subequations}\label{eq:lagrangian-B}
\begin{align}
& Q^k(\bz) \triangleq M\bD_\cX^{\bT^k}(\bx,\bx^k)+\frac{M-1}{2}\norm{\bx-\bx^k}_{\fM}^2+\frac{1}{\sigma^k}\bD_{\cY}(y,y^k)+\theta^k\fprod{r^k,y^k-y}\nonumber\\
& \quad \rev{+\frac{M\theta^k}{2\alpha^k}\norm{
\grad_y\Phi(\bx^k,y^{k})-\grad_y\Phi(\bx^{k-1},y^{k})}_{\cY^*}^2+\frac{M\theta^k}{2\beta^k}\norm{
\grad_y\Phi(\bx^{k-1},y^{k})-\grad_y\Phi(\bx^{k-1},y^{k-1})}^2_{\cY^*},} \label{eq:lagrangian-Q-B}\\
& R^{k+1}(\bz)\triangleq M\bD_\cX^{\bT^k}(\bx,\bx^{k+1})+\frac{M}{2}\norm{\bx-\bx^{k+1}}_{\fM}^2 +\frac{1}{\sigma^{k}}\bD_{\cY}(y,y^{k+1}) +\fprod{r^{k+1},y^{k+1}-y}  \nonumber \\
& \quad \rev{+\frac{M}{2\alpha^{k+1}}\norm{
\grad_y\Phi(\bx^{k+1},y^{k+1})-\grad_y\Phi(\bx^{k},y^{k+1})}_{\cY^*}^2+\frac{M}{2\beta^{k+1}}\norm{
\grad_y\Phi(\bx^{k},y^{k+1})-\grad_y\Phi(\bx^{k},y^{k})}^2_{\cY^*},} \\
& H^k(\bz)\triangleq  f(\bx^k)-f(\bx)+\Phi(\bx^k,y^{k})-\Phi(\bx,y), \label{eq:largrangian-H}\\
& \cE^k(\bx)\triangleq Mf(\bx^{k+1})-f(\btx^{k+1})-(M-1)f(\bx^k)-\fprod{\grad_\bx\Phi(\bx^k,y^{k+1}),\btx^{k+1}-M\bx^{k+1}+(M-1)\bx^k}\nonumber \\
&\quad +M\bD_\cX^{\bT^k}(\bx,\bx^{k+1})-(M-1)\bD_\cX^{\bT^k}(\bx,\bx^k)-\bD_\cX^{\bT^k}(\bx,\btx^{k+1})+M\bD_\cX^{\bT^k}(\bx^{k+1},\bx^k)\nonumber \\
&\quad -\bD_\cX^{\bT^k}(\btx^{k+1},\bx^k)+\frac{M}{2}\norm{\bx-\bx^{k+1}}^2_\fM-\frac{1}{2}\norm{\bx-\btx^{k+1}}^2_\fM-\frac{M-1}{2}\norm{\bx-\bx^k}^2_\fM,\label{eq:Ek}
\end{align}
\end{subequations}}%
where $C^k(\cdot,\cdot)$ 
\rev{defined} in \eqref{eq:test-function-B} and $r^k\triangleq \grad_y\Phi(\bx^k,y^k)-M\grad_y\Phi(\bx^{k-1},y^{k-1})$.  
\end{lemma}
\begin{proof}
\rev{We define an auxiliary sequence {$\{\btx^k\}_{k\geq 1}\subseteq\cX$} such that $\tilde x_i^{k+1}$ is defined as in~\eqref{eq:tilde-problem-B} for all $k\geq 0$.}
The auxiliary sequence {$\{\btx^k\}_{k\geq 1}$} is never actually computed in the implementation of 
\rev{\texttt{RB-APD} or \texttt{RB-APD-B},} and it is defined
for analyzing the convergence behavior of {$\{\bx^k\}_{k\geq 1}\subseteq\cX$}. 
For $k\geq0$, we apply Lemma~\ref{lem_app:prox} for \rev{both 
$\tilde{x}_i$-subproblem in \eqref{eq:tilde-problem-B} and the $y$-subproblem in \eqref{eq:y-subproblem-generic} (see Line~9 of \texttt{RB-APD} and also Line~11 of \texttt{RB-APD-B})},
we obtain two inequalities holding for any $y\in\cY$ and $\bx\in\cX$:
\begin{align}
&h(y^{k+1})-\fprod{s^k,~ y^{k+1}-y} \leq h(y)
+\frac{1}{\sigma^k}\Big[\bD_{\cY}(y,y^k)-{\bD_{\cY}(y,y^{k+1})}-\bD_{\cY}(y^{k+1},y^k)\Big], \label{eq:sc-h-B}\\
&f_i(\tilde{x}_i^{k+1})+\fprod{\grad_{x_i}\Phi(\bx^k,y^{k+1}),\tilde{x}_i^{k+1}-x_i}+\frac{\mu_i}{2} \rev{\norm{x_i-\tilde{x}_i^{k+1}}_{\cX_i}^2}\nonumber\\
&\quad \leq f_i(x)
+\frac{1}{\tau_i^k}\Big[\bD_{\cX_i}(x_i,x_i^k)-\bD_{\cX_i}(x_i,\tilde{x}_i^{k+1})-\bD_{\cX_i}(\tilde{x}_i^{k+1},x_i^k)\Big],\quad \forall i\in\cM, \label{eq:sc-f-B}
\end{align}
where 
\rev{$q^k$ and $s^k$ are defined in~\eqref{eq:qk} and~\eqref{eq:sk}, respectively.}
Note that by invoking \eqref{eq:Lxx_bound}, we may bound the inner product in \eqref{eq:sc-f-B} as follows:
\begin{align}\label{eq:f-prod-B}
\fprod{\grad_{\bx}\Phi(\bx^k,y^{k+1}),\tilde{\bx}^{k+1}-\bx} &= \fprod{\grad_{\bx}\Phi(\bx^k,y^{k+1}),{\bx}^k-\bx}+\fprod{\grad_{\bx}\Phi(\bx^k,y^{k+1}),\tilde{\bx}^{k+1}-\bx^k}\nonumber  \\ 
&\geq\Phi(\bx^k,y^{k+1})- \Phi(\bx,y^{k+1}) +\fprod{\grad_{\bx}\Phi(\bx^k,y^{k+1}),\tilde{\bx}^{k+1}-\bx^k}.
\end{align}
Next, we define two auxiliary sequences for $k\geq 0$:
\begin{subequations}
\begin{align}
\label{eq:Ak-B}
A^{k+1}&\triangleq \frac{1}{\sigma^k}\bD_{\cY}(y,y^k)-\frac{1}{\sigma^k}{\bD_{\cY}(y,y^{k+1})} -\frac{1}{\sigma^k}\bD_{\cY}(y^{k+1},y^k),\\
\label{eq:Bk-B}
{B}^{k+1}&\triangleq \bD_{\cX}^{\bT^k}(\bx,\bx^k)-\bD_{\cX}^{\bT^k}(\bx,\btx^{k+1})-\tfrac{1}{2}\norm{\bx-\btx^{k+1}}_{\fM}^2-\bD_\cX^{\bT^k}(\btx^{k+1},\bx^k).
\end{align}
\end{subequations}
Summing \eqref{eq:sc-f-B} over $i\in\cM$, and using \eqref{eq:f-prod-B} and \eqref{eq:Bk-B} leads to
\begin{align}\label{eq:f-prod-tilde}
f(\btx^{k+1})\leq f(\bx)+\Phi(\bx,y^{k+1})-\Phi(\bx^k,y^{k+1})+{B}^{k+1} -\fprod{\grad_{\bx}\Phi(\bx^k,y^{k+1}),\tilde{\bx}^{k+1}-\bx^k}.
\end{align}
For $k\geq 0$, let $$\Lambda_x^k\triangleq \Phi(\bx^{k+1},y^{k+1})-\Phi(\bx^k,y^{k+1})-\fprod{\grad_{\bx}\Phi(\bx^k,y^{k+1}),\bx^{k+1}-\bx^k},$$ $$E^k\triangleq f(\bx^{k+1})-\frac{1}{M}f(\btx^{k+1})-(1-\frac{1}{M})f(\bx^k)-\frac{1}{M}\fprod{\grad_\bx\Phi(\bx^k,y^{k+1}),\btx^{k+1}-M\bx^{k+1}+(M-1)\bx^k},$$ 
then dividing both sides of \eqref{eq:f-prod-tilde} by $M$ and rearranging the terms lead to
\begin{align}\label{eq:expect-f-B}
&f(\bx^{k+1})+\Phi(\bx^{k+1},y^{k+1})-f(\bx)-\frac{1}{M}\Phi(\bx,y^{k+1})\leq  \\
& \quad \left(1-\frac{1}{M}\right)\left(f(\bx^k)-f(\bx)+\Phi(\bx^k,y^{k+1})\right)+\frac{1}{M}{B}^{k+1}+\Lambda_x^k+E^k\rev{.}\nonumber
\end{align}

Next, multiplying \eqref{eq:sc-h-B} by $\frac{1}{M}$ and adding to \eqref{eq:expect-f-B}, then adding 
\rev{$\Phi(\bx^{k+1},y)/M$} to both sides, and rearranging the terms we obtain:
{
\begin{align}
\label{eq:sum-f-h-B}
 &\frac{1}{M}\big(f(\bx^{k+1}) -f(\bx) +\Phi(\bx^{k+1},y)\big) +\frac{1}{M}\big(h(y^{k+1})-h(y) -\Phi(\bx,y^{k+1})\big) \nonumber \\
&  \leq \frac{1}{M}\left(\Phi(\bx^{k+1},y)-\Phi(\bx^{k+1},y^{k+1})\right)+ \left(1-\frac{1}{M}\right)\big(f(\bx^k)-f(\bx)+\Phi(\bx^k,y^{k+1})\big) \nonumber \\
& \quad -\left(1-\frac{1}{M}\right)\big(f(\bx^{k+1}) -f(\bx) +\Phi(\bx^{k+1},y^{k+1})\big)  +\frac{1}{M}{B}^{k+1}+\frac{1}{M}A^{k+1}\nonumber\\
& \quad +\frac{1}{M}\fprod{s^k,y^{k+1}-y} +\Lambda_x^k+E^k\nonumber \\
& \leq \frac{1}{M}\fprod{\grad_y\Phi(\bx^{k+1},y^{k+1}),y-y^{k+1}}+ \rev{\underbrace{\left(1-\frac{1}{M}\right)\Big(f(\bx^k)-f(\bx)+\Phi(\bx^k,y^{k+1})-\Phi(\bx,y)\Big)}_{(*)}}
\nonumber\\
& \quad -\left(1-\frac{1}{M}\right)\rev{\Big(f(\bx^{k+1}) -f(\bx) +\Phi(\bx^{k+1},y^{k+1})-\Phi(\bx,y)\Big)} +\frac{1}{M}{B}^{k+1}+\frac{1}{M}A^{k+1}\nonumber \\
& \quad  +\frac{1}{M}\fprod{s^k,y^{k+1}-y}+\Lambda_x^k+E^k,
\end{align}}%
where in the last inequality, we use the concavity of $\Phi(\bx^{k+1},\cdot)$. 
Note that for any real number $a\in\reals$, from \eqref{eq:Lyy_bound} we have that for any $\bx\in\cX$, $\bar y,y\in\cY$
\begin{align}\label{eq:lip-conv-plus}
    \eh{a\Phi(\bx,\bar y)\leq a\Phi(\bx,y)+a\fprod{\grad_y \Phi(\bx,\bar y),\bar y-y}+\max\{a, 0\}\cdot\frac{L_{yy}}{2}\norm{y-\bar y}_{\cY}^2.}
\end{align}
Next, we provide an upper bound for 
$(*)$ in \eqref{eq:sum-f-h-B}  
{as follows:}
\begin{align}\label{eq:(a)-B}
(*)&{\leq \Big(1-\frac{1}{M}\Big)\Big(f(\bx^k)-f(\bx)+\Phi(\bx^k,y^{k})-\Phi(\bx,y)+\fprod{\grad_y\Phi(\bx^k,y^k),~y^{k+1}-y^k}\Big)} \nonumber \\
&\leq \Big(1-\frac{1}{M}\Big)\theta^kH^k(\bz)+\Big(1-\frac{1}{M}\Big)(1-\theta^k)\Big(f(\bx^k)-f(\bx)+\Phi(\bx^k,y)-\Phi(\bx,y)\Big) \nonumber \\
&\quad \eh{+\Big(1-\frac{1}{M}\Big)\frac{L_{yy}(1-\theta^k)_+}{2}\norm{y-y^{k}}_{\cY}^2}+ \Big(1-\frac{1}{M}\Big)\theta^k\fprod{\grad_y\Phi(\bx^k,y^k),y-y^k}\nonumber\\
&\quad +\Big(1-\frac{1}{M}\Big)\fprod{\grad_y\Phi(\bx^k,y^k),y^{k+1}-y},
\end{align}
where \rev{$(a)_+=\max\{a,0\}$ and} $H^k(\bz)= f(\bx^k)-f(\bx)+\Phi(\bx^k,y^{k})-\Phi(\bx,y)$; in the first inequality above we used the concavity \rev{of $\Phi(\bx^k,\cdot)$, and in the second inequality, we split the resulting bound into $\theta^k\geq 0$ and $1-\theta^k$ fractions, and used \eqref{eq:lip-conv-plus} for $a=1-\theta^k$, $\bx=\bx^k$ and $\bar y= y^k$.}

Now before combining \eqref{eq:(a)-B} with \eqref{eq:sum-f-h-B}, we first simplify the summation of all inner product terms coming from both inequalities. 
Recall that $s^k=\grad_y\Phi(\bx^k,y^k)+\theta^kq^k$ where $q^k = M(\grad_y\Phi(\bx^k,y^{k})-\grad_y\Phi(\bx^{k-1},y^{k-1}))$ {and} {we define} $r^k{\triangleq} q^k-(M-1)\grad_y\Phi(\bx^k,y^{k})$ {for all $k\geq 0$. Using these definitions, we can rearrange the sum of all inner products in \eqref{eq:(a)-B} and \eqref{eq:sum-f-h-B} as follows:}
\begin{align}\label{eq:inner-prod-B}
&\frac{1}{M}\fprod{\grad_y\Phi(\bx^{k+1},y^{k+1}),y-y^{k+1}}+\Big(1-\frac{1}{M}\Big)\theta^k\fprod{\grad_y\Phi(\bx^k,y^k),y-y^k} \nonumber \\
& \quad +\Big(1-\frac{1}{M}\Big)\fprod{\grad_y\Phi(\bx^k,y^k),y^{k+1}-y}+\frac{1}{M}\fprod{s^k,y^{k+1}-y} \nonumber \\
&=\fprod{\frac{1}{M}\grad_y\Phi(\bx^{k+1},y^{k+1})-\grad_y\Phi(\bx^k,y^k),y-y^{k+1}}\nonumber \\
& \quad +(1-\frac{1}{M})\theta^k\fprod{\grad_y\Phi(\bx^k,y^k),y-y^k} +\frac{\theta^k}{M}\fprod{ q^k,y^k-y}+\frac{\theta^k}{M}\fprod{q^k,y^{k+1}-y^k}\nonumber \\
&= -\frac{1}{M}\fprod{r^{k+1},y^{k+1}-y}+\frac{\theta^k}{M}\fprod{r^k,y^k-y}+ \frac{\theta^k}{M}\fprod{q^k,y^{k+1}-y^k}.
\end{align}
Hence, using \eqref{eq:(a)-B} and \eqref{eq:inner-prod-B} within \eqref{eq:sum-f-h-B}, and the definition of $\cL(\cdot,\cdot)$ we obtain the following inequality:
\begin{align}\label{eq:one-iter-bound-B}
&\frac{1}{M}\left(\cL(\bx^{k+1},y)-\cL(\bx,y^{k+1})\right) \leq \nonumber \\
&\quad  \frac{1}{M}\Big({B}^{k+1}+A^{k+1}-\fprod{r^{k+1},y^{k+1}-y}+\theta^k\fprod{r^k,y^k-y}+\theta^k\underbrace{\fprod{q^k,y^{k+1}-y^k}}_{(**)}\Big) \nonumber \\
& \quad +\Big(1-\frac{1}{M}\Big)\Big((1-\theta^k)\big(\cL(\bx^k,y)-\cL(\bx,y)\big)+\eh{(1-\theta^k)_+}L_{yy}\bD_\cY(y,y^k)
\Big) \nonumber\\
& \quad +\Big(1-\frac{1}{M}\Big)\Big(\theta^kH^k(\bz)-H^{k+1}(\bz)\Big)+\Lambda_x^k+E^k,
\end{align}
where we also used the fact that $\bD_\cY(y,\bar{y})\geq \frac{1}{2}\norm{y-\bar{y}}^2$.

In order to provide a bound for term $(**)$, we provide a general bound for $\fprod{q^k,y-y^k}$ for any $y\in\cY$ as follows. Let $p_x^k\triangleq \grad_y\Phi(\bx^k,y^{k})-\grad_y\Phi(\bx^{k-1},y^{k})$ and $p_y^k\triangleq \grad_y\Phi(\bx^{k-1},y^{k})-\grad_y\Phi(\bx^{k-1},y^{k-1})$. Such definitions immediately imply that $q^k=M(p_x^k+p_y^k)$, for all $k\geq 0$.  Hence,
using Young's inequality twice, once for $\fprod{p_x^k,y-y^k}$ and once for $\fprod{p_y^k,y-y^k}$ and the fact that $\rev{\bD_{\cY}(y,\bar{y})}\geq \frac{1}{2}\norm{y-\bar{y}}_\cY^2$, for any $y,\bar{y}\in\cY$, we obtain that for all $k\geq 0$,
\begin{equation}\label{eq:inner-q-B}
    \abs{\fprod{q^k,y-y^k}}\leq M(\alpha^k+\beta^k)\bD_\cY(y,y^k)+\frac{M}{2\alpha^k}\norm{p_x^k}^2_{\cY^*}+\frac{M}{2\beta^k}\norm{p_y^k}^2_{\cY^*},
\end{equation}
for any $\alpha^k,\beta^k>0$. Moreover, if $L_{yy}=0$, then $\norm{p_y^k}_{\cY^*}=0$; hence, $\abs{\fprod{q^k,y-y^k}}\leq \alpha^k\bD_\cY(y,y^k)+\frac{1}{2\alpha^k}\norm{p_x^k}^2_{\cY^*}$, for any $\alpha^k>0$. \rev{Therefore, first using \eqref{eq:inner-q-B} within \eqref{eq:one-iter-bound-B} for some $\alpha^k,\beta^k>0$ (with the exception of $\beta^k=0$ for when $L_{yy}=0$), then
multiplying both sides of the resulting inequality by $M$; and finally, adding and subtracting $M\bD_\cX^{\bT^k}(\bx,\bx^{k+1})-(M-1)\bD_\cX^{\bT^k}(\bx,\bx^k)-\bD_\cX^{\bT^k}(\bx,\btx^{k+1})+M\bD_\cX^{\bT^k}(\bx^{k+1},\bx^k)
-\bD_\cX^{\bT^k}(\btx^{k+1},\bx^k)$ and $\frac{M}{2}\norm{\bx-\bx^{k+1}}^2_\fM-\frac{1}{2}\norm{\bx-\btx^{k+1}}^2_\fM-\frac{M-1}{2}\norm{\bx-\bx^k}^2_\fM$ to the right-hand side, and rearranging the terms yield the desired result in \eqref{eq:one-step-main}.}
\end{proof}

\subsection{Backtracking Step-size Analysis}\label{sec:proof-step}
\begin{lemma}
\label{lem:stronger-step-condition-B}
{Suppose the sequence $\{[\tau_i^k]_{i\in\cM},\sigma^k,\theta^k\}_{k\geq 0}$ satisfy 
\eh{\eqref{eq:step-size-condition-pd-new1} and \eqref{eq:step-size-condition-pd-new}} for some positive \rev{$\{
\alpha^k\}_{k\geq 0}$}, nonnegative $\{\beta^k\}_{k\geq 0}$, and $\delta\in[0,1)$. \rev{Let $\{\bx^k,y^k\}$ be 
generated according to the recursion in~\eqref{eq:generic_recursion} 
using the given parameter sequence $\{[\tau_i^k]_{i\in\cM},\sigma^k,\theta^k\}$.} Then $\{\bx^k,y^k\}$ and $\{[\tau_i^k]_{i\in\cM},\sigma^k,\theta^k\}$ satisfy 
\eqref{eq:step-size-Ck-B} with the same \rev{$\{
\alpha^k,\beta^k\}$} and $\delta$.}
\end{lemma}
\begin{proof}
The proof is similar to Lemma 3.4. in \cite{hamedani2021primal}.
\end{proof}

\begin{lemma}\label{lem:stepsize-sequence}
\rev{Given arbitrary $\{\tilde{\tau}^k\}_{k\geq 0}\subset \reals_{++}$, and $\bar{\tau},\gamma^0>0$, let $\sigma^{-1}=\gamma^0\bar{\tau}$, and for $k\geq 0$, let $\sigma^k=\gamma^k\tilde{\tau}^k$, $\theta^k=\sigma^{k-1}/\sigma^k$, and 
$\gamma^{k+1}=\gamma^k(1+\underbar{$\mu$}\tilde{\tau}^k)$. 
Moreover, for $i\in\cM$ and $k\geq 0$, let $\tau_{i}^k=\big(\frac{1}{M}(\mu_{i}+\frac{1}{\tilde{\tau}^k})-\mu_{i}\big)^{-1}$. Then,
$\{[\tau_i^k]_{i\in\cM},\sigma^k,\theta^k\}$ satisfies \eqref{eq:step-size-tau-B} and \eqref{eq:step-size-theta-B} with $t^k=\sigma^k/\sigma^0$.}
\end{lemma}
\begin{proof}
\rev{Since $t^k=\sigma^k/\sigma^0$ for $k\geq 0$, we get $$\frac{t^k}{\sigma^k}=\frac{1}{\sigma^0}=\frac{t^{k+1}}{\sigma^{k+1}},\qquad t^{k+1}\theta^k=\frac{\sigma^{k+1}}{\sigma^0}\frac{\sigma^k}{\sigma^{k+1}}=\frac{\sigma^k}{\sigma^0}=t^k;$$ thus, the first condition in \eqref{eq:step-size-theta-B} holds with equality and the second condition is also satisfied for the given choice of parameters.}

\rev{Furthermore, since $t^k=\sigma^k/\sigma^0$ and $\mu_i+\frac{1}{\tau_{i}^k}=\frac{1}{M}(\mu_{i}+\frac{1}{\tilde{\tau}^k})$ for all $i\in\cM$, using these choice of parameters \eqref{eq:step-size-tau-B} can be equivalently written as follows:
\begin{align}
\label{eq:eq:step-size-tau-B-alternative}
\frac{1}{M}\Big(\frac{1}{\tilde\tau^k}+\mu_i\Big)\geq\frac{\sigma^{k+1}}{\sigma^k}\Big(\frac{1}{M}\Big(\frac{1}{\tilde\tau^{k+1}}+\mu_i\Big)-\frac{\mu_i}{M}\Big),\quad \forall~i\in\cM,
\end{align}
which is equivalent to $\frac{1}{\tilde\tau^k}+\mu_i\geq \frac{\sigma^{k+1}}{\sigma^k}\frac{1}{\tilde\tau^{k+1}}=\frac{\gamma^{k+1}\tilde\tau^{k+1}}{\gamma^{k}\tilde\tau^{k}}\frac{1}{\tilde\tau^{k+1}}=\frac{1}{\tilde\tau^{k}}(1+\underbar{$\mu$}\tilde\tau^k)=\frac{1}{\tilde\tau^k}+\underbar{$\mu$}$, which trivially holds for all $i\in\cM$ as $\mu_i\geq\underbar{$\mu$}$ for $i\in\cM$, where the first equality above follows from $\sigma^k=\gamma^k\tilde\tau^k$, the second equality uses $\gamma^{k+1}=\gamma^k(1+\underbar{$\mu$}\tilde\tau^k)$. This completes the proof.}
\end{proof}

\begin{lemma}\label{lem:parameter}
\eh{\sa{Consider $\{\tilde\tau^k\}_{k\geq 0}$ generated by \texttt{RB-APD-B} displayed in Algorithm~\ref{alg:RBPDB} for some $\delta\in[0,1)$ and $c_\alpha,c_\beta\geq 0$ such that $M(c_\alpha+c_\beta)+\delta\leq 1$. When $L_{yy}>0$, set $c_\alpha,~c_\beta> 0$; otherwise, when $L_{yy}=0$, set $c_\alpha>0$ and $c_\beta=0$. There exists a positive sequence $\{\hat\tau^k\}_{k\geq 0}$ such that $\tilde{\tau}^k\geq\eta\hat{\tau}^k$ for all $k\geq 0$. Furthermore,} when $L_{yy}>0$ and $\underbar{$\mu$}= 0$,
$\hat{\tau}^k\geq \min\{\Psi_1,\Psi_2\}$ for $k\geq 0$; when $L_{yy}=0$ and $\underbar{$\mu$}= 0$, $\hat{\tau}^k\geq \Psi_1$ for $k\geq 0$; and when $L_{yy}=0$ and $\underbar{$\mu$}> 0$, $\hat{\tau}^k\geq \Psi_1\sqrt{\gamma^0/\gamma^k}$ for $k\geq 0$, where $\Psi_1$ and $\Psi_2$ are defined in \eqref{eq:tauhat_bound-1,2}.}\footnote{$\underbar{$\mu$}= 0$ implies $\gamma^k=\gamma^0$ for $k\geq 0$, while $\underbar{$\mu$}> 0$ implies $\gamma^{k+1}>\gamma^k$ for $k\geq 0$.}
\end{lemma}
\begin{proof}
\eh{Let us fix arbitrary $k\geq 0$ and $i_k\in\cM$. Lemma~\ref{lem:stronger-step-condition-B} implies that if \eh{\eqref{eq:step-size-condition-pd-new1} and \eqref{eq:step-size-condition-pd-new} hold} then \eqref{eq:step-size-Ck-B} holds as well. First, to show that the backtracking condition is satisfied in a finite number of steps, we will show that there exists $\hat{\tau}^k>0$ such that \eh{\eqref{eq:step-size-condition-pd-new1} and \eqref{eq:step-size-condition-pd-new}} are true for all $\tilde{\tau}^k\in(0,\hat{\tau}^k]$. Using  $\sigma^k=\gamma^k\tilde{\tau}^k$,  $\theta^k=\sigma^{k-1}/\sigma^k$, $\overline L_{xx}=\max_{i\in\cM}\{L_{x_ix_i}\}$, and $\overline L_{yx}=\max_{i\in\cM}\{L_{yx_i}\}$ inequalities in \eh{\eqref{eq:step-size-condition-pd-new1} and \eqref{eq:step-size-condition-pd-new}} hold if}
\begin{align}
\label{eq:pd-equiv}
0\geq -(1-\delta)+\overline L_{xx}\tau^k_{i_k}+\frac{\overline L_{yx}^2}{c_\alpha}\gamma^k\tilde{\tau}^k\tau^k_{i_k},\quad 1-(\delta+M(c_\alpha+c_\beta))\geq\frac{M L_{yy}^2}{c_\beta}(\gamma^k\tilde{\tau}^k)^2.
\end{align}
\rev{Recall that $\bar\mu=\max_{i\in\cM}\mu_i$.} \eh{Suppose $L_{yy}>0$, then \rev{$\tau_{i}^k=\big(\frac{1}{M}(\mu_{i}+\frac{1}{\tilde{\tau}^k})-\mu_{i}\big)^{-1}$ for all $i\in\cM$ implies that} $\tau_{i_k}^k\leq M(1/\tilde\tau^k-(M-1)\bar\mu)^{-1}=M\tilde\tau^k/(1-(M-1)\bar\mu\tilde\tau^k)$ implies that \eqref{eq:pd-equiv} holds for all $\tilde{\tau}^k\in(0,\hat{\tau}^k]$, where
{\small
\begin{align}
\label{eq:tauhat}
&\hat{\tau}^k\triangleq\min\left\{\frac{-\bar{b}+\sqrt{\bar{b}^2+4(1-\delta)\overline L_{yx}^2\gamma^k/(\eh{M}c_\alpha)}}{2\overline L^2_{yx}\gamma^k/c_\alpha},\quad \frac{\sqrt{c_\beta(1-(M(c_\alpha+c_\beta)+\delta))}}{\gamma^k \sqrt{M}L_{yy}}\right\},\\
&\bar{b}\triangleq \overline L_{xx}+\frac{(1-\delta)(M-1)\bar\mu}{M}.
\end{align}}}%
Note that when $L_{yy}=0$, the second inequality in~\eqref{eq:pd-equiv} always holds \rev{due to our choice of $\delta\in[0,1)$ and $c_\alpha,c_\beta\geq 0$ satisfying $M(c_\alpha+c_\beta)+\delta\leq 1$}; hence, $\hat{\tau}_k$ is defined by the first term in \eqref{eq:tauhat}, \rev{i.e., treating $1/0$ in the second term as $+\infty$}. Since in each step of backtracking, $\tilde\tau^k$ is decreased by a factor of $\eta\in(0,1)$, when the backtracking terminates, $\tilde\tau^k\geq \eta\hat{\tau}^k$. \eh{Next, we provide a lower bound on $\hat\tau^k$ by considering the following two cases: (Case I) $\underbar{$\mu$}>0$; and (Case II) $\underbar{$\mu$}=0$. In particular, we will also use the following useful inequality: for any $a\geq 0$ and $b,c>0$, we have $\sqrt{a^2+cb^2}\geq a+\sqrt{c}bd$ where $d\triangleq -\frac{a}{b\sqrt{\bar c}}+\sqrt{\frac{a^2}{b^2\bar c}+1}$ holds for any $\bar c\in(0,c]$.}

\eh{For (Case I), from the assumption we know that $L_{yy}=0$; therefore, $\hat\tau^k=\frac{-\bar{b}+\sqrt{\bar{b}^2+4(1-\delta)\overline L_{yx}^2\gamma^k/(\eh{M}c_\alpha)}}{2\overline L^2_{yx}\gamma^k/c_\alpha}$. Using the fact that $\gamma_{k+1}\geq \gamma_k\geq \gamma_0>0$ for all $k\geq 0$, and the above useful inequality for $a=\bar{b}$, $b=2\sqrt{\frac{1-\delta}{Mc_\alpha}}\overline{L}_{yx}$, $c=\gamma_k$ and $\bar{c}=\gamma_0$ we conclude that $\hat\tau^k\geq \Psi_1\sqrt{\gamma_0/\gamma_k}$.}

For (Case II), when $\underbar{$\mu$}=0$, $\gamma^k=\gamma^0$ for $k\geq 0$. Hence, from \eqref{eq:tauhat}, we have $\hat{\tau}^k=\hat{\tau}^0$ for $k\geq 0$; thus, when $L_{yy}=0$, we get $\hat{\tau}^0\geq \Psi_1$, and when $L_{yy}>0$, we get $\hat{\tau}_0\geq \min\{\Psi_1,\Psi_2\}$.
\end{proof}

\begin{lemma}
\label{lem:k2-rate-B}
\eh{Suppose $\underbar{$\mu$}>0$, and $L_{yy}=0$. Stepsize sequences generated by both \texttt{RB-APD} and \texttt{RB-APD-B}, displayed in Algorithms~\ref{alg:RBPD} and~\ref{alg:RBPDB}, respectively, satisfy $\sigma^k=\Omega(k)$, $\tilde{\tau}^k=\Omega(1/\sigma^k)$, and $\tilde{\tau}^k/\sigma^k=\cO(1/k^2)$ for $k\geq 0$. Indeed, $\sigma^k \geq \frac{\Gamma^2}{3\underbar{$\mu$}} k$, $\tilde{\tau}^k\sigma^k\geq \Gamma^2/\rev{\underbar{$\mu$}^2}$ and $(\gamma^k)^{-1}=\tilde{\tau}^k/\sigma^k\leq 9/(\Gamma^2 k^2)$ for $k\geq 0$, where $\Gamma=\underbar{$\mu$}\tilde{\tau}_0\sqrt{\gamma_0}$ for \texttt{RB-APD} and 
{$\Gamma=\underbar{$\mu$}\eta\Psi_1\sqrt{\gamma_0}$} for \texttt{RB-APD-B} with $\Psi_1$ as defined in~\eqref{eq:tauhat_bound-1,2}. Furthermore, for all $\epsilon>0$, $\sigma^k\geq \frac{\Gamma^2}{(2+\epsilon)\underbar{$\mu$}}k$ and $\tilde{\tau}^k/\sigma^k\leq \rev{(2+\epsilon)^2}/(\Gamma^2 k^2)$ for $k\geq \lceil\frac{1}{\epsilon}\rceil$.}
\end{lemma}
\begin{proof}
The proof follows directly from Lemma 3.7. in \cite{hamedani2021primal}.
\end{proof}

\eh{Next, we prove Lemma \ref{assum:step-size-B}.}

\eh{{\bf Proof of Lemma \ref{assum:step-size-B}.} Indeed, Lemma~\ref{lem:stepsize-sequence} implies that $\{[\tau^k_i]_{i\in\cM},\sigma^k,\theta^k\}$ generated by \texttt{RB-APD-B} satisfies \eqref{eq:step-size-tau-B} and \eqref{eq:step-size-theta-B} for $\{t^k\}$ such that $t^k=\sigma^k/\sigma^0$ for $k\geq 0$. Moreover,  Lemma~\ref{lem:parameter} shows that for any $k\geq 0$, the backtracking condition in Algorithm~\ref{alg:RBPDB} holds after a finite number of inner iterations. Thus, the results of Lemma~\ref{assum:step-size-B} clearly hold.}

\rev{Before proceeding to the proof of our main results, we would like to remind the reader \eh{\sa{Assumption}~\ref{rem:bregman} stating our} assumptions on the Bregman distance generating function $\varphi_{\cX_i}(\cdot)$ for $i\in\cM$.}

\subsection{Asymptotic Convergence Analysis}\label{sec:proof-asymp}
To fix the notation, suppose $\bF:\Omega\rightarrow\reals$ is a random variable, $\bF(\omega)$ denotes a particular realization of $\bF$ corresponding to $\omega\in\Omega$ where $\Omega$ denotes the sample space.

\rev{First, recall that $U_i\in\reals^{{m}\times {m_i}}$ for $i\in\cM$ such that $\bI_m=[U_1,\hdots,{U_M}]$, where $\bI_m$ denotes the $m\times m$ identity matrix --see Definition~\ref{def:bregman}. Therefore, we can write $\bx^{k+1}$ equivalently as follows:
\begin{equation}
    \label{eq:Ek-operation-1}
    \bx^{k+1}=\bx^k+U_{i_k}U_{i_k}^\top (\tilde\bx^{k+1}-\bx^k).
\end{equation}
Also recall that for all $k\geq 1$, $\bE^k[\cdot]=\bE[\cdot\mid \cF_{k}]$, where $\cF_k=\sigma\left(\{i_0,\ldots,i_{k-1}\}\right)$ is the $\sigma$-algebra generated by i.i.d. random variables $\{i_0,\ldots,i_k\}$. Thus,
\begin{equation}
    \label{eq:Ek-operation-2}
    \bE^k[\bx^{k+1}]=\bx^k+\bE^k[U_{i_k}U_{i_k}^\top] (\tilde\bx^{k+1}-\bx^k)=\frac{1}{M}\tilde\bx^{k+1}+\Big(1-\frac{1}{M}\Big)\bx^k.
\end{equation}
Furthermore, for any $\psi:\cX\to\reals$ such that $\psi(\bx)=\sum_{i\in\cM}\psi_i(x_i)$ for some $\psi_i:\cX_i\to\reals$, we also have
\begin{equation}
    \label{eq:Ek-operation-3}
    \bE^k[\psi(\bx^{k+1})]=\frac{1}{M}\sum_{i\in\cM}\left(\psi(\bx^k)+\psi_i(\tilde x^{k+1})-\psi_i(x_i^k)\right)=\frac{1}{M}\psi(\tilde\bx^{k+1})+\Big(1-\frac{1}{M}\Big)\psi(\bx^k).
\end{equation}
Therefore, for $\cE^k(\cdot)$ defined in~\eqref{eq:Ek}, we can conclude that $\bE^k[\cE^k(\bx)]=0$ for any fixed $\bx$ and $k\geq 0$.}

\eh{\rev{Let $\bz^\#=(\bx^\#,y^\#)$ be a saddle point of $\cL$ in~\eqref{eq:original-problem}, and Bregman distance generating functions are selected according to \sa{Assumption}~\ref{rem:bregman}.
Lemma~\ref{lem:stepsize-sequence} implies that $\{[\tau_i^k]_{i\in\cM}, \sigma^k,\theta^k\}$ sequences generated by \texttt{RB-APD} and \texttt{RB-APD-B}, both satisfy \eqref{eq:step-size-tau-B} and \eqref{eq:step-size-theta-B} for $t^k=\frac{\sigma^k}{\sigma^0}$ for $k\geq 0$. Thus,}
we can conclude that 
for $k\geq 1$,
\begin{equation}\label{Qk-a}
    t^{k-1} R^k(\bz^\#)\geq t^kQ^{k}(\bz^\#).
\end{equation}
Finally, note that $C^k_*=C^k(\bx^{k+1},y^{k+1})$ for $\{\alpha^k,\beta^k\}$ sequence defined as $\alpha^k=c_\alpha/\sigma^{k-1}$ and $\beta^k=c_\beta/\sigma^{k-1}$ for all $k\geq 1$ for some $c_\alpha,c_\beta\geq 0$ as stated in Theorems~\ref{thm:backtrack} and~\ref{thm:main}.} 

{Now multiplying inequality \eqref{eq:one-step-main} by $t^k$ evaluated at $(\bx,y)=(\bx^\#,y^\#)$
taking conditional expectation, and using the facts that $\bE[\cE^k(\bx^\#)\mid \cF_k]=0$, $(1-\theta^k)_+L_{yy} 
\rev{=}0$}
\footnote{\rev{When $\underbar{$\mu$}=0$, for both \texttt{RB-APD} and \texttt{RB-APD-B}, $\theta^k \geq 0$ for $k\geq 0$; thus, $(1-\theta^k)_+=0$, which leads to $(1-\theta^k)_+L_{yy}=0$ for $k\geq 0$. On the other hand, for the case $\underbar{$\mu$}>0$, we assume that $L_{yy}=0$, i.e., $\Phi(\bx,\cdot)$ is affine for every fixed $\bx$; hence, we again get $(1-\theta^k)_+L_{yy}=0$, even though $\theta^k<1$ for some $k\geq 0$ is possible for this scenario.}}, 
\eh{and $\eh{C^k_*}\leq -\delta[{M}\bD_\cX^{\bT^k}(\bx^{k+1},\bx^k)+\frac{1}{\sigma^k}\bD_\cY(y^{k+1},y^k)]$ combined with \eqref{Qk-a} lead to
\begin{align}\label{eq:bound-iterates}
&t^k\bE^k[\cL({\bx}^{k+1},y^\#)-\cL(\bx^\#,{y}^{k+1})]+
\bE^k[t^kR^{k+1}(z^\#)+(M-1)t^kH^{k+1}(\bz^\#)]\nonumber\\
&\leq t^{k-1}R^k(\bz^\#)+(M-1)t^{k-1} H^k(\bz^\#)+(M-1)t^k(1-\theta^k)\big(\cL(\bx^k,y^\#)-\cL(\bx^\#,y^\#)\big) \nonumber\\
&\quad -t^k\delta\bE^k[{M}\bD_\cX^{\bT^k}(\bx^{k+1},\bx^k)+\tfrac{1}{\sigma^k}\bD_\cY(y^{k+1},y^k)].
\end{align}
Note when $\underbar{$\mu$}>0$ from the step-size rules we have that for any $k\geq 1$,
\begin{align}
    &\theta^k= \frac{\sigma^{k-1}}{\sigma^k}=\frac{\gamma^{k-1}\tilde\tau^{k-1}}{\gamma^k\tilde\tau^k}\geq \frac{\gamma^{k-1}}{\gamma^{k}}= \frac{1}{1+\underbar{$\mu$}\tilde\tau^{k-1}}\geq \frac{1}{1+\underbar{$\mu$}\tilde\tau^0}\geq \frac{M-1}{M},\label{eq:SC-theta-bound}\\
    \implies & (M-1)(1-\theta^k)t^k\leq t^{k-1}.
\end{align}
On the other hand, when $\underbar{$\mu$}=0$, then $\theta^k\geq 1$ which immediately implies that $(M-1)(1-\theta^k)t^k\leq t^{k-1}$. Therefore, we can rewrite \eqref{eq:bound-iterates} as follows by noting that $\cL(\bx^\#,y^\#)-\cL(\bx^\#,y^{k+1})\geq 0$,
\begin{align}\label{eq:a-b}
    &\bE^k[a^{k+1}]=\bE[a^{k+1}\mid\rev{\cF_{k}}]\leq a^k-b^k,
\end{align}
where $a^k,b^k\in\rev{\cF_{k}}$ are defined as follows
\begin{align}
    &a^k=t^{k-1}R^k(\bz^\#)+(M-1)t^{k-1} H^k(\bz^\#)+t^{k-1}\big(\cL(\bx^k,y^\#)-\cL(\bx^\#,y^\#)\big),\label{eq:ak}\\
    &b^k=t^k\delta\left( 
\rev{\bD_\cX^{\bT^k}(\tilde\bx^{k+1},\bx^k)}+\tfrac{1}{\sigma^k}\bD_\cY(y^{k+1},y^k)\right)\rev{\geq 0}.\label{eq:bk}
\end{align}}%


\eh{Moreover, from concavity of $\Phi(\bx,\cdot)$ and the fact that $\cL(\bx^k,y^\#)-\cL(\bx^\#,y^\#)\geq 0$ we can provide a lower bound on $a^k$ as follows:
\begin{align}
\label{eq:ak-bound}
a^k
&\geq t^{k-1}\Big[(M-1)\big(\cL(\bx^k,y^\#)-\cL(\bx^\#,y^\#)\big)+M\bD_\cX^{\bT^{k-1}}(\bx,\bx^{k})+\frac{M}{2}\norm{\bx^\#-\bx^{k}}_{\fM}^2 +\frac{1}{\sigma^{k-1}}\bD_{\cY}(y^\#,y^{k})   \nonumber \\
& \quad +\fprod{q^{k},y^{k}-y^\#}+\frac{M}{2\alpha^{k}}\norm{p_x^{k}}_{\cY^*}^2+\frac{M}{2\beta^{k}}\norm{p_y^{k}}^2_{\cY^*}\Big],\nonumber \\
&\geq \rev{\frac{M}{2}\norm{\bx^\#-\bx^{k}}^2_{t^{k-1}(\bT^{k-1}+\fM)}+t^{k-1}\Big(\tfrac{1}{\sigma^{k-1}}-M(\alpha^k+\beta^k)\Big)\bD_{\cY}(y^\#,y^{k})}, \nonumber\\
&\geq \rev{\frac{M}{2} t^{k}~\norm{\bx^\#-\bx^{k}}^2_{\bT^{k}+(1-\tfrac{1}{M})\fM}+t^k\Big(\tfrac{1}{\sigma^{k}}-\theta^kM(\alpha^k+\beta^k)\Big)\bD_{\cY}(y^\#,y^{k})},
\end{align}
where the second inequality follows from \eqref{eq:inner-q-B} and $\bD_{\cX_i}(x_i,x_i')\geq \frac{1}{2}\norm{x_i-x_i'}^2_{\cX_i}$ for $i\in\cM$, and the third one follows from \eqref{eq:step-size-tau-B} and \eqref{eq:step-size-theta-B}.}

\rev{Note that for all $i\in\cM$, we have
\begin{align}
\label{eq:inf-cond-1}
    Mt^k\Big(\frac{1}{\tau_i^k}+\big(1-\frac{1}{M}\big)\mu_i\Big)=\frac{t^k}{\tilde\tau^k}=\frac{\sigma^k}{\sigma^0}\frac{1}{\tilde\tau^k}=\frac{\gamma^k}{\gamma_0}\frac{1}{\tilde{\tau}^0}\geq\frac{1}{\bar{\tau}},\quad \forall~k\geq 0,
\end{align}
where we used $t^k=\frac{\sigma^k}{\sigma^0}$, $\sigma^k=\gamma^k\tilde\tau^k$, and $\gamma^k\geq \gamma^0$ for all $k\geq 0$, and $\tilde{\tau}^0\leq \bar{\tau}$. Moreover, for all $k\geq 0$, setting $\alpha^{k+1}=c_\alpha/\sigma^k$ and $\beta^{k+1}=c_\beta/\sigma^k$ for $c_\alpha,c_\beta\geq 0$ such that $1-M(c_\alpha+c_\beta)\geq \delta>0$, implies
\begin{align}
\label{eq:inf-cond-2}
    t^k\Big(\tfrac{1}{\sigma^{k}}-\theta^k M(\alpha^k+\beta^k)\Big)=\frac{1}{\sigma^0}-M\frac{\sigma^{k-1}}{\sigma^0} \frac{c_\alpha+c_\beta}{\sigma^{k-1}}\geq \frac{\delta}{\sigma^0}\geq \frac{\delta}{\gamma^0\bar{\tau}},\quad \forall~k\geq 0,
\end{align}
where we used $t^k=\frac{\sigma^k}{\sigma^0}$ and $\theta^k=\sigma^{k-1}/\sigma^k$ for all $k\geq 0$, and $\sigma^0=\gamma^0\tilde\tau^0\leq \gamma^0\bar\tau$. Finally, combining \eqref{eq:ak-bound} with the lower bounds given in \eqref{eq:inf-cond-1} and \eqref{eq:inf-cond-2}, and using $\bD_\cY(y,y')\geq\frac{1}{2}\norm{y-y'}_{\cY}^2$ for any $y\in\cY$ and $y'\in\dom h$, we get
\begin{align}
    \label{Qk-b}
    a^k\geq \frac{\delta'_x}{2}\norm{\bx^\#-\bx^k}_{\cX}^2+\frac{\delta'_y}{2}\norm{y^\#-y^k}_{\cY}^2\geq 0,\quad\forall~k\geq 0,
\end{align}
where $\delta'_x=\frac{1}{\bar\tau}>0$ and $\delta'_y=\frac{\delta}{\gamma^0\bar\tau}>0$. Therefore, invoking Lemma~\ref{lem:supermartingale}, we can conclude that $\lim_{k\to+\infty}a^k\geq 0$ and $\sum_{k=0}^\infty b^k\in\reals_{++}$ exist in a.s. sense, i.e.,
\begin{align}
    \label{eq:tkTk_sum}
    \sum_{k=0}^{+\infty}t^k\bD_\cX^{\bT^k}(\tilde\bx^{k+1},\bx^k)<\infty\quad a.s.
\end{align}}%

\rev{Since $\{a^k\}$ is an a.s. bounded sequence,} \eqref{Qk-b} implies that $\{\bz^k(\omega)\}$ is a bounded sequence for any $\omega\in\Omega$, \rev{where $\bz^k=(\bx^k,y^k)$}; hence, it has a convergent subsequence $\bz^{k_n}(\omega)\rightarrow \bz^*(\omega)$ as $n\rightarrow \infty$ for some $\bz^*(\omega)\in\cX\times\cY$ -- note that $k_n$ also depends on $\omega$ which is omitted to simplify the notation. Define $\bz^*=(\bx^*,y^*)$ such that $\bz^*=[\bz^*(\omega)]_{\omega\in\Omega}$.

\rev{Next, we argue that $\bz^{k_n\pm 1} \rightarrow \bz^*$ almost surely as $n\rightarrow \infty$. To this aim, first we analyze $\{t^k\bT^k\}$ sequence that appears in the definition of $b^k$ in~\eqref{eq:bk}. Note that \eqref{eq:inf-cond-1} implies that
\begin{align}\label{eq:tkTk-eq}
    \frac{t^k}{\tau_i^k}\eh{=} \frac{1}{M}\frac{\gamma^k}{\gamma^0\tilde\tau^0}-t^k\Big(1-\frac{1}{M}\Big)\mu_i\eh{=}\frac{\gamma^k}{\gamma^0\tilde{\tau}^0}\Big(\frac{1}{M}-\tilde{\tau}^k\Big(1-\frac{1}{M}\Big)\mu_i\Big),\quad \forall~i\in\cM,~\forall~k\geq 0,    
\end{align}
where we used $t^k=\frac{\sigma^k}{\sigma^0}=\frac{\gamma^k\tilde\tau^k}{\gamma^0\tilde\tau^0}$ for $k\geq 0$. Since $\bar{\tau}\geq\tilde\tau^k$ for $k\geq 0$, choosing $\reals_{++}\ni\bar\tau\eh{\leq}\frac{\bar\delta}{\bar\mu (M-1)}$ for some $\bar\delta\in(0,1)$ and $\bar\mu=\max_{i\in\cM}\mu_i$ implies that for all $k\geq 0$, we have
\begin{align}
    \label{eq:tkTk-bound}
    \frac{t^k}{\tau_i^k}\geq\frac{\gamma^k}{\gamma^0\tilde{\tau}^0}\Big(\frac{1}{M}-\bar{\tau}\Big(1-\frac{1}{M}\Big)\mu_i\Big)\geq\frac{\gamma^k}{\gamma^0}\frac{1-\bar\delta}{M \bar{\tau}} \geq\frac{1-\bar\delta}{M\bar{\tau}}>0,\quad\forall~i\in\cM,
\end{align}
which follows from $\gamma^k\geq\gamma^0$ for $k\geq 0$; hence, $t^k\bT^k\succeq \frac{1}{\bar\tau}\frac{1-\bar\delta}{M}\bI$ for $k\geq 0$. Finally, we also have $\frac{t^k}{\sigma^k}=\frac{1}{\sigma^0}\geq\frac{1}{\gamma^0\bar\tau}>0$ for $k\geq 0$. Now, now combining these two results with $\sum_{k=0}^\infty b^k <\infty$ (due to Lemma~\ref{lem:supermartingale}), we can conclude that $b_k\to 0$ implying 
\begin{align}
\label{eq:consecutive-iterates}
    0\leq \norm{\bx^{k+1}-\bx^k}_{\cX}^2\leq\norm{\tilde\bx^{k+1}-\bx^k}_{\cX}^2\to 0, \quad \norm{y^{k+1}-y^k}_{\cY}^2\to 0,
\end{align}
almost surely as $k\to\infty$. Therefore,} for any realization $\omega\in\Omega$ and $\zeta>0$, there exists $N_1(\omega)$ such that for any $n\geq N_1(\omega)$, we have $\max\{\norm{\bz^{k_n}(\omega)-\bz^{k_n-1}(\omega)},~\norm{\bz^{k_n}(\omega)-\bz^{k_n+1}(\omega)}\}< \frac{\zeta}{2}$. Convergence of $\{\bz^{k_n}(\omega)\}$ sequence also implies that there exists $N_2(\omega)$ such that for any $n\geq N_2(\omega)$, $\norm{\bz^{k_n}(\omega)-\bz^*(\omega)}< \frac{\zeta}{2}$. Thus, for $\omega\in\Omega$, letting $N(\omega)\triangleq \max\{N_1(\omega),N_2(\omega)\}$, we conclude that $\norm{\bz^{k_n\pm 1}(\omega)-\bz^*(\omega)}< \zeta$, i.e., $\bz^{k_n\pm 1} \rightarrow \bz^*$ almost surely as $n\rightarrow \infty$.

Fix an arbitrary $\omega\in\Omega$ and consider the subsequence $\{k_n\}_{n\geq 1}$. For all $n\in\mathbb{Z}_+$, $\bx$- and $y$-updates imply that 
\begin{subequations}
\label{eq:opt-cond}
\begin{align}
&\frac{1}{\tau^{k_n}_{i_{k_n}}} \Big(\grad\varphi_{\cX_{i_{k_n}}}\big(\bx^{k_n}(\omega)\big)-\grad\varphi_{\cX_{i_{k_n}}}\big(\bx^{k_n+1}(\omega)\big)\Big)
-\grad_{x_{\rev{i_{k_n}}}}\Phi(\bx^{k_n}(\omega),\rev{y^{k_n+1}}(\omega))\in \partial f_{i_{k_n}}\big(x_{i_{k_n}}^{k_n+1}(\omega)\big), \label{opt-cond-x}\\
&\frac{1}{\sigma^{k_n}}\Big(\grad\varphi_\cY\big(y^{k_n}(\omega)\big)-\grad\varphi_\cY\big(y^{k_n+1}(\omega)\big)\Big)+s^{k_n}(\omega)\in\partial h\big(y^{k_{n}+1}(\omega)\big), \label{opt-cond-y}
\end{align}
\end{subequations}
where \rev{we define $s^k= \grad_{y} \Phi(\bx^k,y^k)+\theta^k q^k$ and $q^k={M}(\grad_y \Phi(\bx^{k},{y^{k}})-\grad_y\Phi(\bx^{k-1},{y^{k-1}}))$ for $k\geq 0$.}

Note that 
the sequence of randomly chosen block coordinates in \texttt{RB-APD} or \texttt{RB-APD-B}, i.e., $\{i_{k_n}\}_{n\geq 1}$, is a Markov chain containing a single recurrent class. More specifically, the states are represented by $\cM$ and starting from state $i\in\cM$ the probability of eventually returning to state $i$ is strictly positive for all $i\in\cM$. Therefore, for any $i\in\cM$, we can select a further subsequence 
$\cK^i\subseteq \{k_n\}_{n\in\mathbb{Z}_+}$ such that $i_\ell=i$ for all $\ell\in \cK^i$. Note that $\cK^i$ is an infinite subsequence w.p. 1 and $\{\cK^i\}_{i\in\cM}$ is a partition of $\{k_n\}_{n\in\mathbb{Z}_+}$. For any $i\in\cM$, one can conclude from \eqref{opt-cond-y} and \eqref{opt-cond-x} that for all $\ell\in\cK^i$,
\begin{subequations}\label{eq:subsublimit}
\begin{align}
&u_i^k\triangleq\frac{1}{\tau^\ell_i} \Big(\grad\varphi_{\cX_i}\big(\bx^{\ell}(\omega)\big)-\grad\varphi_{\cX_i}\big(\bx^{\ell+1}(\omega)\big)\Big)
\rev{-\grad_{x_{i}}\Phi(\bx^{\ell}(\omega),y^{\ell+1}(\omega))}\in \partial f_{i}\big(x_i^{\ell+1}(\omega)\big),\label{eq:subsublimit-x}\\
&v^k\triangleq \frac{1}{\sigma^\ell}\Big(\grad\varphi_\cY\big(y^{\ell}(\omega)\big)-\grad\varphi_\cY\big(y^{\ell+1}(\omega)\big)\Big)+s^{\ell}(\omega)\in\partial h\big(y^{\ell+1}(\omega)\big).\label{eq:subsublimit-y}
\end{align}
\end{subequations}
Since 
$\cK^i\subseteq \{k_n\}_{n\in\mathbb{Z}_+}$, we have that $\lim_{\ell\in\cK^i}\bz^{\ell}(\omega)=\lim_{\ell\in\cK^i}\bz^{\ell+1}(\omega)=\bz^*(\omega)$.

{\bf (Part I) of Theorems \ref{thm:main} and \ref{thm:backtrack}.} \rev{Here we consider the case  $\underbar{$\mu$}=0$.} 
We first show that for any $\omega\in\Omega$, $\bz^*(\omega)$ is a saddle point of \eqref{eq:original-problem} by considering the optimality conditions for the updates of $\bx^{k+1}$ and $y^{k+1}$ of the \texttt{RB-APD} and \texttt{RB-APD-B} algorithms. \rev{Then, we argue that $\bz^*$ is indeed the unique limit point of $\{\bz^k\}$, i.e., $\bz^k\to \bz^*$ as $k\to\infty$.}

\rev{Next, we argue that $\sup_{k\geq 0}\frac{1}{\tau_i^k}<\infty$ for $i\in\cM$ and $\sup_{k\geq 0}\max\{\frac{1}{\sigma^k},\theta^k\}<\infty$. Once we have this result,} using the fact that for any $i\in\cM$, $\grad\varphi_{\cX_i}$ and $\grad\varphi_\cY$ are continuously differentiable on $\dom f_i$ and $\dom h$, respectively, it follows from Theorem 24.4 in~\cite{rockafellar2015convex} that by taking the limit of both sides of \eqref{eq:subsublimit} 
we get $\mathbf{0}\in \grad_{x_i}\Phi\big(\bx^*(\omega),y^*(\omega)\big)+\partial f_i(x_i^*(\omega))$, and $\mathbf{0}\in \partial h(y^*(\omega))-\grad_y\Phi(\bx^*(\omega),y^*(\omega))$, which implies that $\bz^*(\omega)$ is a saddle point of \eqref{eq:original-problem} for any $\omega\in\Omega$. \rev{Indeed, since $\underbar{$\mu$}=0$, for $k\geq 0$, $\gamma^k=\gamma^0$; hence, $\sigma^k=\gamma^0\tilde\tau^k$. Moreover, from Lemma~\ref{lem:parameter}, we have ${\eta\Psi}\leq  \tilde\tau^k\leq\bar\tau$ for $k\geq 0$, where $\Psi=\Psi_1$ if $L_{yy}=0$ and $\Psi=\min\{\Psi_1,\Psi_2\}$ if $L_{yy}>0$. Note that these bounds on $\tilde\tau^k$ hold for both \texttt{RB-APD} and \texttt{RB-APD-B}. Next, using the $\tau_i^k$ update rule, $\sigma^k=\gamma^0\tilde\tau^k$ and $\theta^k=\frac{\sigma^{k-1}}{\sigma^k}$, we get the following uniform bounds holding for all $k\geq 0$:
\begin{align}
    &0<\frac{1}{\tau_i^0}\leq \frac{1}{\tau_i^k}\leq \frac{1}{M}\frac{1}{\tilde\tau^k}-\frac{M-1}{M}\mu_i\leq \frac{1}{\eta \Psi M},\quad\forall~i\in\cM,\label{eq:tau-i-bound}\\ 
    &\frac{1}{\gamma^0\bar\tau}\leq \frac{1}{\sigma^k}=\frac{1}{\gamma^0\tilde\tau^k}\leq\frac{1}{\eta\Psi\gamma^0},\quad 0\leq \theta^k=\frac{\sigma^{k-1}}{\sigma^k}=\frac{\tilde\tau^{k-1}}{\tilde\tau^k}\leq \frac{\bar\tau}{\eta\Psi}.\label{eq:sigma-theta-bound}
\end{align}}%
\rev{Next, we show that $\bz^k\to \bz^*$ almost surely as $k\to\infty$, and for this result we will use the following bound on $\{t^k\}$:}
\begin{align}
    \eh{0\leq t^k=\frac{\sigma^k}{\sigma^0}=\frac{\gamma^0\tilde\tau^k}{\sigma^0}\leq \rev{\frac{\bar\tau}{\tilde\tau^0}\leq \frac{\bar\tau}{\eta\Psi}.}}\label{eq:tk-upper}
\end{align}


\eh{Since \eqref{eq:bound-iterates} is true for any saddle point $\bz^\#$, letting $\bz^\#=\bz^*$ and repeating the same 
\rev{arguments we used for showing \eqref{eq:a-b}, we can} conclude that 
\begin{equation}\label{eq:d-b}
    \bE^k[d^{k+1}]\leq d^k-b^k,
\end{equation}
where $d^k\triangleq t^{k-1}[(M-1)H^k(\bz^*)+R^k(\bz^*)+\cL(\bx^k,y^*)-\cL(\bx^*,y^*)]$ and $b^k$ is defined in \eqref{eq:bk}. Moreover, similar to \eqref{Qk-b}, we can show that 
\begin{equation}\label{eq:dk-lower}
    d^k\geq \frac{\delta'_x}{2}\norm{\bx^*-\bx^k}_{\cX}^2+\frac{\delta'_y}{2}\norm{y^*-y^k}_{\cY}^2\geq 0,\quad\forall~k\geq 0,
\end{equation}
where $\delta'_x,\delta'_y$ are defined after inequality \eqref{Qk-b}. Next, invoking Lemma~\ref{lem:supermartingale} again for \eqref{eq:d-b}, one can conclude that $d_*\triangleq \lim_{k\rightarrow \infty} d^k \geq 0$ exists almost surely.}
\rev{Now, we show that $d^{k+1}\to 0$ as $k\to \infty$. Let $\{\bz^{k_n}\}_{n\geq 0}$ be the subsequence we considered earlier such that $\bz^{k_n}\to \bz^*$ a.s. as $n\to \infty$. Using \eqref{eq:consecutive-iterates}, it is trivial to check that $H^{k_n+1}(\bz^*)\to 0$ and $\cL(\bx^{k_n+1},y^*)-\cL(\bx^*,y^*)\to 0$ as $k\to\infty$. Thus, $\lim_{n\to\infty}d^{k_n+1}=\lim_{n\to\infty}t^{k_n}R^{k_n+1}(\bz^*)$. Consider
\begin{align*}
    R^{k_n+1}(\bz^*)\triangleq &M\bD_\cX^{\bT^{k_n}}(\bx^*,\bx^{k_n+1})+\frac{M}{2}\norm{\bx^*-\bx^{k_n+1}}_{\fM}^2 +\frac{1}{\sigma^{k_n}}\bD_{\cY}(y^*,y^{k_n+1}) \\&+\fprod{\grad_y\Phi(\bx^{k_n+1},y^{k_n+1})-M\grad_y\Phi(\bx^{k_n},y^{k_n}),y^{k_n+1}-y^*}\\
    &+\frac{M}{2c_\alpha}\sigma^{k_n}\norm{
\grad_y\Phi(\bx^{k_n+1},y^{k_n+1})-\grad_y\Phi(\bx^{k_n},y^{k_n+1})}_{\cY^*}^2\\
&+\frac{M}{2c_\beta}\sigma^{k_n}\norm{
\grad_y\Phi(\bx^{k_n},y^{k_n+1})-\grad_y\Phi(\bx^{k_n},y^{k_n})}^2_{\cY^*},
\end{align*}
where we set $\alpha^{k+1}=c_\alpha/\sigma^k$ and $\beta^{k+1}=c_\beta/\sigma^k$ for some $c_\alpha,c_\beta\geq 0$ as described in~Lemma~\ref{lem:parameter}.}
\rev{Using \eqref{eq:tau-i-bound}, \eqref{eq:sigma-theta-bound} and \eqref{eq:tk-upper} together with the fact that $\bz^{k_n\pm 1}(\omega)\rightarrow \bz^*(\omega)$ for any $\omega\in\Omega$, we conclude that $0=t^{k_n}R^{k_n+1}(\bz^*(\omega))=\lim_{n\rightarrow \infty} d^{k_n+1}(\omega)$ for any $\omega\in\Omega$, where we also used the fact that $\{\bz^{k_n}\}$ is a bounded sequence which implies $\{\grad_y\Phi(\bx^{k_n+1},y^{k_n+1})-M\grad_y\Phi(\bx^{k_n},y^{k_n})\}_{n\geq 0}$ is bounded as well due to continuity of $\grad_y\Phi$.}
Henceforth, $\lim_{k\rightarrow \infty} d^k=0$ almost surely which together with \eqref{eq:dk-lower} implies that $\bz^k\rightarrow \bz^*$ almost surely.

{\bf (Part II) of Theorems \ref{thm:main} and \ref{thm:backtrack}.}
\eh{Suppose $\underbar{$\mu$}>0$. 
Recall that in this case, \rev{for all $i\in\cM$, we set $\varphi_{\cX_i}(x_i)=\frac{1}{2}\norm{x_i}_{\cX_i}^2$, where $\norm{x_i}_{\cX_i}=\sqrt{\fprod{x_i,x_i}}$.} 
Indeed, since $\sum_{k=0}^{\infty} b^k<\infty$,  $t^k\bD_\cX^{\bT^k}(\btx^{k+1},\bx^k)\to 0$ holds. 
\rev{Moreover, \eqref{eq:tkTk-bound} shows that} for any $i\in\cM$ and $k\geq 0$, $\frac{t^k}{\tau_i^k}\geq \frac{\gamma^k}{\gamma_0}\frac{1-\bar\delta}{M\bar\tau}$; therefore, we have $0\leq \gamma^k\norm{\bx^{k+1}-\bx^k}^2_{\cX}\leq \gamma^k\norm{\btx^{k+1}-\bx^k}^2_{\cX}\to 0$.
Moreover, 
we have that $\gamma^k=\sigma^k/\tilde\tau^k\geq (\Gamma/(\underbar{$\mu$}\tilde\tau^k))^2\geq (M\Gamma/(\underbar{$\mu$}\tau_i^k))^2$ for all $k\geq 0$, \rev{where the first inequality follows from Lemma \ref{lem:k2-rate-B} and the last one uses \eqref{eq:tau-i-bound}}. Therefore, one can easily conclude that $\norm{\bx^{k+1}-\bx^k}_{\cX}/\tau_i^k\to 0$ for any $i\in\cM$. \rev{Thus, invoking invoking \cite[Theorem 24.4]{rockafellar2015convex}, \eqref{eq:subsublimit-x} implies that $-\grad_\bx\Phi(\bx^*,y^*)\in\partial f(\bx^*)$ assuming $\grad\varphi_{\cX_i}$ is Lipschitz.
Finally, it follows from \eqref{eq:sigma-theta-bound} \rev{and \eqref{eq:subsublimit-y} that 
$\grad_y \Phi(\bx^*,y^*)\in\partial h(y^*)$, where we used \cite[Theorem 24.4]{rockafellar2015convex} assuming} $\grad\varphi_{\cY}$ is continuous. Therefore, we establish that any limit point of $\{\bz^k\}$ is a saddle point of \eqref{eq:original-problem}.}} 
\qed

\subsection{Convergence Rate Analysis}\label{sec:proof-rate}
\rev{Next we use the one-step result shown in 
Lemma~\ref{lem:one-step} to derive 
a useful bound for the ergodic sequence generated by either \texttt{RB-APD} or \texttt{RB-APD-B}, which will help us establish the desired convergence rate results.} 

\begin{lemma}\label{lem:general-bound}
Suppose Assumptions~\ref{assum} and~\ref{rem:bregman} hold. Given some $\gamma^0>0$ and $\bar{\tau}\in\left(0,\tfrac{1}{\bar{\mu}(M-1)}\right)$, let 
$\{\bx^k,y^k\}_{k\geq 0}$ be the iterate sequence generated by either \texttt{RB-APD-B}, stated in Algorithm~\ref{alg:RBPDB}, or by \texttt{RB-APD}, stated in Algorithm~\ref{alg:RBPD}. If \texttt{RB-APD} is used, we  
assume that \eqref{eq:initial_step_condition} holds for some $c_\alpha,c_\beta\geq 0$ and $\delta\in[0,1)$ as described in Theorem~\ref{thm:backtrack}.
Then for any 
\rev{$(\bx,y)\in\dom f\times\dom h$} and $K\geq 1$,
\begin{eqnarray}\label{eq:lagrange-general-bound}
\lefteqn{\rev{T^K\big(\cL(\bar{\bx}^K,y)-\cL(\bx,\bar{y}^K)\big)}}\\
&&\leq M\bD_\cX^{\bT^0}(\bx,\bx^0) +\frac{M-1}{2}\norm{\bx-\bx^0}_{\fM}^2+\big(\frac{1}{\sigma^0}+\rev{\theta^0}(M-1)L_{yy}\big)\bD_{\cY}(y,y^0) \nonumber\\
&&\mbox{} -t^{K}\Big(M\bD_\cX^{\bT^K}(\bx,\bx^K)+\frac{M-1}{2}\norm{\bx-\bx^K}_{\fM}^2 +\big(\frac{1}{\sigma^{K}}-M\theta^K(\alpha^K+\beta^K)\big)\bD_{\cY}(y,y^K)\Big)\nonumber\\
&&\mbox{}+(M-1)\Big(\rev{\cL(\bx^0,y)-\cL(\bar{\bx}^K,y)}+\sum_{k=0}^{K-1}t^k(1-\theta^k)_+L_{yy}\bD_{\cY}(y,y^k)\Big)+\rev{\sum_{k=0}^{K-1}t^k\cE^k(\bx)},\nonumber
\end{eqnarray}
where $T^K=\sum_{k=0}^{K-1}t^k$, $\bar{\bx}^K=\frac{1}{T^K+M-1}\Big(\sum_{k=0}^{K-2}(Mt^k-(M-1)t^{k+1})\bx^{k+1}+Mt^{K-1}\bx^K\Big)$ and $\bar{y}^K=\frac{1}{T^K}\sum_{k=0}^{K-1}t^ky^{k+1}$ \rev{for $\{t^k\}_{k\geq 0}\subset\reals_{++}$ such that $t^k=\sigma^k/\sigma^0$ for $k\geq 0$.}
\end{lemma}
\begin{proof}
\rev{By employing Lemma \ref{lem:one-step}, we aim to provide a 
convergence rate analysis for both convex-concave setting, i.e., $\underbar{$\mu$}=0$, and strongly convex-concave setting, i.e., $\underbar{$\mu$}=0$ and $L_{yy}=0$. Suppose the Bregman distance generating functions are set according to \sa{Assumption}~\ref{rem:bregman}. Lemma~\ref{lem:stepsize-sequence} implies that $\{[\tau_i^k]_{i\in\cM}, \sigma^k,\theta^k\}$ sequences generated by \texttt{RB-APD} and \texttt{RB-APD-B}, both satisfy \eqref{eq:step-size-tau-B} and \eqref{eq:step-size-theta-B} for $t^k=\frac{\sigma^k}{\sigma^0}$ for $k\geq 0$ --here, for \texttt{RB-APD}, assuming \eqref{eq:initial_step_condition} holds for some $c_\alpha,c_\beta\geq 0$ and $\delta\in[0,1)$ as described in Theorem~\ref{thm:backtrack}. Thus,} 
one can easily verify 
that \rev{for any $\bz\in\dom f\times \dom h$,} we have $t^{k+1}Q^{k+1}(\bz)-t^{k}R^{k+1}(\bz)\leq 0$ for all \rev{$k\geq 0$}. 

Now, multiplying both sides of \eqref{eq:one-step-main} by \rev{$t^k=\sigma^k/\sigma^0>0$}, summing over $k=0$ to $K-1$, we obtain
\begin{align}\label{eq:sum-convex-B}
&\sum_{k=0}^{K-1}t^k\left(\cL(\bx^{k+1},y)-\cL(\bx,y^{k+1})\right)\leq 
Q^0(\bz)-t^{K-1}R^K(\bz)+(M-1)\big(
\rev{\theta^0} H^0(\bz)-t^{K-1}H^K(\bz)\big)\nonumber\\
&+\sum_{k=0}^{K-1}t^k(M-1)\Big((1-\theta^k)(\cL(\bx^k,y)-\cL(\bx,y))+(1-\theta^k)_+L_{yy}\bD_{\cY}(y,y^k)\Big)+\sum_{k=0}^{K-1}t^k (\eh{C^k_*}+\cE^k(\bx)),
\end{align}
\rev{where we used $t^0=1$ and $t^{k+1}\theta^{k+1}=t^k$ for $k\geq 0$.
Due to initialization $\bx^{-1}=\bx^0$, and $y^{-1}=y^0$, the definitions of $Q^k(\cdot)$ and $H^k(\cdot)$ given in \eqref{eq:lagrangian-Q-B} and \eqref{eq:largrangian-H}, respectively, imply that
\begin{align*}
    Q^0(\bz)&=M\bD_\cX^{\bT^0}(\bx,\bx^0) +\frac{M-1}{2}\norm{\bx-\bx^0}_{\fM}^2+\frac{1}{\sigma^0}\bD_{\cY}(y,y^0)+\theta^0 (M-1)\fprod{\grad_y\Phi(\bx^0,y^0),~y-y^0)},\\
    H^0(\bz)&=f(\bx^0)-f(\bx)+\Phi(\bx^0,y^0)-\Phi(x,y).
\end{align*}}
Using 
the bound \eqref{eq:Lyy_bound}, we have that
\begin{align}\label{eq:H0-B}
H^0(\bz)+\fprod{\grad_y\Phi(\bx^0,y^0),y-y^0}&\leq f(\bx^0)-f(\bx)+\Phi(\bx^0,y)-\Phi(\bx,y)+\frac{L_{yy}}{2}\norm{y-y^0}_\cY^2\nonumber\\
&\leq \cL(\bx^0,y)-\cL(\bx,y)+L_{yy}\bD_\cY(y,y^0).
\end{align}
Moreover, using concavity of $\Phi(\bx,\cdot)$, for any $\bx\in\cX$, we can find a lower bound on $H^K(\bz)$,
\begin{align}\label{eq:lower-Hk-B}
H^K(\bz)&\geq f(\bx^K)-f(\bx)+\Phi(\bx^K,y)-\Phi(\bx,y)+\fprod{\grad_y\Phi(\bx^K,y^K),y^K-y} \nonumber \\
&=\cL(\bx^K,y)-\cL(\bx,y)+\fprod{\grad_y\Phi(\bx^K,y^K),y^K-y}.
\end{align}
\rev{Within $R^K(\bz)$, there is $\fprod{r^K,y^K-y}$ term, where $r^K=\grad_y\Phi(\bx^K,y^K)-M\grad_y\Phi(\bx^{K-1},y^{K-1})$, and in order to upper bound the right-hand-side of~\eqref{eq:sum-convex-B}, we first provide an intermediate inequality:
\begin{eqnarray}
    \label{eq:lower-Hk-R-B}
    \lefteqn{(M-1)H^K(\bz)+\fprod{r^K,y^K-y}}\\
    &&\geq (M-1)\Big(\cL(\bx^K,y)-\cL(\bx,y)\Big)+\fprod{q^K, y^K-y}\nonumber\\
    &&\geq (M-1)\Big(\cL(\bx^K,y)-\cL(\bx,y)\Big)-M(\alpha^K+\beta^K)\bD_{\cY}(y,y^K)\nonumber\\
    &&\mbox{}-\frac{M}{2\alpha^K}\norm{\grad_y\Phi(\bx^{K},y^{K})-\grad_y\Phi(\bx^{K-1},y^{K})}_{\cY^*}^2
    -\frac{M}{2\beta^K}\norm{\grad_y\Phi(\bx^{K-1},y^{K})-\grad_y\Phi(\bx^{K-1},y^{K-1})}^2_{\cY^*},\nonumber
\end{eqnarray}
where the first inequality follows from \eqref{eq:lower-Hk-B} and $q^K=M\big(\grad_y\Phi(\bx^{K},y^{K})-\grad_y\Phi(\bx^{K-1},y^{K-1})\big)$ and for the second inequality we used \eqref{eq:inner-q-B} to lower bound $\fprod{q^K,y-y^K}$.}
Therefore, \eqref{eq:sum-convex-B}, \eqref{eq:H0-B} and \eqref{eq:lower-Hk-R-B} together imply that
\begin{align}\label{eq:sum-convex-q-B}
&\sum_{k=0}^{K-1}t^k\big(\cL(\bx^{k+1},y)-\cL(\bx,y^{k+1})\big)+\sum_{k=\rev{1}}^{K-1}t^k(M-1)(\theta^k-1)\left(\cL(\bx^k,y)-\cL(\bx,y)\right)\nonumber\\
& +(M-1)t^{K-1}\big(\cL(\bx^K,y)-\cL(\bx,y)\big)\leq 
M\bD_\cX^{\bT^0}(\bx,\bx^0)+\frac{M-1}{2}\norm{\bx-\bx^0}_{\fM}^2 \nonumber\\
&\quad +\big(\frac{1}{\sigma^0}+\rev{\theta^0}(M-1)L_{yy}\big)\bD_{\cY}(y,y^0)+(M-1)
\big(\cL(\bx^0,y)-\cL(\bx,y)\big)
\nonumber\\
&\quad -t^{K}\Big(M\bD_\cX^{\bT^K}(\bx,\bx^K)+\frac{M-1}{2}\norm{\bx-\bx^K}_{\fM}^2 +\big(\frac{1}{\sigma^{K}}-M\theta^K(\alpha^K+\beta^K)\big)\bD_{\cY}(y,y^K)\Big)\nonumber\\
&\quad +\sum_{k=0}^{K-1}t^k(M-1)(1-\theta^k)_+L_{yy}\bD_{\cY}(y,y^k)+\sum_{k=0}^{K-1}t^k (\eh{C^k_*}+\cE^k(\bx)),
\end{align}
\rev{where we used the fact that \eqref{eq:step-size-tau-B} and \eqref{eq:step-size-theta-B} hold for both \texttt{RB-APD} and \texttt{RB-APD-B}.}  

Note that the left hand side of \eqref{eq:sum-convex-q-B} can be lower bounded using Jensen's inequality twice:
\begin{align}\label{eq:lagrange-lower-B}
&\sum_{k=0}^{K-1}t^k\big(\cL(\bx^{k+1},y)-\cL(\bx,y^{k+1})\big)+\sum_{k=\rev{1}}^{K-1}t^k(M-1)(\theta^k-1)\big(\cL(\bx^k,y)-\cL(\bx,y)\big) \nonumber\\
& +(M-1)t^{K-1}\big(\cL(\bx^K,y)-\cL(\bx,y)\big)\geq \nonumber\\
&\quad \rev{(T^K+M-1)}\big(\cL(\bar{\bx}^K,y)-\cL(\bx,y)\big)+T^K\big(\cL(\bx,y)-\cL(\bx,\bar{y}^K)\big),
\end{align}
\rev{which follows from convexity of $\cL(\cdot,y)$ and $-\cL(\bx,\cdot)$ for every fixed $(\bx,y)$. In the above application of Jensen's inequality on $-\cL(\bx,\cdot)$, the convex combination coefficients are $\{t^k\}_{k=0}^{K-1}$, which satisfy $t^k=\sigma^k/\sigma^0\geq 0$ and $T^K=\sum_{k=0}^{K-1}t^k$, while in the application of Jensen's inequality on $\cL(\cdot,y)$, the convex combination coefficients are $\{Mt^k-(M-1)t^{k+1}\}_{k=0}^{K-1}$ and $Mt^{K-1}\geq 0$ --note that they sum to $T^K+M-1$, and $Mt^k-(M-1)t^{k+1}\geq 0$ follows from $\theta^{k+1}\geq \frac{M-1}{M}$ for all $k\geq 0$; indeed, we have already argued that when $\underbar{$\mu$}=0$, $\theta^k\geq 1$ for $k\geq 0$, and when $\underbar{$\mu$}>0$, \eqref{eq:SC-theta-bound} shows that $\theta^{k+1}\geq \frac{M-1}{M}$ for all $k\geq 0$.}

\rev{Note that for \texttt{RB-APD} since the parameter choice satisfy \eh{\eqref{eq:step-size-condition-pd-new1} and \eqref{eq:step-size-condition-pd-new}} with some $\delta\in[0,1)$,  $\alpha^{k+1}=c_\alpha/\sigma^k$ and $\alpha^{k+1}=c_\beta/\sigma^k$ for $k\geq 0$; therefore, Lemma~\ref{lem:stronger-step-condition-B} implies that \eqref{eq:step-size-Ck-B} holds with the same $\{
\alpha^k,\beta^k\}$ and $\delta$. Moreover, Lemma~\ref{lem:parameter} implies that \eqref{eq:step-size-Ck-B} always holds for \texttt{RB-APD-B}. Thus, 
we have that $\eh{C^k_*}\leq 0$ for $k\geq 0$;} hence, combining \eqref{eq:lagrange-lower-B} with \eqref{eq:sum-convex-q-B} leads to the desired result.
\end{proof}

As we discussed before, we provide a uniform  analysis of the rate results for \texttt{RB-APD} and \texttt{RB-APD-B}. 
Now we are ready to provide the rate result for {\bf (Part I)} and {\bf (Part II)} of Theorem \ref{thm:backtrack} and \ref{thm:main}.

\eh{Indeed, Lemma \ref{assum:step-size-B} implies that the step-size sequence $\{[\tau_i^k]_{i\in\cM},\sigma^k,\theta^k\}$ selected in \texttt{RB-APD-B} algorithm is well-defined satisfying \eqref{eq:step-size-tau-B} and \eqref{eq:step-size-theta-B} for $\{t^k\}$ such that $t^k=\sigma^k/\sigma^0$ for $k\geq 0$, and $C^k_*\leq 0$ for $k\geq 0$.} Next, we show that $\{[\tau_i^k]_{i\in\cM},\sigma^k,\theta^k\}$ selected in \texttt{RB-APD} algorithm satisfies Assumption~\ref{assum:step-size}.
Indeed, since $\theta^k=\sigma^{k-1}/\sigma^k$, for $\alpha^k=c_\alpha/\sigma^{k-1}$ and $\beta^k=c_\beta/\sigma^{k-1}$,  \eh{\eqref{eq:step-size-condition-pd-new1} and \eqref{eq:step-size-condition-pd-new}} can be written as
\begin{align}
\label{eq:step_cond_induction}
\frac{1-\delta}{\tau^k_{i_k}}\geq L_{x_{i_k}x_{i_k}}+\frac{L_{yx_{i_k}}^2}{c_\alpha}\sigma^k,\quad \rev{1-\delta -M(c_\alpha+c_\beta)}\geq \frac{ML_{yy}^2}{c_\beta}(\sigma^k)^2.
\end{align}
Clearly, the initial step-sizes selected as in Remark \ref{rem:step} implies that \eqref{eq:step_cond_induction} holds for $k=0$. When $\mu=0$, i.e., (Part I), we have $\gamma^k=\gamma^0$ and $\theta^k=1$ for $k\geq 0$; hence, $\tilde\tau^k=\tilde\tau^0$ and $\sigma^k=\sigma^0$ for $k\geq 0$. Thus, \eh{\eqref{eq:step-size-condition-pd-new1} and \eqref{eq:step-size-condition-pd-new}} hold for all $k\geq 0$ for $\{[\tau^k_i]_{i\in\cM}, \sigma^k,\theta^k\}$ produced by \texttt{RB-APD}. For the case $\mu>0$, i.e., (Part II), we will use induction to show that \eh{\eqref{eq:step_cond_induction} holds}. Recall that for this case, we assume $L_{yy}=0$; hence, the second condition in \eqref{eq:step_cond_induction} holds for any $\sigma^k$ as long as $1\geq \delta+\rev{M(c_\alpha+c_\beta)}$. Now suppose the first condition in \eqref{eq:step_cond_induction} holds for some $k\geq 0$, using $\sigma^{k+1}=\sigma^k\sqrt{\gamma^{k+1}/\gamma^k}$ and $\gamma^{k+1}/\gamma^k\geq 1$, we get
\rev{\begin{align*}
\frac{1-\delta}{\tau_{i}^{k+1}}
&=\frac{1-\delta}{M}\Big(\frac{1}{\tilde\tau^{k+1}}-\mu_i\Big)-(1-\delta)\mu_i=\frac{1-\delta}{M\tilde\tau^{k}}\sqrt{\frac{\gamma^{k+1}}{\gamma^k}}-(1-\delta)\Big(1-\frac{1}{M}\Big)\mu_i,\\
&=\frac{1-\delta}{\tilde\tau_i^k}\sqrt{\frac{\gamma^{k+1}}{\gamma^k}}+(1-\delta)\left(\sqrt{\frac{\gamma^{k+1}}{\gamma^k}}-1\right)\Big(1-\frac{1}{M}\Big)\mu_i,\\
&\geq\Big(L_{x_ix_i}+\frac{L^2_{yx_i}}{c_\alpha}\sigma^k\Big)\sqrt{\frac{\gamma^{k+1}}{\gamma^k}}\geq L_{x_ix_i}+\frac{L^2_{yx_i}}{c_\alpha}\sigma^{k+1},\quad\forall~i\in\cM. 
\end{align*}}%
This completes the induction. Moreover, Lemma~\ref{lem:stepsize-sequence} implies that $\{[\tau_i^k]_{i\in\cM},\sigma^k,\theta^k\}$ generated by \texttt{RB-APD} satisfies \eqref{eq:step-size-tau-B} and \eqref{eq:step-size-theta-B} for $\{t^k\}$ such that $t^k=\sigma^k/\sigma^0$ for $k\geq 0$. Thus, Assumption~\ref{assum:step-size} holds for $\{\alpha^k,\beta^k,t^k\}_{k\geq 0}$ as in the algorithm.

{\bf Proof of Theorem \ref{thm:backtrack} and \ref{thm:main} (Part I).} 
\rev{Suppose the Bregman distance generating functions are set according to \sa{Assumption}~\ref{rem:bregman}, and for the results in (Part I), we assume that $Z\triangleq\dom f\times \dom h$ is a compact set.}
In order to show the convergence rate \rev{for the expected gap}, we will use the result in Lemma \ref{lem:general-bound} by taking the supremum over 
\rev{$Z$}, and \rev{then computing the expectation of an appropriate upper bound on the supremum with respect to randomness in coordinate selection.} 

\rev{Now, recall the bound \eqref{eq:lagrange-general-bound} established in Lemma \ref{lem:general-bound}. Since $\underbar{$\mu$}=0$, we know that $\{\theta^k\}_{k\geq 0}$ generated by either \texttt{RB-APD} or \texttt{RB-APD-B} both satisfy $\theta^k\geq 1$ for all $k\geq 0$; thus, we have that $(1-\theta^k)_+L_{yy}\bD_{\cY}(y,y^k)= 0$ for all $k\geq 0$. Furthermore, since $\alpha^{k+1}=c_\alpha/\sigma^k$ and $\beta^{k+1}=c_\beta/\sigma^k$ for all $k\geq 0$ for some $c_\alpha,c_\beta\geq 0$ such that $M(c_\alpha+c_\beta)<1$, we know that $\frac{1}{\sigma^{K}}-M\theta^K(\alpha^K+\beta^K)\geq 0$. Thus, after dropping the nonpositive terms on the right-hand side of \eqref{eq:lagrange-general-bound},} taking supremum \rev{of the resulting bound} 
over $\bz=(\bx,y)\in\dom f\times\dom h$, we get
\rev{\begin{align}
    &\eh{\cG}(\bar\bz^K)\leq \frac{1}{T^K}\Big(\sup_{\bx\in \dom f}B_1^K(\bx)+\sup_{y\in \dom h}B_2^K(y)\Big).
    \label{eq:gap-bound}\\
    &B_1^K(\bx)\triangleq M\bD_\cX^{\bT^0+(1-\frac{1}{M})\fM}(\bx,\bx^0)+\sum_{k=0}^{K-1}t^k\cE^k(\bx),\\
    &B_2^K(y)\triangleq (\tfrac{1}{\sigma^0}+\theta^0(M-1)L_{yy})\bD_\cY(y,y^0)
    +(M-1)\big(\cL(\bx^0,y)-\cL(\bar\bx^K,y)\big).
\end{align}}%
\rev{Due to compactness of $Z=\dom f\times \dom h$ and continuity of $\cL(\cdot,\cdot)$, there exists $\bar{B}_2<+\infty$ such that $(M-1)\sup_{y\in\dom h}\big\{\cL(\bx^0,y)-\cL(\bar\bx^K,y)\big\}\leq \bar B_2$ for all $K\geq 1$. One can easily construct a crude bound: $\bar B_2=(M-1)\sup\{\cL(\bx^0,y)-\cL(\bx,y):\ (\bx,y)\in\dom f\times\dom h\}<+\infty$. That said, in many practical situations a much tighter bound can be obtained, e.g., in case $\cL(\cdot, y)$ is Lipschitz with a uniform constant $L_x>0$ for all $y\in\dom h$, then $\bar B_2=(M-1)L_xD_x$, where $D_x=\sup_{\bx,\bx'\in\dom f}\norm{\bx-\bx'}_\cX$ denotes the diameter of $\dom f$. For instance, let $f$ be the indicator function of a compact convex set $X$, then for every $y\in\dom h$, $\cL(\cdot, y)$ is indeed Lipschitz with a uniform constant $L_x=\sup\{\norm{\grad_\bx\Phi(\bx,y)}_{\cX^*}:\ \bx\in X,\ y\in\dom h\}<\infty$ due to continuity of $\grad_\bx\Phi$ and compactness of $X\times\dom h$. Thus,
\begin{align}
\label{eq:y-bound}
\exists \bar{B}_2<+\infty~:\quad \sup_{y\in\dom h}B_2^K(y)\leq \bar B_2+\Big(\tfrac{1}{\sigma^0}+\theta^0(M-1)L_{yy}\Big)\sup_{y\in\dom h}\bD_\cY(y,y^0).
\end{align}}%

\rev{Next, we claim that $\bE[\sup_{\bx\in\dom f}B_1^K(\bx)]$ can be bounded as follows:
\begin{align*}
    \hbox{\bf Claim 1: } &\exists~\bar{B}_1<+\infty\ \hbox{such that}\\
    &\bE\Big[\sup_{\bx\in\dom f}B_1^K(\bx)\Big]\nonumber\\
    &\leq  \sup_{\bx\in\dom f}M\bD_\cX^{(1+\frac{1}{M})\bT^0+\fM}(\bx,\bx^0)+\frac{(M-1)L_{\varphi_{\cX}}^2}{2}\bE\Big[\sum_{k=0}^{K-1}t^k\norm{\tilde\bx^{k+1}-\bx^k}^2_{\bT^k+\fM}\Big]\\
    &\leq  \sup_{\bx\in\dom f}M\bD_\cX^{(1+\frac{1}{M})\bT^0+\fM}(\bx,\bx^0)+\frac{(M-1)L_{\varphi_{\cX}}^2}{2}\cdot\bar{B}_1<+\infty,\quad \forall~K\geq 1.
\end{align*}}%

\rev{It follows from \eqref{eq:Ek-operation-2} and \eqref{eq:Ek-operation-3} that}
\begin{subequations}\label{eq:expectation-func}
\begin{align}
&\rev{\bE^k}\big[Mf(\bx^{k+1})-f(\btx^{k+1})-(M-1)f(\bx^k)
\big]=0,\\
&\rev{\bE^k}\big[\fprod{\grad_\bx\Phi(\bx^k,y^{k+1}),\btx^{k+1}-M\bx^{k+1}+(M-1)\bx^k}
\big]=0, \\
&\rev{\bE^k}\big[M\bD_\cX^{\bT^k}(\bx^{k+1},\bx^k)-\bD_\cX^{\bT^k}(\btx^{k+1},\bx^k)
\big]=0.
\label{eq:expectation-x}
\end{align}
\end{subequations}
For $k\geq 0$, we define 
\begin{align}
    \Xi^k(\bx)\triangleq M\bD_\cX^{\bT^k+\rev{\fM}}(\bx,\bx^{k+1})-(M-1)\bD_\cX^{\bT^k+\rev{\fM}}(\bx,\bx^k)-\bD_\cX^{\bT^k+\rev{\fM}}(\bx,\btx^{k+1});
\end{align} 
\rev{hence, \eqref{eq:expectation-func} implies that 
\begin{align}
\label{eq:E-Xi}
    \bE\left[\sup_{\bx\in\dom f}B_1^K(\bx)\right]=\bE\left[\sup_{\bx\in\dom f}\Big\{M\bD_\cX^{\bT^0+(1-\frac{1}{M})\fM}(\bx,\bx^0)+\sum_{k=0}^{K-1}t^k\Xi^k(\bx)\Big\}\right].
\end{align}
Furthermore, for all $k\geq 0$, we also define}%
\begin{align}
    \tilde\Gamma_1^{k+1}\triangleq \rev{(\bT^k+\fM)}(\grad\varphi_\cX(\tilde\bx^{k+1})-\grad\varphi_\cX(\bx^k)),\quad
    \Gamma_1^{k+1}\triangleq\rev{(\bT^k+\fM)}(\grad\varphi_\cX(\bx^{k+1})-\grad\varphi_\cX(\bx^k));
\end{align}
 hence, from the definition of Bregman \rev{distances}, we get
 \begin{align}
     &\bD^{\rev{\bT^k+\fM}}_\cX(\bx,\bx^k)-\bD^{\rev{\bT^k+\fM}}_\cX(\bx,\tilde\bx^{k+1})-\bD^{\rev{\bT^k+\fM}}_\cX(\tilde\bx^{k+1},\bx^k)=\fprod{\tilde\Gamma_1^{k+1},~\bx-\tilde\bx^{k+1}},\quad\forall~\bx\in\cX,\\
     &\bD^{\rev{\bT^k+\fM}}_\cX(\bx,\bx^k)-\bD^{\rev{\bT^k+\fM}}_\cX(\bx,\bx^{k+1})-\bD^{\rev{\bT^k+\fM}}_\cX(\bx^{k+1},\bx^k)=\fprod{\Gamma_1^{k+1},~\bx-\bx^{k+1}},\quad\forall~\bx\in\cX.
 \end{align}
 \rev{Therefore,}
\begin{align}\label{eq:bregman-inner}
\Xi^k(\bx)&= \bD^{\rev{\bT^k+\fM}}_\cX(\tilde\bx^{k+1},\bx^k)+\fprod{\tilde\Gamma_1^{k+1},\bx-\tilde\bx^{k+1}} -M\Big(\bD^{\rev{\bT^k+\fM}}_\cX(\bx^{k+1},\bx^k)+\fprod{\Gamma_1^{k+1},\bx-\bx^{k+1}}\Big)\nonumber\\
&=M\bD_\cX^{\rev{\bT^k+\fM}}(\bx^k,\bx^{k+1})-\bD_\cX^{\rev{\bT^k+\fM}}(\bx^k,\tilde\bx^{k+1})+\fprod{\tilde\Gamma_1^{k+1}-M\Gamma_1^{k+1},\bx-\bx^k},
\end{align}
where we used $\fprod{\tilde\Gamma_1^{k+1},\bx^k-\tilde\bx^{k+1}}=-\bD_\cX^{\rev{\bT^k+\fM}}(\bx^k,\tilde\bx^{k+1})-\bD_\cX^{\rev{\bT^k+\fM}}(\tilde\bx^{k+1},\bx^k)$ and $\fprod{\Gamma_1^{k+1},\bx^k-\bx^{k+1}}=-\bD_\cX^{\rev{\bT^k+\fM}}(\bx^k,\bx^{k+1})-\bD_\cX^{\rev{\bT^k+\fM}}(\bx^{k+1},\bx^k)$.
\rev{Next, we define an auxiliary sequence $\{\bv^k\}_{k\geq 0}$ by initializing $\bv^0=\bx^0\in\dom f$, and invoking Lemma \ref{lem:inner-w} with} 
$\bdelta^k=\tilde\Gamma_1^{k+1}-M\Gamma_1^{k+1}$ and \rev{$\cA=\bT^k+\fM$} 
for all $k\geq 0$.
Therefore, 
\eqref{eq:bregman-inner} implies%
\begin{align}\label{eq:bregman-upper-bound}
\Xi^k(\bx) &\leq M\bD_\cX^{\rev{\bT^k+\fM}}(\bx^k,\bx^{k+1})-\bD_\cX^{\rev{\bT^k+\fM}}(\bx^k,\tilde\bx^{k+1}) + \bD_\cX^{\rev{\bT^k+\fM}}(\bx,\bv^k)-\bD_\cX^{\rev{\bT^k+\fM}}(\bx,\bv^{k+1})\nonumber\\
&\quad \rev{+\fprod{\tilde\Gamma_1^{k+1}-M\Gamma_1^{k+1},\bv^k-\bx^k}}+\frac{1}{2}\norm{\tilde\Gamma^{k+1}-M\Gamma^{k+1}}^2_{*,(\rev{\bT^k+\fM})^{-1}}.
\end{align}
\rev{Thus, \eqref{eq:bregman-upper-bound} immediately implies that
\begin{align}
    \sum_{k=0}^{K-1}t^k\Xi^k(\bx)\leq &\bD_\cX^{\rev{\bT^0+\fM}}(\bx,\bx^0)+\sum_{k=0}^{K-1} t^k\Big(M\bD_\cX^{\rev{\bT^k+\fM}}(\bx^k,\bx^{k+1})-\bD_\cX^{\rev{\bT^k+\fM}}(\bx^k,\tilde\bx^{k+1}) \Big)\nonumber\\
    &\mbox{}+\sum_{k=0}^{K-1}t^k\Big(\fprod{\tilde\Gamma_1^{k+1}-M\Gamma_1^{k+1},\bv^k-\bx^k}+\frac{1}{2}\norm{\tilde\Gamma^{k+1}-M\Gamma^{k+1}}^2_{*,(\rev{\bT^k+\fM})^{-1}}\Big),\label{eq:sum_Xi}
\end{align}
which follows from $t^k (\bT^k+\fM)\succeq t^{k+1} (\bT^{k+1}+\fM)$ for $k\geq 0$ whenever $\underbar{$\mu$}=0$\footnote{\rev{For $i\in\cM$ and $k\geq 0$, $\frac{1}{\tau_i^k}+\mu_i=\frac{1}{M}\Big(\frac{1}{\tilde\tau^k}+\mu_i\Big)$; hence, $t^k(\frac{1}{\tau_i^k}+\mu_i)\geq t^{k+1}(\frac{1}{\tau_i^{k+1}}+\mu_i)$ is equivalent to $\frac{1}{\tilde\tau^k}+\mu_i\geq \frac{t^{k+1}}{t^k}\Big(\frac{1}{\tilde\tau^{k+1}}+\mu_i\Big)=\frac{\gamma^{k+1}}{\gamma^k}\Big(\frac{1}{\tilde\tau^{k}}+\mu_i\frac{\tilde\tau^{k+1}}{\tilde\tau^k}\Big)$, which clearly holds because $\underbar{$\mu$}=0$ implies $\gamma^{k+1}=\gamma^k$ and we also have $\tilde\tau^{k+1}\leq \tilde\tau^{k}$, for all $k\geq 0$.}}, and in the above inequality, we also used $t^0\bD_\cX^{\rev{\bT^0+\fM}}(\bx,\bv^0)=\bD_\cX^{\rev{\bT^0+\fM}}(\bx,\bx^0)$, and $t^{K-1}\bD_\cX^{\rev{\bT^{K-1}+\fM}}(\bx,\bv^{K})\geq 0$ for $\bx\in\cX$.}

\rev{Next, similar to \eqref{eq:expectation-func}, one can easily verify that} 
\begin{align}
\label{eq:2moments}
    \rev{\bE^k\big[M\Gamma_1^{k+1}\big]=\tilde\Gamma_1^{k+1},\quad \bE^k\big[\|M\Gamma_1^{k+1}\|_{*,(\rev{\bT^k+\fM})^{-1}}^2\big]=M\|\tilde\Gamma_1^{k+1}\|_{*,(\rev{\bT^k+\fM})^{-1}}^2,}
\end{align}
which implies that
\begin{align}\label{eq:Gammak}
\bE^k[\|\tilde\Gamma_1^{k+1}-M\Gamma_1^{k+1}\|_{*,(\rev{\bT^k+\fM})^{-1}}^2]&=(M-1)\norm{\grad\varphi_\cX(\btx^{k+1})-\grad\varphi_\cX(\bx^k)}^2_{*,\rev{(\bT^k+\fM)}}\nonumber\\
&\leq (M-1)L_{\varphi_\cX}^2\norm{\btx^{k+1}-\bx^k}_{\rev{\bT^k+\fM}}^2,
\end{align}
where \rev{in the equality we used \eqref{eq:2moments} and the identity $\bE[\norm{X-\bE[X]}_2^2]=\bE[\norm{X}_2^2]-\norm{\bE[X]}_2^2$ holding for any random variable $X$; and for the inequality above,} we used the Lipschitz continuity of $\grad\varphi_\cX(\cdot)$. \rev{Furthermore, for $k\geq 0$, one also has
\begin{subequations}
    \label{eq:zero-expectation}
    \begin{align}
    &\bE^k[M\bD_\cX^{\rev{\bT^k+\fM}}(\bx^k,\bx^{k+1})-\bD_\cX^{\rev{\bT^k+\fM}}(\bx^k,\tilde\bx^{k+1})]=0,\\
    &\bE^k[\fprod{\tilde\Gamma_1^{k+1}-M\Gamma_1^{k+1},\bv^k-\bx^k}]=0,
\end{align}
\end{subequations}
where the first one follows from the same arguments we used for showing \eqref{eq:expectation-x}, and the second one follows from \eqref{eq:2moments}.}

\rev{Finally, \eqref{eq:E-Xi}, \eqref{eq:sum_Xi}, \eqref{eq:Gammak} and \eqref{eq:zero-expectation} together with the tower property of expectation imply that
\begin{align*}
    \bE\left[\sup_{\bx\in\dom f}
    B_1^K(\bx)\right]\leq\sup_{\bx\in\dom f}M\bD_\cX^{(1+\frac{1}{M})\bT^0+\fM}(\bx,\bx^0)+\frac{(M-1)L^2_{\varphi_\cX}}{2}\bE\left[\sum_{k=0}^{K-1}t^k\norm{\btx^{k+1}-\bx^k}_{\rev{\bT^k+\fM}}^2\right].
\end{align*}
Note that \eqref{eq:tkTk_sum}, \eqref{eq:tau-i-bound}  and \eqref{eq:tk-upper} imply that $$\sum_{k=0}^{K}t^k\norm{\btx^{k+1}-\bx^k}_{\rev{\bT^k+\fM}}^2\to \sum_{k=0}^{+\infty}t^k\norm{\btx^{k+1}-\bx^k}_{\rev{\bT^k+\fM}}^2<+\infty\quad a.s.\quad K\to +\infty.$$
Since we assume $\dom f$ is compact, Lebesgue's dominated convergence theorem implies that
\begin{align}
\bar{B}_1\triangleq \frac{(M-1)L^2_{\varphi_\cX}}{2}\bE\left[\sum_{k=0}^{+\infty}t^k\norm{\btx^{k+1}-\bx^k}_{\rev{\bT^k+\fM}}^2\right]<+\infty.
\end{align}
Thus, 
the uniform bound in \eqref{eq:E-Xi} implies that
\begin{align}
   \bE\left[\sup_{\bx\in\dom f}B_1^K(\bx)\right]\leq \bar{B}_1+\sup_{\bx\in\dom f}M\bD_\cX^{(1+\frac{1}{M})\bT^0+\fM}(\bx,\bx^0), \quad \forall~K\geq 1. 
\end{align}}%
\rev{This completes the proof of {\bf Claim 1}. Therefore, 
the result in \eqref{eq:rate-final} can be deduced from \eqref{eq:gap-bound}, \eqref{eq:y-bound}, and {\bf Claim 1}.} Furthermore, since $\sigma^k=\gamma^0\tilde\tau^k$ for $k\geq 0$, we conclude that  $T^K=\sum_{k=0}^{K-1}\sigma^k/\sigma^0 \geq {\frac{\eta\Psi}{\tilde\tau^0} K}$.\qed

{\bf Proof of Theorem \ref{thm:backtrack} and \ref{thm:main} (Part II).} \rev{Let $(\bx^*,y^*)$ be a sadde point of $\cL$.} In strongly convex-concave setting, i.e., $\underbar{$\mu$}>0$, we assume that $L_{yy}=0$ and $\varphi_\cX(\cdot)=\frac{1}{2}\norm{\cdot}^2$. \rev{When $L_{yy}=0$, defining $0^2/0=0$, one-step result in Lemma~\ref{lem:one-step} continues to holds with $\beta^{k}=0$ for all $k\geq 0$. Thus, consider setting $\alpha^{k}=c_\alpha/\sigma^{k-1}$ for $k\geq 0$ for some $c_\alpha>0$ and $\delta\in[0,1)$ such that $M c_\alpha+\delta\leq 1$. Therefore, evaluating the result of Lemma \ref{lem:general-bound} given in \eqref{eq:lagrange-general-bound} at $\bx=\bx^*$, and substituting $L_{yy}=0$ and $\alpha^{k}=c_\alpha/\sigma^{k-1}$ for $k\geq 0$, we get}
\begin{eqnarray}\label{eq:lagrange-sc}
\lefteqn{T^K\big(\cL(\bar{\bx}^K,y)-\cL(\bx^*,\bar{y}^K)\big)+(M-1)\big(\cL(\bar\bx^K,y)-\cL(\bx^*,y)\big)}\\
&&\leq \tfrac{M}{2}\norm{\bx^*-\bx^0}^2_{\bT^0+(1-\frac{1}{M})\fM}+\tfrac{1}{\sigma^0}\bD_{\cY}(y,y^0)+(M-1)\Big(\cL(\bx^0,y)-\cL({\bx}^*,y)\Big)\nonumber\\
&&\mbox{} +\rev{\sum_{k=0}^{K-1}t^k\cE^k(\bx^*)} -t^{K}\left(\tfrac{M}{2}\norm{\bx^*-\bx^K}^2_{\bT^K+(1-\frac{1}{M})\fM}+\tfrac{1}{\sigma^{K}}\big(1-Mc_\alpha\big)\bD_{\cY}(y,y^K)\right).\nonumber
\end{eqnarray}
Recall that \eqref{eq:expectation-func} implies $\bE^k[\cE(\bx^*)]=0$. Furthermore, since $(\bx^*,y^*)$ is a saddle point, we have $\cL(\bx^*,y^*)\geq \cL(\bx^*,y)$ and $\cL(\bx,y^*)\geq \cL(\bx^*,y^*)$ for any $\bx\in\cX$ and $y\in\cY$. Therefore, when we substitute $y=y^*$ in \eqref{eq:lagrange-sc}, the left-hand side is non-negative, and we get the following inequality:
\begin{eqnarray}\label{eq:iterate-rate-sc}
\lefteqn{t^{K}\bE\left[\tfrac{M}{2}\norm{\bx^*-\bx^K}^2_{\bT^K+(1-\frac{1}{M})\fM}+\tfrac{1}{\sigma^{K}}\big(1-Mc_\alpha\big)\bD_{\cY}(y^*,y^K)\right]}\\
&&\leq \tfrac{M}{2}\norm{\bx^*-\bx^0}^2_{\bT^0+(1-\frac{1}{M})\fM}+\tfrac{1}{\sigma^0}\bD_{\cY}(y^*,y^0)+(M-1)\Big(\cL(\bx^0,y^*)-\cL({\bx}^*,y^*)\Big).\nonumber
\end{eqnarray}
According to \texttt{RB-APD} and \texttt{RB-APD-B}, we have $M\bT^k+(M-1)\fM=\frac{1}{\tilde\tau^k}\bI_m$ for $k\geq 0$; hence, 
\begin{align*}
    t^k(M\bT^k+(M-1)\fM)=\frac{\sigma^k}{\sigma^0}\frac{1}{\tilde\tau^k}\bI_m=\frac{\gamma^k}{\sigma^0}\bI_m,\quad\forall~k\geq 0,
\end{align*}
which follows from $\sigma^k=\gamma^k\tilde\tau^k$ for $k\geq 0$. Thus, \eqref{eq:iterate-rate-sc} leads to the desired result in \eqref{eq:xy-bound-II}. 

Next, we will show \rev{the convergence rate in terms of the primal objective function value of the ergodic primal iterate sequence. Let $y^*(\bx)=\argmax_{y\in\cY}\cL(\bx,y)$ be the unique maximizer for any given $\bx\in\dom f$. After dropping non-positive terms from the right-hand side of \eqref{eq:lagrange-sc}, substituting $y=y^*(\bar\bx^K)$ and taking the expectation of both sides, we get
\begin{eqnarray}\label{eq:lagrange-sc-func}
\lefteqn{T^K\bE\left[\cL(\bar{\bx}^K,y^*(\bar{\bx}^K))-\cL(\bx^*,\bar{y}^K)\right]}\\
&&\leq \tfrac{\gamma^0}{2\sigma^0}\norm{\bx^*-\bx^0}^2_{\cX}+\tfrac{1}{\sigma^0}\sup_{y\in\dom h}\bD_{\cY}(y,y^0)+(M-1)\bE\Big[\cL(\bx^0,y^*(\bar{\bx}^K))-\cL(\bar{\bx}^k,y^*(\bar{\bx}^K))\Big],\nonumber
\end{eqnarray}
where we used $(M\bT^0+(M-1)\fM)=\frac{\gamma^0}{\sigma^0}\bI_m$. Note that $\cL(\bx^0,y^*(\bar{\bx}^K))-\cL(\bar{\bx}^k,y^*(\bar{\bx}^K))\leq F(\bx^0)-F(\bar\bx^k)\leq F(\bx^0)-F(\bx^*)$ and $F(\bar\bx^K)-F(\bx^*)\leq\cL(\bar{\bx}^K,y^*(\bar{\bx}^K))-\cL(\bx^*,\bar{y}^K)$ w.p. 1. Therefore, the result in \eqref{eq:F-bound} can be concluded immediately.}  

\rev{Finally, $\cO(1/K^2)$ rate for both \eqref{eq:xy-bound-II} and \eqref{eq:F-bound} follows from Lemma \ref{lem:k2-rate-B}, which} implies that $\gamma^K=\sigma^K/\tilde\tau^K=\Omega(K^2)$ and $T^K=\sum_{k=1}^K\sigma^k/\sigma^0=\Omega(K^2)$.\qed
\end{document}